\documentclass[12pt]{amsart}
\usepackage[foot]{amsaddr}

\usepackage{amssymb,amscd,amsthm,verbatim,amsmath,color,fancyhdr,mathrsfs}
\usepackage{graphicx}
\usepackage{turnstile}
\usepackage[colorlinks = true, urlcolor={blue}, citecolor={blue}, backref=page]{hyperref}

\usepackage[letterpaper, left=2.5cm, right=2.5cm, top=2.5cm,bottom=2.5cm,dvips]{geometry}

\usepackage{styfile} % contains additional packages outside amsart template
\newtheorem{theorem}{Theorem}[section]
\newtheorem{proposition}[theorem]{Proposition}
\newtheorem{lemma}[theorem]{Lemma}
\newtheorem{remark}[theorem]{Remark}

\newtheorem{definition}[theorem]{Definition}

\newtheorem{assumption}{Assumption}

\DeclarePairedDelimiterX{\bracket}[3]{#1}{#2}{#3}
\newcommand{\round}[1]{\bracket*{(}{)}{#1}}

\newcommand{\squarebrack}[1]{\bracket*{\lbrack}{\rbrack}{#1}}

\providecommand*{\eu}{\ensuremath{\mathrm{e}}} % Euler's number
\providecommand{\newoperator}[3]{\newcommand*{#1}{\mathop{#2}#3}}
\providecommand{\renewoperator}[3]{\renewcommand*{#1}{\mathop{#2}#3}}

\renewoperator{\Re}{\mathrm{Re}}{\nolimits}
\renewoperator{\Im}{\mathrm{Im}}{\nolimits}

\providecommand*{\slot}[1]{\ifblank{#1}{\,\cdot\,}{#1}}

\DeclarePairedDelimiterXPP{\nrm}[2]{}{\lVert}{\rVert}{\ensuremath{_{#1}}}{\ifblank{#2}{\,\cdot\,}{#2}}
\newcommand{\norm}[2]{\nrm*{#1}{#2}}

\newcommand{\enorm}[1]{\norm{2}{#1}} % Euclidean norm
\newcommand{\esqnorm}[1]{\enorm{#1}^2} % square of Euclidean norm

\newcommand{\opnorm}[1]{\norm{\mathrm{op}}{#1}}
\newcommand{\normF}[1]{\norm{\mathrm{F}}{#1}}

\newcommand{\abs}[1]{\bracket*{\lvert}{\rvert}{#1}}

\newcommand{\inner}[1]{\bracket*{\langle}{\rangle}{#1}}

\DeclarePairedDelimiterXPP\prob[1]{\mathbb{P}}{\lbrace}{\rbrace}{}{\renewcommand\given{\nonscript\:\delimsize\vert\nonscript\:\mathopen{}}#1} % base probability
\newcommand{\Prob}[1]{\prob*{#1}} % probability operator

\DeclarePairedDelimiterXPP\probability[2]{\mathbb{P}_{#1}}{\lbrace}{\rbrace}{}{\renewcommand\given{\nonscript\:\delimsize\vert\nonscript\:\mathopen{}}#2} % base probability operator with subscript
\DeclarePairedDelimiterXPP\expectation[1]{\mathbb{E}}{\lbrack}{\rbrack}{}{\renewcommand\given{\nonscript\:\delimsize\vert\nonscript\:\mathopen{}}#1} % base expectation
\newcommand{\E}[1]{\expectation*{#1}} % expectation operator

\DeclarePairedDelimiterXPP\expectationdist[2]{\mathbb{E}_{#1}}{\lbrack}{\rbrack}{}{\renewcommand\given{\nonscript\:\delimsize\vert\nonscript\:\mathopen{}}#2} % base expectation with subscript
\newcommand{\Exp}[2]{\expectationdist*{#1}{#2}} % expectation with subscript

\DeclarePairedDelimiterXPP\variance[1]{\mathrm{Var}}{\lbrack}{\rbrack}{}{\renewcommand\given{\nonscript\:\delimsize\vert\nonscript\:\mathopen{}}#1} % base variance
\DeclarePairedDelimiterXPP\variancedist[2]{\mathrm{Var}_{#1}}{\lbrack}{\rbrack}{}{\renewcommand\given{\nonscript\:\delimsize\vert\nonscript\:\mathopen{}}#2} % % base variance with subscript
\DeclarePairedDelimiterXPP\covariance[2]{\mathrm{Cov}}{(}{)}{}{#1,\mathopen{}#2} % base covariance
\DeclarePairedDelimiterXPP\covariancedist[3]{\mathrm{Cov}_{#1}}{(}{)}{}{#2,\mathopen{}#3} % base covariance
\newcommand{\inv}[1]{\frac{1}{#1}}
\newcommand{\indicator}[2]{\mathbbm{1}\ensuremath{_{#1}}\ifblank{#2}{}{\set{#2}}} % \ifblank{#2}{}{\set{#2}}
\newoperator{\supp}{\mathrm{supp}}{\nolimits}

\makeatletter
\providecommand*{\diff}
{\@ifnextchar^{\DIfF}{\DIfF^{}}}
\def\DIfF^#1{
	\mathop{\mathrm{\mathstrut d}}
	\nolimits^{#1}\gobblespace}
\def\gobblespace{
	\futurelet\diffarg\opspace}
\def\opspace{
	\let\DiffSpace\!
	\ifx\diffarg(
	\let\DiffSpace\relax
	\else
	\ifx\diffarg[
	\let\DiffSpace\relax
	\else
	\ifx\diffarg\{
	\let\DiffSpace\relax
	\fi\fi\fi\DiffSpace}

\providecommand*{\pdiff}
{\@ifnextchar^{\pDIfF}{\pDIfF^{}}}
\def\pDIfF^#1{
	\mathop{\mathrm{\mathstrut \partial}}
	\nolimits^{#1}\gobblespace}
\def\gobblespace{
	\futurelet\diffarg\opspace}
\def\opspace{
	\let\DiffSpace\!
	\ifx\diffarg(
	\let\DiffSpace\relax
	\else
	\ifx\diffarg[
	\let\DiffSpace\relax
	\else
	\ifx\diffarg\{
	\let\DiffSpace\relax
	\fi\fi\fi\DiffSpace}

\providecommand*{\deriv}[3][]{\frac{\diff^{#1}{#2}}{\diff {#3}^{#1}}}

\makeatother

\providecommand\given{}
\newcommand\SetSymbol[1][]{
	\nonscript\:#1\vert
	\allowbreak
	\nonscript\:
	\mathopen{}}
\DeclarePairedDelimiterX\Set[1]\{\}{
	\renewcommand\given{\SetSymbol[\delimsize]}
	#1
}

\newcommand{\set}[1]{\Set*{#1}}
\newcommand{\interval}[1]{\squarebrack{#1}}

\newcommand{\setinline}[1]{\Set{#1}}

\newcommand{\R}{\mathbb{R}}
\newcommand{\Rd}[1]{\mathbb{R}^{#1}}
\newcommand{\Natural}{\mathbb{N}}

\newcommand{\Integer}{\mathbb{Z}}

\newcommand{\id}{\mathrm{id}}

\newcommand{\N}{\mathbb{N}}

\newcommand{\Acal}{\mathcal{A}}

\newcommand{\Dcal}{\mathcal{D}}
\newcommand{\Ecal}{\mathcal{E}}
\newcommand{\Fcal}{\mathcal{F}}

\newcommand{\Hcal}{\mathcal{H}}

\newcommand{\Mcal}{\mathcal{M}}

\newcommand{\Scal}{\mathcal{S}}
\newcommand{\Tcal}{\mathcal{T}}

\newcommand{\Vcal}{\mathcal{V}}
\newcommand{\Wcal}{\mathcal{W}}

\newcommand{\tauvec}{{\boldsymbol{\tau}}}

\newcommand{\eps}{\epsilon}

\newcommand{\rr}{\mathbb{R}}
\newcommand{\eq}{\begin{equation}}
\newcommand{\en}{\end{equation}}
\newcommand{\SP}[1]{\textcolor{red}{SP:#1}}

\providecommand*{\Graphons}{{\widehat\Wcal}}
\providecommand*{\cut}{\square}
\newcommand{\cutnorm}[1]{\norm{\cut}{#1}}

\newcommand{\Ent}{\mathcal{E}}

\newcommand{\RaghavS}[1]{{\color{orange}[Somani: #1]}}
\newcommand{\Tripathi}[1]{{\color{cyan}[Tripathi: #1]}}

\title{Stochastic optimization on matrices and a graphon McKean-Vlasov limit}

\author{Zaid Harchaoui}
\address{Zaid Harchaoui\\ Department of Statistics \\ University of Washington\\ Seattle WA 98195, USA\\ {Email: zaid@uw.edu}}
\author{Sewoong Oh}
\address{Sewoong Oh\\ Paul G. Allen School of Computer Science \& Engineering \\ University of Washington\\ Seattle WA 98195, USA\\ {Email: sewoong@cs.washington.edu}}
\author{Soumik Pal}
\address{Soumik Pal\\ Department of Mathematics \\ University of Washington\\ Seattle WA 98195, USA\\ {Email: soumik@uw.edu}}
\author{Raghav Somani}
\address{Raghav Somani\\ Paul G. Allen School of Computer Science \& Engineering \\ University of Washington\\ Seattle WA 98195, USA\\ {Email: raghavs@cs.washington.edu}}
\author{Raghavendra Tripathi}
\address{Raghavendra Tripathi\\ Department of Mathematics \\ University of Washington\\ Seattle WA 98195, USA\\ {Email: raghavt@uw.edu}}

\keywords{gradient flows on graphons, stochastic approximations, stochastic gradient descent on matrices, exchangeable arrays of diffusions, reflected Brownian motions}
	
\subjclass[2020]{60K35,62L20,60G09,68W20}

\thanks{This research is partially supported by the following grants. All the authors are supported by NSF grant DMS-2134012. Additionally, Oh is supported by NSF grant CCF-2019844 and CNS-2002664; Pal is supported by NSF grant DMS-2052239 and DMS-2133244; Harchaoui is supported by NSF grant CCF-2019844, DMS-2023166, DMS-2133244. Thanks to the PIMS Kantorovich Initiative for facilitating this collaboration. The authors are listed in alphabetical order.}

\date{\today}

\begin{document}

\begin{abstract}
    We consider stochastic gradient descents on the space of large symmetric matrices of suitable functions that are invariant under permuting the rows and columns using the same permutation. We establish deterministic limits of these random curves as the dimensions of the matrices go to infinity while the entries remain bounded. Under a ``small noise'' assumption the limit is shown to be the gradient flow of functions on graphons whose existence was established in [Oh, Somani, Pal, and Tripathi, \emph{J Theor Probab 37, 1469–1522 (2024)}]. We also consider limits of stochastic gradient descents with added properly scaled reflected Brownian noise. The limiting curve of graphons is characterized by a family of stochastic differential equations with reflections and can be thought of as an extension of the classical McKean-Vlasov limit for interacting diffusions to the graphon setting. The proofs introduce a family of infinite-dimensional exchangeable arrays of reflected diffusions and a novel notion of propagation of chaos for large matrices of diffusions converging to such arrays in a suitable sense.
\end{abstract}

\maketitle

%%%%%%%%%%%%%%%%%%%%%%%%%%%%%%
% indenting Table of Content for readability
% \makeatletter
% \def\@tocline#1#2#3#4#5#6#7{\relax
%   \ifnum #1>\c@tocdepth % then omit
%   \else
%     \par \addpenalty\@secpenalty\addvspace{#2}%
%     \begingroup \hyphenpenalty\@M
%     \@ifempty{#4}{%
%       \@tempdima\csname r@tocindent\number#1\endcsname\relax
%     }{%
%       \@tempdima#4\relax
%     }%
%     \parindent\z@ \leftskip#3\relax \advance\leftskip\@tempdima\relax
%     \rightskip\@pnumwidth plus4em \parfillskip-\@pnumwidth
%     #5\leavevmode\hskip-\@tempdima
%       \ifcase #1
%       \or\or \hskip 1em \or \hskip 2em \else \hskip 3em \fi%
%       #6\nobreak\relax
%     \hfill\hbox to\@pnumwidth{\@tocpagenum{#7}}\par% <---- \dotfill -> \hfill
%     \nobreak
%     \endgroup
%   \fi}
% \makeatother
%%%%%%%%%%%%%%%%%%%%%%%%%%%%%%

% \tableofcontents

% \vskip.5in
\section{Introduction}\label{sec:intro}

The study of particle systems under mean-field interaction is a classical topic in probability theory~\cite{gartner88}. It involves multidimensional diffusions that interact through their empirical distributions of the type
\begin{equation}\label{eq:mckeanvlasov}
    \diff X_i(t) = b\left(X_i(t), \hat{\mu}^{(N)}(t) \right) \diff t + \diff B_i(t) , \quad i \in \squarebrack{N} , \quad t\in\R_+,
\end{equation}
where $N\in\Natural$, $X_i(t) \in \Rd{d}$ for all $i\in\squarebrack{N}$ and for some $d\in\Natural$, and
$\hat{\mu}^{(N)}(t) \coloneqq \frac{1}{N} \sum_{i=1}^N \delta_{X_i(t)}$, is the empirical distribution of the vector $\left(X_i(t)\right)_{i \in [N]}$ at time $t\in \R_+$, and $\left(B_i \right)_{i \in [N]}$ is a vector of i.i.d. standard $d$-dimensional Brownian motions. Prominent examples of such particle systems include the diffusion given by the SDE
\begin{equation}\label{eq:granmediapart}
    \diff X_i(t)= - \nabla V\left(X_i(t) \right)\diff t - \frac{1}{N} \sum_{j=1}^N \nabla W\left( X_i(t) - X_j(t) \right) \diff t + \diff B_i(t) ,\quad t\in\R_+,
\end{equation}
for $i\in\squarebrack{N}$,
where $V$ and $W$ are differentiable convex functions on $\Rd{d}$. However, any drift that is symmetric in the coordinates (``mean-field interactions'') can be represented as~\eqref{eq:mckeanvlasov} for some suitable function $b$. Often, the SDE~\eqref{eq:mckeanvlasov} includes a reflection term to constrain the coordinate process to a subset of the Euclidean space~\cite{sznitmanreflect}. The study of such systems originated from the probabilistic study of the Boltzmann and Vlasov equations due to Kac~\cite{kac1956foundations}, McKean~\cite{mckean75}, Dobrushin~\cite{dobrushin79}, Tanaka~\cite{tanaka78} and many others. For modern surveys, see Sznitman~\cite{SznitmanSF}, Villani~\cite{villani12notes}, Chaintron and Diez~\cite{ChaintronDiez} and Jabin~\cite{Jabin14}. 

Under suitable assumptions, as the number of particles go to infinity, it is known that the process of empirical distributions of the particle system converges to the solutions of families of well-known PDEs. For example, for the system~\eqref{eq:granmediapart}, the random process $\hat{\mu}^{(N)}$  converges weakly to the solution of granular media equation~\cite{CGM08}, as $N\rightarrow \infty$. The convergence is often obtained via \textit{propagation of chaos} where, in the large particle limit, a finite collection of randomly chosen particles evolves independently and identically. Furthermore, a randomly chosen particle in the large particle limit is distributed according to the McKean-Vlasov SDE~\cite{gartner88}: $\diff X(t) = b\left( X(t), \mu(t) \right)\diff t + \diff B(t)$, $t\in\R_+$, where $\mu(t)$ is the law of $X(t)$.

In this work we study an analogous evolution of symmetric matrices where the coordinates interact via a suitably symmetric function. As an example, consider the function $R_n$ defined on $\Mcal_n^{0}$, the set of all $n\times n$ symmetric matrices with entries in $[0, 1]$, given by
\begin{equation} 
\label{eqn:examplern}
    R_n(A)\coloneqq\frac{1}{n}\mathbb{E}\norm{2}{Y-n^{-1}AX}^2.
\end{equation}
where $(X, Y)\in \mathbb{R}^n\times \mathbb{R}^n$ is a random vector. Minimizing $R_n$ is the classical least squares regression problem. However, notice that even in this simple setup, this problem is non-trivial because of the restriction that entries of $A$ are in $[0, 1]$. If we assume that $(X, Y)$ is exchangeable, that is, $(X, Y)\stackrel{d}{=}(X^{\sigma}, Y^{\sigma})$ for any permutation $\sigma$ of $[n]\coloneqq \setinline{1,2,\ldots,n}$, then
the function $R_n$ satisfies a \emph{permutation invariance} property. That is, its value does not change if we permute the rows and columns of the matrix $A$ by the same permutation over $\squarebrack{n}$. Another rich source of such permutation invariant functions comes from the functions on unlabelled weighted graphs, for example, homomorphism density functions. Optimization of homomorphism density functions is a challenging problem that is being actively investigated~\cite{bianchi2021limit,neeman2023typical}. Projected stochastic gradient methods are empirically studied for optimizing such problems~\cite{ChernThesis}. We refer the reader to Section~\ref{sec:example} for more details on such examples. Consider the following diffusion on symmetric $n\times n$ matrices 
\begin{equation}\label{eq:exampleR}
   \diff X_{n}(t) = -n^2\nabla R_n\round{X_{n}(t)}\diff t + \beta\diff B_{n}(t) + \diff L_{n}(t) ,\qquad t\in\R_+,
\end{equation}
where $B_{n}$ is a system of $n\times n$ symmetric matrix-valued process of coordinatewise independent Brownian motions and $L_{n}$ is the coordinatewise bounded variation local time process that constrains each coordinate process to stay in the interval $\interval{0,1}$ (see Section~\ref{sec:RBM} for details). One may ask what is an appropriate notion of limit of such a process as $n\rightarrow \infty$? Does~\eqref{eq:exampleR} exhibit propagation of chaos? Note that the function $R_n$ in~\eqref{eq:exampleR} is not covered by the classical McKean-Vlasov theory since $R_n(A)$ is not symmetric in the $n^2$ (up to symmetry) many entries of a matrix $A$. Therefore, $R_n$ cannot be expressed as a function of the empirical distribution of the entries of the argument matrix. The same is true for any arbitrary differentiable function over $n \times n$ symmetric matrices that is invariant under permuting the rows and the columns using the same permutation. Spectral functions, for example, satisfy such an invariance, as do functions on edge-weighted graphs (represented by their adjacency matrices) that are invariant under vertex relabeling. %Consider SGD over such a function (see below for more details). 
This particular class of symmetry is captured, not by empirical measures but by \emph{graphons}. In other words, such functions can be thought of as functions on the space of graphons instead of measures. 

Analogous to the classical McKean-Vlasov theory, we show in this paper that, under suitable assumptions,~\eqref{eq:exampleR} exhibits a propagation of chaos. Furthermore, in $n\to \infty$ limit, the coordinates of $X_n(t)$ become conditionally independent, and the evolution of a randomly chosen coordinate can be described by a novel graphon-valued McKean-Vlasov equation. The existence and uniqueness of such a process are established in Proposition~\ref{prop:exists}. Proposition~\ref{prop:mck_vlasov} shows that the process $X_n(t)$ converges to a deterministic curve on the space of graphons, $\Graphons$ (see Section~\ref{sec:background}). We also refer the reader to see Section~\ref{sec: Linear regression} for details of our example.
%Note that the particular example of $R_n$ in~\eqref{eq:exampleR} is a spectral function and the question of hydrodynamic limit can be considered on the ``particle system'' of the $n$ real eigenvalues of the matrix argument. This is an interesting connection to dynamical Random Matrix Theory \cite[Chapter 12]{RMTdyn}.

Recently, various authors~\cite{delattre2016note,bhamidi2019weakly,coppini2022note,dupuis2022large} have investigated McKean-Vlasov limits for interacting particle systems on dense graphs. This is akin to equation~\eqref{eq:granmediapart} where particles interact only if they are neighbors in some underlying graph.  In these works, the McKean-Vlasov system describes the evolution of random particles from an infinite ensemble where the underlying interaction is determined by a graph or graphon. Extensions to the sparse regime can be found in ~\cite{lacker2019local, oliveira2019interacting,bet2020weakly, bayraktar2020graphon, oliveira2020interacting}. We note that our McKean-Vlasov limit describes the evolution of the graphon itself, and not the distribution of any particle system. We borrow the name McKean-Vlasov to stress that each edge-weight evolves by an \textit{ensemble} effect of all the other edge weights, but that ensemble is a graphon and not the empirical distribution of any particle system as done in the papers cited above.

Notice that~\eqref{eq:exampleR} arises as the limit of the projected stochastic gradient descent algorithm, which is used in practice to optimize $R_n$. As mentioned above, we establish that the curves described by~\eqref{eq:exampleR} converge to a deterministic curve on the space of graphons. In the zero-noise limit, the (deterministic) limiting curve on the space of graphons is a gradient flow and hence converges to the minimizer exponentially fast. Thus, the evolution~\eqref{eq:exampleR} gives a way to numerically approximate the minimizer. More generally, the limiting curve converges to stationary points and thus~\eqref{eq:exampleR} provides an algorithm to numerically approximate these stationary points that may be useful in obtaining reasonable guesses regarding the structure of the minimizers in such problems. We describe the projected gradient descent and projected stochastic gradient descent algorithms in more detail in the following paragraphs.

Projected Gradient Descent (GD) based algorithms are the workhorse in optimizing such functions~\cite{cauchy1847methode,bubeck2015convex,bottou2018optimization}. However, in most cases, computing gradients can be computationally intensive. In practice, stochastic approximation algorithms based on projected Stochastic Gradient Descent (SGD) are instead used to minimize such functions since they are often faster to simulate~\cite{robbins1951stochastic,kiefer1952stochastic}. The details of this common Markov chain are described later in the section, and the reader can refer to the monographs~\cite{benaim1999dynamics,kushner2003stochastic,borkar2009stochastic,moulines2011nonasymptotic,kushner2012stochastic} for a detailed overview. Roughly, if the current state is a symmetric matrix $A$, one jumps to a new state by taking a small step along the negative Euclidean gradient $-\nabla R_n(A)$, and potentially adding independent, centered, and variance-bounded noise to each matrix entry (up to symmetry). Each matrix entry is then projected onto the interval $\interval{0,1}$ to satisfy the entrywise constraint.

Gradient descent (GD), with small step sizes, approximates the Euclidean gradient flow obtained as a solution to Cauchy's problem
\[
    \dot{A}_{i,j}(t) = -\nabla_{i,j}R_n(A(t)), \qquad (i,j)\in\squarebrack{n}^{2}, \qquad t\in\R_+,
\]
in the interior of $\Mcal_n^{0}$. Here $\R_+$ denotes the set of non-negative real numbers, which is used to index time, $\nabla_{i,j}$ refers to the partial derivative with respect to the $(i,j)$-th matrix entry. It is therefore natural to understand a suitable scaling limit of SGD on the space of such matrices. 

A previous work~\cite[Theorem 4.17]{oh2021gradient} showed that under suitable assumptions on $\round{R_n}_{n \in \Natural}$, the implicit Euler update scheme approximates a gradient flow curve, in an appropriate sense, over the space of \emph{graphons}, $\Graphons$, when the step size is taken to zero and $n$ grows to infinity. The reader is referred to~\cite{lovasz2006limits,borgs2008convergent,borgs2012convergent} and Section~\ref{sec:background_graphons} for the required exposition on graphons. In this work, we ask a similar question for SGD-based algorithms. We show that under an appropriate ``small noise'' assumption and a consistency and other suitable assumptions on the functions $\round{R_n}_{n \in \Natural}$, the SGD iterations converge appropriately to a limiting deterministic curve that is a gradient flow on the space of graphons. Moreover, when an extra Gaussian noise is added to each SGD iterate, the noisy SGD iterations also converge to a deterministic curve on graphons which admits a McKean-Vlasov description. Similar McKean-Vlasov system has been studied in~\cite{athreya2023path}, however, the focus of~\cite{athreya2023path} is to study a particular Markov chain on large graphs, namely a version of the Metropolis Markov chain. These Markov chains are designed to mimic the gradient flow in the limit.

Very roughly, $\Wcal$, the set of bounded symmetric measurable functions on $[0, 1]^2$ or {\em kernels}, is our limiting space for symmetric matrices. The set of graphons, $\Graphons$, is obtained as a quotient of $\Wcal$ where we identify two kernels to be the same if one can be obtained from the other by using the same measure-preserving transformation on its ``rows'' and ``columns'' (see Section~\ref{sec:background_graphons}). Thus, a function $R\colon\Graphons\to\R$ over graphons naturally extends to a function over the set of kernels $\Wcal$.  For any $n\in\Natural$, the set of symmetric matrices $\Mcal_n$, over which algorithms like GD and SGD operate on, can be naturally identified with a subset, \emph{finite dimensional kernels}, $\Wcal_n\subset \Wcal$ of the kernels (see Section~\ref{sec:background_graphons} for details). This identification/embedding will be denoted by $K$ (as in kernel) and its inverse will be denoted by $M_n$ (as in matrix). Using $K$, the restriction of the function $R$ to $\Wcal_n$ can be viewed as a function $R_n$ on $\Mcal_n$.
\sloppy Define the projection operator $P\colon\R\to\interval{-1,1}$ as 
%$P(x) \coloneqq \argmin_{z\in\interval{-1,1}}\abs{x-z}^2 = -\indicator{(-\infty,1)}{x} + x\indicator{\interval{-1,1}}{x} + \indicator{(1,\infty)}{x}$ for $x\in\R$, where $\indicator{A}{}$ is the indicator function of $A\subseteq\R$. 
\[
     P(x) \coloneqq \begin{cases}
         -1 & \text{if } x\in(-\infty,-1),\\
         x & \text{if } x\in\interval{-1,1},\\
         1 & \text{if } x\in(1,\infty).
     \end{cases}
\]
The operator $P$ can be used coordinatewise on matrices and kernels. For every $n\in\Natural$, let $\tauvec_n \coloneqq \round{\tau_{n,k}}_{k\in\Integer_+}$, be a sequence of positive step sizes (also known as the learning rate). Here $\Integer_+$ denotes the set of all non-negative integers. Given the step size sequence $\tauvec_n$, we can define a monotonically increasing sequence of times $\round{t_{n,k}}_{k\in\Integer_+}$, defined as a cumulative sum of $\tauvec_n$, i.e., $t_{n,0} = 0$ and $t_{n,k}\coloneqq \sum_{j=0}^{k-1}\tau_{n,j}$ for any $k\in\Natural$. We assume $\tauvec_n$ to have a divergent sum so to cover the whole non-negative real line $\R_+$, i.e., to satisfy $\lim_{k\to\infty}t_{n,k} = \infty$. We define the norm of the step size sequence $\tauvec_n$ as $\abs{\tauvec_n} \coloneqq \sup_{k\in\Integer_+}\tau_{n,k}$, which is assumed to be finite. We now describe our first iterative scheme.

\begin{definition}[Projected GD]\label{def:PGD}
    Let $n\in\Natural$ and let $R_n\colon \Mcal_n\to \R$ be a differentiable function. The projected GD iterates of $R_n$ starting at $V_{n,0}\in\Mcal_n$ is defined to be a sequence of symmetric matrices $\round{V_{n,k}}_{k\in\Integer_+}$ given iteratively as
    \begin{align}
        V_{n,k+1} = P\round{V_{n,k} - n^2\tau_{n,k}\nabla R_n\round{V_{n,k}}} ,\qquad k\in\Integer_+ .\label{eq:PGD}\tag{PGD}
    \end{align}
\end{definition}
There is a natural notion of gradient of functions defined on $\Graphons$ that we call Fr\'echet-like derivative (see Definition~\ref{def:frechet_like_derivative}), and is related to the Euclidean gradients in finite dimensions by a scaling of $n^2$. Suppose $R$ is such a function whose Fr\'echet-like derivative evaluation map is denoted by $\phi$. If $R_n$ is obtained from $R$ by restricting $R$ to $\Mcal_n$ and the function $R_n$ is differentiable up to the boundary of $\Mcal_n$ for every $n\in\Natural$, then it is shown in~\cite[Lemma 4.16]{oh2021gradient} that
\begin{align}
    n^2\nabla R_n = M_n \circ \phi \circ K.\label{eq:scaling_gradient}
\end{align}
Simply put, $n^2$ times the Euclidean gradient of $R_n$ at a matrix argument $A$ can be identified as the Fr\'echet-like derivative $\phi$ of $R$ at the kernel argument $K(A)$. The time in the Euclidean gradient in Definition~\ref{def:PGD} is therefore scaled by $n^2$ following the relation~\eqref{eq:scaling_gradient}. The~\ref{eq:PGD} algorithm is essentially the explicit Euler iteration scheme up to the projection.

We now define the stochastic optimization setup for $R_n$. In order to do so, we first fix some notations and make some assumptions on $R$ and $R_n$. Let $\round{\xi_{k+1}}_{k\in\Integer_+}$ be an i.i.d. sequence of random variables with some distribution $\Dcal$ over some arbitrary measurable space $(\Omega,\Acal)$. Let $g\colon \Wcal \times \Omega \to L^\infty\big(\interval{0,1}^{(2)}\big)$ where $L^\infty\big(\interval{0,1}^{(2)}\big)$ is the set of all bounded measurable functions $\phi:[0, 1]^2\to \R$ such that $\phi(x, y)=\phi(y, x)$. To emphasize that $\phi$ is symmetric, we denote the domain by $[0, 1]^{(2)}$ which denotes the set $\{(x, y)\in [0, 1]^2: x\leq y\}$. Define $g_n$ on $\Mcal_n \times \Omega$ as $g_n(A;\xi) = g(K(A);\xi)$ for every $n\in\Natural$ and $A\in\Mcal_n$, and assume that
\begin{align}
    \nabla R_n = \Exp{\xi\sim\Dcal}{g_n(\slot{};\xi)}.
\end{align}

% We now define the SGD iterates of $R_n$. In order to do so, we first fix some notations and make some assumptions on $R$ and $R_n$. Let $\round{\xi_{k+1}}_{k\in\Integer_+}$ be an i.i.d. sequence of random variables with some distribution $\Dcal$ over some arbitrary measurable space $(\Omega,\Acal)$. Let $\ell\colon\Wcal\times \Omega \to \R$ be a function and let
% \begin{align}
%     R \coloneqq \Exp{\xi\sim\Dcal}{\ell(\slot{};\xi)}.
% \end{align}

Under suitable assumptions (see Assumption~\ref{asmp:small_noise}) on the function $g$, the function $R$ is invariant under measure preserving transformations and hence defines a function on $\Graphons$.
% Under suitable assumption (see Assumption~\ref{asmp:small_noise}), on function $\ell$, the function $R$ is invariant under measure preserving transformations and hence defines a function on $\Graphons$.
We are interested in stochastic analogues of the iteration scheme in Definition~\ref{def:PGD}, for such a function $R$, possibly with a noise at each iteration. In other words, our interest lies in noisy variations of projected GD iterations (see Definition~\ref{def:PGD}). In this setting, we will consider two ways to introduce noise at each iteration.
\begin{enumerate}
    \item\label{item:small_noise} {\bf Small noise}: We can replace the Euclidean derivative $\nabla R_n$ in equation~\eqref{eq:PGD} by its unbiased stochastic proxy $g_n\round{\slot{}; \xi_{k+1}}$. As a special case, $g$ can be obtained from a function $\ell\colon\Wcal \times \Omega \to \R$, as $g(\slot{};\xi) \coloneqq (D_{\Wcal})\ell(\slot{};\xi)$ for all $\xi\in\Omega$, where $(D_{\Wcal}\ell)(\slot{};\xi)$ is the Fr\'echet-like derivative (see Definition~\ref{def:frechet_like_derivative}) of $\ell(\slot{};\xi)$. Such a stochastic approximation is known as Stochastic Gradient Descent (SGD).
    % \item {\bf Small noise}: We can replace the Euclidean derivative $\nabla R_n$ in equation~\eqref{eq:PGD} by its unbiased stochastic proxy $\nabla \ell_n\round{\slot{}; \xi_{k+1}}$, where $\ell_n\colon \Mcal_n\times  \Omega  \to \R$ is the restriction of $\ell$ on $\Mcal_n$, defined as $\ell_n(\slot{};\xi) \coloneqq \ell\round{\slot{};\xi}\circ K$ for all $\xi\in \Omega$.
    \item\label{item:large_noise} {\bf Large noise}: We can add an additive noise to iterates in equation~\eqref{eq:PGD} before the projection, as we describe in Definition~\ref{def:PNSGD} below.
\end{enumerate}

We can now define the noisy analogs of~\eqref{eq:PGD}, that is, \emph{projected (noisy) SGD}. We will use the operator $\circ$ over symmetric matrices to denote the Hadamard (elementwise) product.

\begin{definition}[Projected SGD with and without noise]\label{def:PNSGD}
    Let $n\in\Natural$. Starting at $W_{n,0}\in\Mcal_n$, the projected (noisy) SGD algorithm produces a sequence of iterates $\round{W_{n,k}}_{k\in\Integer_+}$ defined as
    \begin{align}\label{eq:PNSGD}
        W_{n,k+1} = P\round{W_{n,k} - n^2\tau_{n,k}g_n\round{W_{n,k};\xi_{k+1}} + \tau_{n,k}^{1/2}G_{n,k}} ,\qquad k\in\Integer_+. \tag{PNSGD}
    \end{align}
    Here $\round{G_{n,k}}_{k\in\Integer_+}$ is an $n\times n$ symmetric matrix valued martingale difference sequence independent of $\round{\xi_{k+1}}_{k\in\Integer_+}$. We only consider the noise $G_{n,k}$, for $k\in\Integer_+$, of the form $G_{n, k}=\Sigma_n(W_{n,k})\circ Z_{n,k}$ for some $\Sigma_n$ that maps matrices in $\Mcal_n$ to $n\times n$ symmetric matrices with non-negative entries and $\round{Z_{n,k}}_{k\in\Integer_+}$ is a sequence of independent $n\times n$ symmetric random matrices with standard normal entries (up to matrix symmetry).
\end{definition}

Due to the natural identification of $\Mcal_n$ with $\Wcal_n$, the GD iterates $\round{V_{n,k}}_{k\in\Integer_+}\subset \Mcal_n$ and the SGD iterates $\round{W_{n,k}}_{k\in\Integer_+}\subset \Mcal_n$ in Definitions~\ref{def:PGD} and~\ref{def:PNSGD} respectively, can be viewed as kernel valued iterates $\big(V^{(n)}_{k}\big)_{k\in\Integer_+}\subset \Wcal_n$ and $\big(W^{(n)}_{k}\big)_{k\in\Integer_+}\subset \Wcal_n$, under the embeddings $V^{(n)}_k = K(V_{n,k})$ and $W^{(n)}_k = K(W_{n,k})$ respectively for $k\in\Integer_+$. This allows us to interpret~\eqref{eq:PGD} and~\eqref{eq:PNSGD} as kernel-valued updates.  

We consider piecewise constant interpolations of the iterates (see Definition~\ref{def:interpolation}) and in this paper, we establish the existence of the scaling limit of these curves. We also characterize the limit under the absence of ``large noise". Our limiting procedure takes two steps. First, for every fixed $n\in\Natural$, we take the step size, i.e., $\abs{\tauvec_n}\to 0$ to obtain a limiting SDE on $\Mcal_n$. We then characterize the limit of the SDEs as $n\to\infty$ as an absolutely continuous curve on the space of graphons. 

\iffalse
Assumption~\ref{asmp:bdd_var} in Assumption~\ref{asmp:small_noise} essentially puts an assumption on the variance of the Fr\'echet-like derivatives of $\ell\round{\slot{};\xi}$ for a.e. $\xi\in\Omega$.\RaghavS{For $\ell$ to have a Fr\'echet-like derivative, it needs to be an invariant function, which needs property~\ref{item:ell_invariance}.} The assumption essentially requires that the random variable $D_\Wcal \ell(A;\xi)$ for $\xi\sim\Dcal$ has uniformly bounded variance for all $A\in\cup_{n\in\Natural}\Wcal_n$. That is, there exists $\sigma \geq 0$ such that 
\[
    \Exp{\xi\sim\Dcal}{\enorm{(D_\Wcal \ell)(A;\xi) - \phi(A)}^2} \leq \sigma^2,
\]
where $(D_\Wcal \ell)(A;\xi)$ is the Fr\'echet-like derivative of $\ell(\slot{};\xi)$ at $A\in\cup_{n\in\Natural}\Wcal_n$.\RaghavS{We can replace Assumption~\ref{asmp:bdd_var} in Assumption~\ref{asmp:small_noise} with this kernel, dimension independent form. The finite dimensional version can be derived using it.}\RaghavS{Need an example to justify this assumption. $\Rightarrow$} 
\fi

% We now state our main results. The appropriate assumptions are described in Section~\ref{sec:background}.  Theorem~\ref{thm:SGD_to_SDEn} states that for every fixed $n\in\Natural$, the curves obtained by piecewise constant interpolations of the projected noisy SGD iterations in~\eqref{eq:PNSGD} converge, as $\abs{\tauvec_n}\to 0$, to a random continuous curve which can be described by the solution of an SDE on $\Mcal_n$. 

\begin{theorem}\label{thm:SGD_to_SDEn}
    Let $n\in \N$ be fixed, and suppose Assumptions~\ref{asmp:R_ell_phi},~\ref{asmp:small_noise} and~\ref{asmp:large_noise} hold (see Section~\ref{sec:assumptions}). Let $W_n\colon\R_+\to\Mcal_n$ be the piecewise constant interpolation (Definition~\ref{def:interpolation}) of noisy SGD iterates $\round{W_{n, k}}_{k\in\Integer_+}$ as defined in~\eqref{eq:PNSGD}. Then, $W_n$ converges weakly in the space of c\`adl\`ag processes to $X_n$ as $\abs{\tauvec_n}\to 0$ that satisfies the SDE:
    %Let $X_n$ be a process on $\Mcal_n$, defined on some probability space, that satisfies the following Skorokhod SDE
    \begin{equation}\label{eq:RSDE}
        \diff X_n(t) = -n^2\nabla R_n(X_n(t))\diff t + \Sigma_n(X_n(t)) \circ \diff B_n(t) + \diff L^{-}_n(t) - \diff L^{+}_n(t),\tag{RSDE}
    \end{equation}
    for $t\in\R_+$, starting at $X_n(0)=W_{n, 0}$. Here $B_{n}$ is an $n\times n$ symmetric matrix valued process with coordinatewise independent standard Brownian motions up to matrix symmetry, and $\round{X_{n},L_{n}^+,L_{n}^-}$ solves the Skorokhod problem with respect to the set $\Mcal_n$ (see Section~\ref{sec:RBM}).
    %Then there exists a process $\widetilde{W}_n\colon\R_+ \to \Wcal_n$ defined on the same probability space as $X_n$ such that the $\mathrm{Law}\big(\widetilde{W}_n\big)=\mathrm{Law}(W_n)$ and,
  %  \[
   %     \lim_{\abs{\tauvec_n} \to 0} \E{\sup_{s\in\interval{0,T}} \enorm{\widetilde{W}^{(n)}(s)-X^{(n)}(s)}^2} = 0,
   % \]
   % for any $T>0$, where $\widetilde{W}^{(n)}(s) = K\big(\widetilde{W}_n(s)\big)$ and $X^{(n)}(s) = K(X_n(s))$ for all $s\in\R_+$.
\end{theorem}

Note that the diffusion coefficients in~\eqref{eq:RSDE} act diagonally on the Brownian increments for each coordinate of the matrix valued process. In practice it makes sense to consider non-diagonal diffusion coefficients as an approximation to SGD. See~\cite{li2019stochastic} for a discussion.
%By discretizing such an SDE one recovers the standard SGD iteration scheme under the ``small noise'' setup ). % for a discussion where the authors develop a weak approximation framework to analyze SGD with positive step sizes.  
Practitioners also use variants of SGD under the ``small noise'' setup where instead of having a single unbiased stochastic proxy of the gradient, an average over independent batches of stochastic gradients is used at every step. Authors in~\cite{malladi2022on} derive weak SDE approximations of various popularly used stochastic optimization algorithms that use batches. However this existing literature does not cover SDEs with boundary terms.

Our main interest is in the limit of the kernel valued stochastic process $X^{(n)}(\cdot)=K\round{X_n(\cdot)}$ (Theorem~\ref{thm:SGD_to_SDEn}), as $n \rightarrow \infty$. This limit is a deterministic curve in $\Graphons$ that we now describe. Consider, for simplicity, the special case when each $\Sigma_n$ is $\beta$ times the identity matrix for some $\beta >0$. On a probability space that supports a standard linear Brownian motion $B_{1,2}(\cdot)$ and a pair of independent $\mathrm{Uni}\interval{0,1}$ random variables $(U_1, U_2)$ and given some $W_0 \in \Wcal$, one can construct a unique solution of the following family of one-dimensional reflected diffusions. Given $(U_1,U_2)=(x,y)$, for some $(x,y)\in \interval{0,1}^{(2)}$, let $X_{1,2}$ be a diffusion with state space $\interval{-1,1}$ with the initial condition $X_{1,2}(0)=W_0(x, y)$, and satisfying \begin{align}\label{eq:infinite_SDE0}
    \diff X_{1,2}(t) &= -\phi\left( \Gamma (t)\right)(x, y)\diff t + \beta \diff B_{1,2}(t) + \diff L^-_{1,2}(t) - \diff L^+_{1,2}(t),
\end{align}
for some $\beta\in\R_+$ and $t\in \R_+$. Here, $\phi$ is the Fr\'echet-like derivative of $R$ in \eqref{eq:scaling_gradient}, $L^-_{1,2}$ and $L^+_{1,2}$ are the local time processes such that $(X_{1,2},L^+_{1,2},L^-_{1,2})$ solves the Skorokhod problem with respect to $\interval{-1,1}$ (see Section~\ref{sec:RBM}). The kernel-valued process $\Gamma\colon\R_+ \to \Wcal$ is given by
\begin{equation}\label{eq:whatisgammat0}
    \Gamma(t)(u,v) \coloneqq  \E{ X_{1,2}(t) \given (U_1,U_2)=(u,v)},\quad \forall\; (u,v)\in\interval{0,1}^{(2)},
\end{equation}
and any $t\in\R_+$. In Proposition~\ref{prop:exists}, we show that the coupled system $(X_{1,2}, \Gamma)$ exists in a strong sense and is pathwise unique and that the kernel-valued process $X^{(n)}$ in Theorem~\ref{thm:SGD_to_SDEn} converges to the curve $\Gamma$ in the following sense an $n\to\infty$.

\begin{theorem}\label{thm:SDEn_to_infty}
    Suppose Assumptions~\ref{asmp:R_ell_phi},~\ref{asmp:large_noise}, and~\ref{asmp:phi_cut_lip} hold (see Section~\ref{sec:assumptions}). Then, for any sequence of initial kernels $\big(W^{(n)}_0 \in\Wcal_n\big)_{n\in\Natural}$ that converges in  $L^2\big(\interval{0,1}^{(2)}\big)$ norm $\enorm{}$, i.e., 
    \begin{equation}\label{asmp:initassump}
        \lim_{n\rightarrow \infty}\enorm{W_{0}^{(n)}-W_0}=0,
    \end{equation}
    the process of random kernels $\round{X^{(n)}(t)=K(X_n(t))}_{t\in\R_+}$ obtained from solutions of the SDE~\eqref{eq:RSDE}, converges locally uniformly in the cut norm, in probability, to the curve $\Gamma \colon\R_+ \to \Wcal$, with $\Gamma(0)=W_0$, defined in equation~\eqref{eq:whatisgammat0}  as $n\to\infty$.
\end{theorem}
\begin{remark}
    The assumption $\norm{2}{W_0^{(n)}-W_0}\to 0$ can not be weakened to $\norm{\cut}{W_0^{(n)}-W_0}\to 0$ as $n\to \infty$. To see this, take $\nabla R_n\equiv 0$ 
    and $\Sigma\equiv 1$ and let $W_0\equiv 0$.  It is clear that $\Gamma(t)\equiv 0$ for all $t\geq 0$.  On the other hand, let $\xi$ be a random variable taking values $-1/2$ and $+1$ with probability $2/3$ and $1/3$ respectively. And, let $W_0^{(n)}$ be the step-kernel corresponding to $n\times n$ symmetric random matrix whose entries (on and above the diagonal) are  i.i.d. and has the same distribution as $\xi$. Then, $\norm{\cut}{W_0^{(n)}-W_0}\to 0$ almost surely. However, in this case,  the coordinates of $X_n$ are i.i.d. (up to the matrix symmetry) and have the same distribution as an RBM (reflected at $\pm 1$) with initial distribution $\xi$. In particular, $K(X_n(t))$ converges to $W(t)\equiv \E{X_{n, 1, 2}(t)}$.  It is therefore sufficient to show that $\E{X_{n, 1, 2}(t)}$ is not identically $0$ for a.e. $t\in\R_+$. 
    
    To see this, we argue by contradiction. If $\E{X_{n, 1, 2}(t)}=0$ for all $t\geq 0$ then $\frac{\diff}{\diff t}\E{X_{n, 1, 2}(t)}=0$. Using~\cite[Exercise 1.12, pg-407]{RY}, we obtain that $\frac{\diff}{\diff t}\E{X_{n, 1, 2}(t)}=\frac{2}{3}(p_t(-\frac{1}{2})-p_t(\frac{3}{2}))+\frac{1}{3}\round{p_t(2)-1}\neq 0$, where $p_t$ is the standard heat kernel at time $t$. This yields a contradiction. 
\end{remark}

\begin{remark}
    We should also remark that arranging for $W^{(n)}_0$ such that $\norm{2}{W^{(n)}_0-W_0}\to 0$ as $n\to \infty$ is not difficult. For any $W_0$ and $n\in\Natural$, let $W^{(n)}_0$ be the $L^2\big([0,1]^{(2)}\big)$ projection of $W_0$ on $\Wcal_n$. Then $W^{(n)}_0$ satisfies this condition.
\end{remark}

In Section~\ref{sec:convergence_sde} a more general statement with state-dependent diffusion has been proved (see Proposition~\ref{prop:mck_vlasov}). It is worth noting that presence of noise and the boundary $\set{-1,1}$ in our problem makes it non-trivial.
% To elucidate, consider the following iteration scheme without the boundary projection,
%  \begin{align}\label{eq:NSGD}
%     W_{n,k+1} = W_{n,k} - n^2\tau_{n,k}g_n\round{W_{n,k};\xi_{k+1}} + \tau_{n,k}^{1/2}G_{n,k} ,\qquad k\in\Integer_+ ,\tag{NSGD}
% \end{align}
% starting at some $n\times n$ symmetric matrix $W_{n, 0}$, where $\round{G_{n, k}}_{k\in\Integer_+}$ is a $n\times n$ symmetric matrix valued martingale difference sequence with independent and centered entries. Since for all $k\in\Integer_+$, we have $\lim_{n\to\infty}\cutnorm{G_{n, k}} = 0$, it easily follows that the additive noise does not matter in the limit. But this is not the case when there is a boundary involved as in the iteration scheme~\eqref{eq:PNSGD}.
%To see this, take the simplest case when $g_n = \nabla \ell_n\equiv 0$, and consider the iteration scheme~\eqref{eq:PNSGD}, but without the projection, say starting at $W_{n, 0}\in\Mcal_n$.
% The gradient flow in this case should be a constant curve. The corresponding GD iteration would also stay at $W_{n, 0}$ at all times $t\in\R_+$, for each $n\in\Natural$.
% If we consider the iteration scheme~\eqref{eq:NSGD}, then $W_n$ is a random curve over $n\times n$ symmetric matrices for each $n\in\Natural$ with $\E{W_{n}(t)}=W_{n, 0}$ for all $t\in\R_+$.
To see this, consider~\eqref{eq:RSDE} for a constant function $R_n$ (i.e., $\nabla R_n\equiv 0$) and without the local times, say starting at $W_{n, 0}\in\Mcal_n$. The solution is a symmetric matrix of independent Brownian motions. It can be easily checked that, if $\lim_{n\rightarrow \infty}\cutnorm{W_{n,0} - W_0}=0$, then $\lim_{n\to \infty}\sup_{t\in\interval{0,T}}\cutnorm{X^{(n)}(t)-W_0} = 0$ for any finite $T>0$. However, if we consider~\eqref{eq:RSDE} again with $\nabla R_n\equiv 0$ but with reflection at the boundary, the coordinate processes are independent reflected Brownian motions. In this case the cut limit of $X^{(n)}(t)$ is also the cut limit of the kernel $\E{X^{(n)}(t)}$. But reflecting Brownian motions do not have constant expectations in time due to boundary effect. Hence, the limit of $X^{(n)}(t)$ is not constant in $t$. But, if this limit were a gradient flow, it would be a constant.

\subsection{Scaling limit without added noise}\label{sec:beta_0}
When $\Sigma_n \equiv 0$, equation~\eqref{eq:RSDE} reduces to
\begin{align}
    \diff X_n(t) = -n^2\nabla R_n(X_n(t))\diff t + \diff L^{-}_n(t) - \diff L^{+}_n(t),\quad t\in\R_+,\quad X_n(0) = W_{n,0},\label{eq:RSDE_beta0}
\end{align}
such that $(X_n,L^+_n,L^-_n)$ solves the Skorokhod problem on $\Mcal_n$ (see Section~\ref{sec:RBM} for details). Moreover, it is shown in Section~\ref{sec:convergence_nsgd} that the solution of~\eqref{eq:RSDE_beta0} is the same as the solution of~\eqref{eqn:GF_n} given below. Furthermore, it is shown in~\cite[Theorem 4.4, Theorem 4.14]{oh2021gradient} that if the solution $X_n\colon\R_+\to\Mcal_n$ of
\begin{equation}\label{eqn:GF_n}
    \diff X_n(t) = -n^2\nabla R_n(X_n(t))\circ\indicator{G_n(X_n(t))}{}\diff t,\qquad t\in\R_+,
\end{equation}
exists, where $G_n(A)$ is the subset of $\squarebrack{n}^{2}$ (defined in equation~\eqref{eq:G_n} later in Section~\ref{sec:without_noise}),
% \RaghavS{expand below def.}
% \begin{align*}
%     G_n(A)&\coloneqq \set{\abs{A}<1}\cup \set{A=1, \nabla R_n(A)>0}\cup \set{A=-1, \nabla R_n(A)<0},
% \end{align*}
% for all $A\in\Mcal_n$,
then $X_n$ is a gradient flow on $\Mcal_n$ in a suitable sense. Further, it is shown in~\cite[Theorem 4.17]{oh2021gradient} that under reasonable assumptions on $R$, the sequence of solutions $\round{X_n}_{n\in\Natural}$ of equation~\eqref{eqn:GF_n} obtained for all natural numbers $n\in\Natural$, converge to an absolutely continuous curve $W\colon\R_+\to\Wcal$ (appropriately in the cut metric (see Definition~\ref{def:cut_metric})), which is a curve of maximal slope~\cite{ambrosio2005gradient} (a.k.a. gradient flow) of $R$, as $n\to\infty$. This yields the following.
\begin{theorem}\label{thm:zeroTempLimit}
    Suppose Assumptions~\ref{asmp:R_ell_phi} and~\ref{asmp:small_noise} hold (see Section~\ref{sec:assumptions}). Let $R$ be continuous in the cut norm, and $\lambda$-semiconvex with respect to $\enorm{}$ for some $\lambda\in\R$ (see Section~\ref{sec:background_graphons} for definitions). For every $n\in\Natural$, let $X_n\colon \R_+ \to \Mcal_n$ be a gradient flow of $R_n$ staring at $X_n(0) = W_{n,0} = M_n\big(W^{(n)}_0\big)\in\Wcal_n$, and satisfying equation~\eqref{eq:RSDE_beta0}. If $\big(W^{(n)}_0\big)_{n\in\Natural}$ converges to $W_0\in\Wcal$ in the cut norm, then,
    \[
        \lim_{n\to \infty}\sup_{s\in[0,T]}\norm{\cut}{K\round{X_n(s)}-W(s)} =  0,
    \]
    for any $T>0$, where $W$ defined as $W(t) \coloneqq W_0-\int_0^t\phi(W(s))\indicator{G_{W(s)}}{}$ for $t\in\R_+$, is the gradient flow for $R$.
\end{theorem}

We should mention that our method allows us to also obtain a non-asymptotic rate of convergence. We refer the reader to Remark~\ref{rem:Rate} for details.

As an example, consider the function $R_n$ considered at the beginning of Section~\ref{sec:intro}. $R_n$ is the restriction to $\Mcal_n^{0}$ of the function $R$ on $\Wcal_0\coloneqq \set{W\in \Wcal\given W(x, y)\in [0,1] \text{ for a.e. } (x,y)\in[0,1]^{(2)} }$ given by 
\[
    R(W) = \inv{2}(H_{\mathrel{-}}(W)-e)^2+\inv{2}(H_{\triangle}(W)-\tau)^2+\Ecal(W),
\]
where $\Ecal\coloneqq\int_0^{1}\int_0^{1} h(W(x, y))\diff x\diff y$. The function $H_{F}$ is the homomorphism density of $F$~\cite[Section 5.1.2]{oh2021gradient}. The function $R$ satisfies all the assumptions of Theorem~\ref{thm:zeroTempLimit}. See Section~\ref{subsec:EdgeTriangle} for details.

\subsection{SGD and permutation symmetries in Deep Neural Networks (DNNs)}
We end this section with a significant example where the permutation invariant functions arise, namely, DNNs.
DNNs typically consist of a sequence of matrices that share row/column labels with their adjacent ones. Most modern DNNs possess permutation symmetries in their parametric representations. That is, their output is invariant under permutations applied to the rows/columns of the matrices appearing in DNN representation. The goal is to obtain the sequence of matrices that minimizes the risk function $R_n$ for $n\in\Natural$. This can be thought of as a generalization of the linear regression example discussed in the introduction and in Section~\ref{sec: Linear regression}. Authors in~\cite{ainsworth2022git} empirically study the effectiveness of SGD in optimizing the non-convex DNN risk functions $R_n$ for large $n\in\Natural$. For simplicity, consider the special case when the DNN is parameterized through a single finite symmetric matrix and therefore does not involve shared labels. Let $\round{U_{n,k}}_{k\in\Integer_+}$ and $\round{V_{n,k}}_{k\in\Integer_+}$  be the SGD iterations, starting at two independent initializations, say, $U_{n,0} \neq V_{n,0}$. %sampled from the same distribution on $\interval{-1,1}$ 
Authors in~\cite{ainsworth2022git} observe that $\round{U_{n,k}}_{k\in\Integer_+}$ and $\round{V_{n,k}}_{k\in\Integer_+}$ can be ``aligned'' by optimizing over the set of all permutations. That is, for every $k\in\Integer_+$, they solve for
\[
    \pi^*_k \in \argmin_{\pi_k\in S_n}\normF{U_{n,k} - V_{n,k}^{\pi_k}}^2,
\]
where $\normF{}$ denotes the Frobenius norm, $S_n$ is the set of all permutations of $[n]$, and $V_{n,k}^{\pi_k}$ is the matrix $V_{n,k}$ with rows and columns relabeled by the permutation $\pi_k\in S_n$. The authors observe an emergent property of SGD called ``linear mode connectivity'' (LMC)~\cite{frankle2020linear}. This property essentially says that $R_n$ does not fluctuate a lot on $W_{n,k}(\lambda)$ for large $k\in\Integer_+$, where
\[
    W_{n,k}(\lambda) = (1-\lambda)U_{n,k} + \lambda V_{n,k}^{\pi_k^*}, \qquad \lambda\in [0, 1].
\]
%That is $R_n\round{W_{n,k}(\lambda)} - R_n(U_{n,k})$ and $R_n\round{W_{n,k}(\lambda)} - R_n\big(V_{n,k}^{\pi^*_k}\big)$ are not large for any $\lambda\in[0,1]$. 
Further, they observe that $R_n(W_{n,k}(\lambda))$ approaches a constant uniformly on $\lambda\in[0,1]$ as $n$ goes to infinity. Authors in~\cite{benzing2022random} observe through experiments that for a fixed and large enough $k\in\Integer_+\setminus\set{0}$, the permutation $\pi^*_k$, has negative convexity gap
\[
    R_n\round{(1-\lambda)U_{n,0} + \lambda V_{n,0}^{\pi_k^*}} - \squarebrack{(1-\lambda) R_n(U_{n,0}) + \lambda R_n\round{V_{n,0}^{\pi_k^*}}}.
\]
Following these empirical observations and the hypothesis made by the authors in~\cite{entezari2022the}, it makes sense to consider DNNs up to their permutation symmetries, and as a consequence, study limiting behaviors of stochastic optimization algorithms over the space of graphons. This requires some generalization of our theory and is an important direction for future work. 

%\clearpage
\section{Background, Assumptions and Setup}\label{sec:background}

Since we want to obtain continuous time scaling limits of the iterative schemes defined in Definition~\ref{def:PGD} and Definition~\ref{def:PNSGD}, we will use piecewise constant interpolations.
\begin{definition}[Piecewise constant interpolation]\label{def:interpolation}
    Given a sequence $\round{a_k}_{k\in\Integer_+}$ over any domain, and a sequence of positive step sizes $\tauvec = \round{\tau_{k}}_{k\in\Integer_+}$, we can define a piecewise constant interpolation of $\round{a_k}_{k\in\Integer_+}$ as a right-continuous curve $a\colon \R_+ \to \set{a_k}_{k\in\Integer_+}$ as
    \begin{align*}
        a(t) \coloneqq a_{k} , \quad \text{if }\quad t\in\left\lbrack t_{k} ,t_{k+1}\right) ,
    \end{align*}
    for some $k\in\Integer_+$, where $t_{0} = 0$ and $t_{k}\coloneqq \sum_{j=0}^{k-1}\tau_{j}$ for any $k\in\Natural$.
\end{definition}
We now provide a background on graphons (see~\cite{lovasz2012large,janson2010graphons} for broader expositions).
\subsection{Background on Graphons}\label{sec:background_graphons}
Consider the set $\mathcal{S}$ of all bounded, Borel measurable function $W\colon [0,1]^{(2)}\to \mathbb{R}$ such that $W(x, y)=W(y, x)$ for a.e. $(x,y)\in\interval{0,1}^{(2)}$. For any function $W\in \mathcal{S}$ one can define the \emph{cut norm}, $\cutnorm{}\colon\Scal \to \R_+$ as 
\begin{align}
    \cutnorm{W} \coloneqq \sup_{S, T\subseteq [0, 1]}\abs{\int_{S\times T}W(x, y)\diff x\diff y}, \qquad W\in\Scal,\label{eq:cut_norm}
\end{align}
where the supremum is taken over Borel measurable sets $S,T\subseteq\interval{0,1}$. The cut norm was first introduced in~\cite{frieze1999quick} in the context of matrices and was later extended to $\Scal$ in~\cite{borgs2008convergent}. In the following definitions, let $\Tcal$ denote the set of all measure preserving transformations on $\interval{0,1}$ equipped with the Lebesgue measure. We say $W_1\cong W_2$ (i.e., $W_1$ and $W_2$ are {\em weakly isomorphic}) if there exists $W\in \mathcal{S}$ and measure preserving transformations $\varphi_1, \varphi_2 \in \Tcal$ such that for $W^{\varphi_i}\in \Scal$ defined as $W^{\varphi_i}(x,y) \coloneqq W\round{\varphi_i(x),\varphi_i(y)}$ for a.e. $(x,y)\in\interval{0,1}^{(2)}$ and $i\in[2]$, $W_1=W^{\varphi_1}$, and $W_2=W^{\varphi_2}$.

The cut norm $\cutnorm{}$ induces a metric called the \emph{cut metric}, denoted by $\delta_{\cut}$, when restricted to the quotient space $\widehat{\Scal}\coloneqq\mathcal{S}/{\cong}$. We denote the equivalence class of $W\in\Scal$ under weak isomorphism ($\cong$) as $\squarebrack{W}\coloneqq \set{U\in\Scal \given U\cong W}\in \widehat{\Scal}$. We now define the cut metric.

\begin{definition}[Cut Metric~{\cite[Section 3.2]{borgs2008convergent}}]\label{def:cut_metric}
    Let $[W_1], [W_2]\in \widehat{\Scal}$. Then, 
    \[
        \delta_{\cut}([W_1], [W_2])\coloneqq \inf_{\varphi, \psi\in\Tcal} \cutnorm{W_1^{\varphi}-W_2^{\psi}} .
    \]
    % where the infimum is taken over all measure preserving transformations $\varphi, \psi\in \Tcal$.
\end{definition}
More generally, given any norm $\norm{}{}$ on $\mathcal{S}$, one can define an induced metric $\delta_{\norm{}{}}$ on $\widehat{\Scal}$ as
\begin{align}
    \delta_{\norm{}{}}([W_1], [W_2])\coloneqq \inf_{\varphi, \psi\in\Tcal} \norm{}{W_1^{\varphi}-W_2^{\psi}} ,
\end{align}

In particular, the induced metric due to the $L^2$ norm, $\enorm{}\colon L^2(\interval{0,1}^{(2)})\to\R_+$, is called the \emph{invariant $L^2$ metric}, $\delta_2$, and it would be used in our discussion.

As defined in Section~\ref{sec:intro}, the set of kernels $\Wcal\subset \mathcal{S}$ is the set of measurable, symmetric functions $W\colon [0,1]^{(2)}\to [-1, 1]$ and correspondingly $\Graphons\coloneqq\Wcal/{\cong}$ is the set of graphons. For most of our discussion, we will be concerned only with the space of graphons equipped with either the cut metric $\delta_{\cut}$ or the invariant $L^2$ metric $\delta_{2}$. The metrics on $\Scal$ induced by the norms $\cutnorm{}$ and $\enorm{}$ with be denoted by $d_\cut$ and $d_2$ respectively.

For every $n\in\Natural$, the set $\Mcal_n$ can be naturally identified with a subset of $\Wcal$. Let $\Vcal_n\coloneqq \set{V_i}_{i\in\squarebrack{n}}$ be a partition of the interval $\interval{0,1}$ into contiguous intervals of equal length (Lebesgue measure). We define the set of kernels $\Wcal_n\subset\Wcal$ which contain kernels which are constant a.e. over sets in $\Vcal_n\times \Vcal_n$.

We note some crucial properties of these metric spaces that will be frequently used throughout this paper even without explicitly mentioning. 
\begin{enumerate}
    \item Properties of $\delta_\cut$:
        \begin{enumerate}
            \item The topology induced by the cut metric $\delta_\cut$ on $\Graphons$ is compact~\cite{Lovsz2007SzemerdisLF},~\cite[Section 9.3]{lovasz2012large}.
            \item Convergence in the cut metric is related to the convergence of homomorphism functions via the counting and the inverse counting lemmas~\cite[Section 7.2, Lemma 10.23, Lemma 10.32]{lovasz2012large}.
        \end{enumerate}
    \item Properties of $\delta_{2}$:
        \begin{enumerate}
            \item The metric space $(\Graphons,\delta_2)$ is a geodesic metric space~\cite[Theorem 3.5]{oh2021gradient}.
            \item The metric space $(\Graphons,\delta_2)$ is complete and separable but not compact.
            \item Convergence in $\delta_2$ implies convergence in $\delta_\cut$, implying that the topology generated by $\delta_2$ is stronger that the one generated by $\delta_\cut$ on $\Graphons$.
        \end{enumerate}
\end{enumerate}

As $(\Graphons, \delta_{2})$ is a geodesic metric space, it therefore makes sense to talk about \emph{geodesically convex} or \emph{geodesically semiconvex} functions.

% \begin{definition}[Constant speed geodesics w.r.t. $\delta_2$]\label{def:constant_speed_geodesics}
%     A curve $\omega\colon \interval{0,1}\to \Graphons$ is a constant speed geodesic w.r.t. $\delta_2$ if for all $0\leq r\leq s\leq 1$,
%     \begin{equation}
%         \delta_2\round{\omega(r),\omega(s)} = \delta_2\round{\omega(0),\omega(1)}\round{s-r}.\label{eqn:def_constant_speed}
%     \end{equation}
% \end{definition}

% \begin{definition}[$\lambda$-semiconvexity along curves w.r.t. $\delta_2$]\label{def:lambda_cvx_along_curve}
%     A function $R\colon \Graphons\to\R$ is said to be $\lambda$-semiconvex with respect to $\delta_2$ along a curve $\omega\colon \interval{0,1}\to \Graphons$ for some $\lambda\in\R$, if
%     \begin{align}
%         \round{R\circ \omega}(t) &\leq (1-t)\round{R\circ \omega}(0) + t\round{R\circ \omega}(1) - \inv{2}\lambda t(1-t)\delta_2^2\round{\omega(0),\omega(1)}\; ,\label{eq:lambda_cvx_along_curve}
%     \end{align}
%     for all $t\in\interval{0,1}$. Particularly, if the above inequality holds for $\lambda = 0$, then we say that $R$ is convex with respect to $\delta_2$ along the curve $\omega$.
% \end{definition}

\begin{definition}[$\lambda$-geodesic semiconvexity w.r.t. $\delta_2$]\label{def:semiconvexity}
    A function $R\colon \Graphons\to\R$ is $\lambda$-geodesically semiconvex with respect to $\delta_2$, if for any $[W_0],[W_1]\in \Graphons$ there exists a constant speed geodesic $\omega\colon\interval{0,1}\to \Graphons$ w.r.t. $\delta_2$ with $\omega(0)=[W_0]$ and $\omega(1)=[W_1]$ such that $R$ is $\lambda$-semiconvex on $\omega$ with respect to $\delta_2$ for some $\lambda\in\R$. (See~\cite[Definition 2.14-2.16]{oh2021gradient}).
\end{definition}

In Section~\ref{sec:intro}, we noted in equation~\eqref{eq:scaling_gradient} that Euclidean gradient $\nabla R_n$ of $R_n$ is closely related to what we call the Fr\'echet-like derivative of $R\colon\Wcal \to \R$. We state its definition below.
\begin{definition}[Fr\'echet-like derivative on $\Wcal$]\label{def:frechet_like_derivative}
    The Fr\'echet-like derivative of $R\colon \Wcal\to\R$ at $V\in\Wcal$ is given by $\phi(V) \in L^\infty\big(\interval{0,1}^{(2)}\big)$ that satisfies the following condition,
    \begin{align}
      \lim_{\substack{W\in\Wcal,\\\norm{2}{W-V}\to 0}}\frac{R(W) - R(V) - \round{ \inner{\phi(V),W}- \inner{\phi(V),V}}}{\norm{2}{W-V}} &= 0,\label{eq:frechet_limit_def}
    \end{align}
    where $\inner{\slot{},\slot{}}$ is the usual inner product on $L^2\big(\interval{0,1}^{(2)}\big)$. If $R$ admits a Fr\'echet-like derivative at every $V\in\Wcal$, we say that $R$ is Fr\'echet differentiable.
\end{definition}
\begin{remark}
    Note that here we define the Fr\'echet-like derivative for all functions if it exists, unlike as defined in~\cite[Definition 4.6]{oh2021gradient} where it is only defined for invariant functions. This is done so to allow $\ell(\slot{};\xi)$ (see item~\ref{item:small_noise} in Section~\ref{sec:intro}) to be Fr\'echet-differentiable for all $\xi\in\Dcal$ despite it not necessarily being an invariant function. %This restricts us from concluding that $[V]\in\Graphons$ and $[(D_\Wcal \ell)(V;\xi)]$ are `coupled' graphons for $V\in\Wcal$. Since $R = \Exp{\xi\sim\Dcal}{\ell(\slot{})}$ is invariant, both definitions agree, and it allows us to view $[V]$ and $[D_\Wcal R(V)]$ as coupled graphons.
\end{remark}

The scaling limit as we obtain in Theorem~\ref{thm:SDEn_to_infty}, under certain assumptions can be shown to be absolutely continuous with respect to $d_2$ (see Proposition~\ref{prop:velocity}). We state its definition for the sake of completeness.
\begin{definition}\label{def:AC}
    A curve $W\colon\R_+ \to \Wcal$ is absolutely continuous with respect to $d_2$ if there exists $m\in L^1(\R_+)$ such that for all $0\leq r<s<\infty$,
    \[
        d_2(W(r),W(s)) = \enorm{W(r)-W(s)} \leq \int_r^s m(t) \diff t.
    \]
    The set of all absolutely continuous curves on $(\Wcal,d_2)$ will be denoted by $\mathrm{AC}(\Wcal,d_2)$.
\end{definition}

\subsection{Assumptions}\label{sec:assumptions}
In this section we state all the required assumptions we need to prove our results (see Theorem~\ref{thm:SGD_to_SDEn} and Theorem~\ref{thm:SDEn_to_infty}).

\begin{assumption}\label{asmp:R_ell_phi}
    We make following assumptions on $R$, $g$ and $\phi$:
    \begin{enumerate}
        \item\label{item:Rn_C1} For every $n\in\Natural$, the function $R_n$ is in $C^1(\Mcal_n)$ up to the boundary of $\Mcal_n$.
        \item\label{item:phi_lip} The map $\phi$ is $\kappa_2$-Lipschitz with respect to $\enorm{}$, for some constant $\kappa_2\in\R_+$. That is,
        \[
            \enorm{\phi(W_1)-\phi(W_2)} \leq \kappa_2\enorm{W_1-W_2}, \qquad \forall\ W_1,W_2\in\Wcal.
        \]
        \item For every $n\in\Natural$, the function $g_n(\slot{};\xi) = g(\slot{};\xi) \circ K$ is in $C^0\round{\Mcal_n}$ up to the boundary of $\Mcal_n$ for all $\xi\in \Omega$.
        % \item Derivative and expectations commute on $\ell_n$, i.e., $\nabla R_n = \Exp{\xi\sim\Dcal}{\nabla \ell_n\round{\slot{}; \xi}}$.
    \end{enumerate}
\end{assumption}

\begin{assumption}\label{asmp:small_noise}
    We assume the following about the ``small noise''.% For every $n\in\Natural$.
    \begin{enumerate}
        % \item 
        % The stochastic gradient $\nabla \ell_n(\slot{};\xi)$ is pointwise bounded by an integrable function on $\Omega$ for a.e. $\xi\sim\Dcal$, implying $\nabla R_n = \Exp{\xi\sim\Dcal}{\nabla \ell_n\round{\slot{}; \xi}}$ by the Leibniz integral rule.
        %\item\label{asmp:bdd_var} The centered random variable $n^2\nabla \ell_n(\slot{};\xi) - n^2\nabla R_n$ for $\xi\sim\Dcal$ has uniformly bounded variance. That is, there exists $\sigma \geq 0$ such that 
        %\[
        %    \Exp{\xi\sim\Dcal}{\inv{n^2}\normF{n^2\nabla \ell_n(A_n;\xi) - n^2\nabla R_n(A_n)}^2} \leq \sigma^2,
        %\]
        %for all $A_n\in\Mcal_n$.
        \item Law of the random variable $g(W;\xi)$ for $\xi\sim\Dcal$ is invariant under measure preserving transformations for all $W\in\Wcal$, i.e., $\mathrm{Law}\round{g(W;\xi)} = \mathrm{Law}\round{g(W^\varphi;\xi)}$ for all $\varphi\in\Tcal$. %We denote by $\Tcal$, the set of all measure preserving transformations on $\interval{-1,1}$ equipped with the Lebesgue measure.
        \item\label{asmp:bdd_var} The random variable $g(\slot{};\xi)$ for $\xi\sim\Dcal$ has uniformly bounded variance over all finite dimensional kernels. That is, there exists $\sigma \geq 0$ such that for all $A\in\cup_{n\in\Natural}\Wcal_n$,
        \[
            \Exp{\xi\sim\Dcal}{\enorm{g(A;\xi) - \phi(A)}^2} \leq \sigma^2.
        \]
        % where $(D_\Wcal \ell)(A;\xi)$ is the Fr\'echet-like derivative (see definition~\ref{def:frechet_like_derivative}) of $\ell(\slot{};\xi)$ at $A$.
    \end{enumerate}
    % \RaghavS{We should expect a scaling in $n$ for assumptions made on $\Sigma_n$, just like we see in equation~\eqref{eq:grad_Rn_lipschitz}. This should be consistent with the Lipschitz constant $\kappa_2$ of $\Sigma\colon\Wcal \to L^\infty(\interval{0,1}^{(2)})$ w.r.t. $\enorm{}$, its uniform bound of $M_\infty$ w.r.t. $\norm{\infty}{}$.}
\end{assumption}
% \begin{assumption}\label{asmp:large_noise}
%     We assume the following about the ``large noise'' for every $n\in\Natural$.
%     The operator $\Sigma_n\colon \Mcal_n \to \R_+^{\squarebrack{n}^{(2)}}$ is $\kappa_2$-Lipschitz with respect to $\normF{}$, i.e., for every $A_n,B_n\in\Mcal_n$, and $G_n\in\Rd{\squarebrack{n}^{(2)}}$,
%     \[
%         \normF{\Sigma_n(A_n)\circ G_n -  \Sigma_n(B_n)\circ G_n} \leq \kappa_2 \normF{A_n-B_n}\normF{G_n}.
%     \]
% \end{assumption}
\begin{assumption}\label{asmp:large_noise}
    We assume the following on the ``large noise'' for every $n\in\Natural$.
    \begin{enumerate}
        \item There exists a function $\Sigma\colon\Wcal \to L^\infty(\interval{0,1}^{(2)})$ such that the diffusion coefficient functions $\round{\Sigma_n}_{n\in\Natural}$ are restrictions of $\Sigma$, i.e., for every $n\in\Natural$, $\Sigma_n =  M_n \circ \Sigma \circ K$ on $\Mcal_n$.
        \item The map $\Sigma\colon\Wcal \to L^\infty(\interval{0,1}^{(2)})$ is $\kappa_2$-Lipschitz in $\enorm{}$ and uniformly bounded in $\norm{\infty}{}$ by some constant $M_\infty\in\R_+$, i.e., for all $U,V\in\Wcal$,
    \begin{align*}
        \enorm{\Sigma(U)-\Sigma(V)} &\leq \kappa_2\enorm{U-V}, \qquad \text{and}\qquad \norm{\infty}{\Sigma(U)} \leq M_\infty.
    \end{align*}
    \end{enumerate}
    % Further, following Assumption~\ref{asmp:large_noise}, 
\end{assumption}

% \SP{(2) is not an assumption but a corollary of (1). It should be removed from the list and remarked outside the Assumption environment.}

\begin{assumption}\label{asmp:phi_cut_lip}
    There exists a constant $\kappa_{\cut}\in\R_+$ such that, for almost every $(x,y) \in \interval{0,1}^{(2)}$, the map $\phi_{x,y} \coloneqq \phi(\slot{})(x,y)$ is $\kappa_{\cut}$-Lipschitz in cut norm $\cutnorm{}$. That is, for every $U,V\in\Wcal$,
    % \begin{equation*}\label{asmp:lipcut}
    \[
        \abs{\phi_{x,y}(U) - \phi_{x,y}(V)} \le \kappa_{\cut} \cutnorm{U-V}.
    \]
    % \end{equation*}
\end{assumption}

% Note that following Assumption~\ref{asmp:large_noise} and Assumption~\ref{asmp:sigma_n}, since $\Mcal_n$ is a compact subset of $\Rd{\squarebrack{n}^{(2)}}$, the operator $\Sigma_n$ is uniformly bounded, i.e., there exists a constant $M_\infty\in\R_+$ such that for every $A_n\in\Mcal_n$,
% \[
%     \max_{(i,j)\in\squarebrack{n}^{(2)}}\abs{\Sigma_n(A_n)(i,j)} \leq M_\infty.
% \]

\subsection{System of reflected diffusions}\label{sec:RBM}
For $n\in\Natural$, consider the domain $\Mcal_n$. Notice that $\Mcal_n$ is a cube, and is closed with respect to the usual topology.
% and $\pdiff \Mcal_n = \Mcal_n \cap \round{\cup_{(i,j)\in\squarebrack{n}^{(2)}} \set{A\in\Mcal_n \given \abs{A_{i,j}}=1 }}$. The outer normal $\eta_n\colon \pdiff \Mcal_n \to \Rd{\squarebrack{n}^{(2)}}$ to every point $A\in\pdiff \Mcal_n$ is defined as
% \[
%     \eta_n(A)_{i,j} \coloneqq \begin{cases}
%         \sign{A_{i,j}} & \text{if } \abs{A_{i,j}} = 1\\
%         0 & \text{if }\abs{A_{i,j}} < 1
%     \end{cases} .
% \]
Consider the SDE:
\begin{align}
    \diff X_{n}(t) &= -n^2\nabla R_n(X_{n}(t))\diff t + \Sigma_n(X_n(t))\circ \diff B_{n}(t) + \diff L_{n}^-(t) - \diff L_{n}^+(t),\label{eq:SDE_with_L}
\end{align}
for $t\in\interval{0,T}$ for some fixed $T\in\R_+$ and starting at $X_n(0)=X_{n,0}\in \Mcal_n$. Here $\Sigma_n$ is a map from $\Mcal_n$ to the set of $n\times n$ symmetric matrices with non-negative entries, $B_{n}$ is a $n\times n$ symmetric matrix valued process containing a set of standard Brownian motions $\round{B_{n,(i,j)}}_{(i,j)\in\squarebrack{n}^{(2)}}$ which are independent up to matrix symmetry, and the processes $L_{n}^-$ and $L_{n}^+$ are local times at the boundary. More precisely, they satisfying the following conditions:
\begin{enumerate}
    \item The processes $X_{n}$, $L_{n}^+$ and $L_n^-$ are adapted processes.
    \item The process $L_{n}^-$ and $L_n^+$ are coordinatewise non decreasing processes a.e.
    % \item Define the measure $\mu_L$ on $\interval{0,T}$ as $\mu_L\round{\interval{0,t}} = \norm{\rm TV}{L_n}(t)$ for $t\in\interval{0,T}$, then
    % \[
    %     \mu_L\round{\set{t\in \interval{0,T} \given X_{n}(t)\in \Mcal_n^\circ} } = 0 ,
    % \]
    % where $\Mcal_n^\circ$ denotes the interior of $\Mcal_n$,
    \item For every $(i,j)\in\squarebrack{n}^{2}$,
    % \[
    %     L_{n}(t) = \int_0^t \eta\round{X_{n}(s)}\mu_L(\diff s) .
    % \]
    \begin{align}
        \begin{split}
            \int_0^\infty \indicator{}{X_{n,(i,j)}(t) > -1} \diff L_{n,(i,j)}^-(t) &= 0, \qquad\text{and}\quad\\
            \int_0^\infty \indicator{}{X_{n,(i,j)}(t) < +1} \diff L_{n,(i,j)}^+(t) &= 0.
        \end{split}
    \end{align}
\end{enumerate}
% For our purposes, we will define the processes $L^+_{n}$ and $L^-_{n}$ as
% \begin{align*}
%         L^+_{n}(t) \coloneqq \round{L_{n}(t)}^{-} ,\qquad L^-_{n}(t) \coloneqq \round{L_{n}(t)}^{+} ,\qquad t\in\R_+ ,
% \end{align*}
% such that the SDE in equation~\eqref{eq:SDE_with_L} is equivalent to
% \begin{align}
%     \diff X_{n}(t) &= -n^2\nabla R_n(X_{n}(t)) + \beta \diff B_{n}(t) - \diff L_{n}^+(t) + \diff L_{n}^-(t) ,\label{eq:master_SDE}
% \end{align}
% with initial condition $X_{n}(0) = X_{n,0}\sim \rho_{n,0}\in \Pcal_2\round{\Mcal_n}$.
We say that $\round{X_{n},L_{n}^+,L_{n}^-}$ solves the Skorokhod problem with respect to the set $\Mcal_n$. Following~\cite[Definition 1.2]{kruk2007explicit}, the strong solution $\round{X_{n},L_{n}^+,L_{n}^-}$ of the Skorokhod problem exists and is unique if $n^2\nabla R_n$ and $\Sigma_n$ are Lipschitz with respect to $\normF{}$ (following Assumption~\ref{asmp:R_ell_phi}, Assumption~\ref{asmp:large_noise} and equation~\eqref{eq:scaling_gradient}).%\RaghavS{If $n^2\nabla R_n$ is bounded/Lipschitz/uniform linear growth?}.\Tripathi{Here $n$ is fixed dimension. So the condition is only on $R_n$. We don't need uniformity in dimension.}

% Since $n^2\nabla R_n$ is Lipschitz continuous on the compact set $\Mcal_n$, it follows from~\cite[Lemma C.1]{javanmard2020analysis} that if $\mathrm{Law}(X_{n,0}) = \rho_{n,0}\in\Pcal_2\round{\Mcal_n}$, then the Skorokhod problem admits a unique solution $X_{n}$ with continuous paths.

\subsubsection{The Lipschitz property of the Skorokhod map}
Let $Y_1$ and $Y_2$ be two real valued stochastic processes. Let $\Lambda_{\interval{-1,1}}$ denote the Skorokhod map that maps the set of c\`adl\`ag functions on $\interval{0,T}$ to itself. If $(X_1 \coloneqq \Lambda_{\interval{-1,1}}(Y_1),L_1^+,L_1^-)$ and $(X_2 \coloneqq \Lambda_{\interval{-1,1}}(Y_2),L_2^+,L_2^-)$ solve the Skorokhod problem with respect to the set $\interval{-1,1}$, then the Skorokhod map $\Lambda_{\interval{-1,1}}$ is $4$-Lipschitz under the uniform metric~\cite[Corollary 1.6]{kruk2007explicit}, i.e.,
\begin{align}
    \sup_{t\in\interval{0,T}}\abs{X_1(t) - X_2(t)} \leq 4\sup_{t\in\interval{0,T}}\abs{Y_1(t) - Y_2(t)}, \qquad\forall\ T\in\R_+.\label{eq:skorokhod_lipschitz}
\end{align}

% Following~\cite[Theorem 2.2]{dupuis1991lipschitz}, if $(X_n^{(1)},L^{(1)+}_n,L^{(1)-}_n)$ and $(X_n^{(2)},L^{(2)+}_n,L^{(2)-}_n)$ solve the Skorokhod problem with respect to the set $\Mcal_n$, and $X^{(i)}_n$ be obtained from the process $Y^{(i)}_n$ defined as
% \[
%     Y^{(i)}_n(t) = X^{(i)}_n(0) - \int_0^t n^2\nabla R_n(X^{(i)}_{n}(s))\diff s + \beta B_{n}(t),\qquad t\in\R_+,
% \]
% for $i\in\set{1,2}$ by applying the Skorokhod map, then
% \begin{align}
%     \sup_{t\in\interval{0,T}}\normF{X_n^{(1)} - X_n^{(2)}} &\leq 2n \sup_{t\in\interval{0,T}}\normF{Y_n^{(1)} - Y_n^{(2)}}.
% \end{align}
%\clearpage
\section{Convergence of Projected Noisy Stochastic Gradient Descent}\label{sec:convergence_nsgd}
\iffalse
Let us consider a random process $B\colon\R_+\to \Rd{\Natural^{(2)}}$, such that
% where $B_{i,j}=B_{j,i}$ for $(i,j)\in\Natural^{(2)}$, and the coordinate processes $\{B_{(i,j)}\}_{i\leq j}$ are independent standard Brownian motions.
$\round{B_{i,j}}_{(i,j)\in\Natural^{(2)}}$ is an infinite doubly indexed array of independent Brownian motions up to matrix symmetry.
For every $n\in\Natural$, define the $\Rd{\squarebrack{n}^{(2)}}$-valued random process $B_n \coloneqq \round{B_{(i,j)}}_{(i,j)\in\squarebrack{n}^{(2)}}$.% up to matrix symmetry.
\fi

% \RaghavS{Since $\Wcal$ is contained in $L^\infty\big(\interval{0,1}^{(2)}\big)$, there exist non-negative constants $M_2\leq M_\infty <\infty$ such that $\enorm{\phi(W)}\leq M_2$ and $\norm{\infty}{\phi(W)}\leq M_\infty$ for all $W\in \Wcal$.}\Tripathi{This follows from assumption 4 below.}\RaghavS{Assumption 4 only provides a Lipschitz constant. The Lipschitz constant of $\phi$ can be $1$ but its norm can be in trillions. Consider $R(x) = x^2/2 + 10^{12}x + 47, \implies \phi(x)=x+10^{12}$ for $x\in\interval{-1,1}$.}\Tripathi{yes, it can be. But for example one can easily get $\enorm{\phi(W)}\leq \kappa_2+\enorm{\phi(W_0)}$ where $W_0$ is $0$ graphon. And similarly, for $M_{\infty}$ one can get $\kappa_{\infty}+\norm{\infty}{\phi(W_0)}$. And, in general, one can not do better than that.}

The goal of this section is to show that for each $n\in\Natural$, the projected noisy SGD iterates, defined in~\eqref{eq:PNSGD}, converges weakly to the strong solution of the SDE~\eqref{eq:RSDE} as $\abs{\tauvec_n}\to 0$. This is done in two steps that we describe below.

Recall the projected noisy SGD iterates defined in Definition~\ref{def:PNSGD}, starting from $W_{n,0}\in\Mcal_n$, rewritten for convenience:%
\begin{align}
    W_{n, k+1} &= P\round{W_{n, k} - n^2\tau_{n,k}\nabla R_n\round{W_{n, k}} - \tau_{n,k}\Delta M_{n,k} + \tau_{n,k}^{1/2}G_{n, k}}, \label{eq:PSGD+Noise}\tag{PNSGD}
\end{align}
for $k\in\R_+$, where $\round{G_{n, k}}_{k\in\Integer_+}$ is any $n\times n$ real symmetric matrix valued martingale difference sequence with each element containing centered and independent entries up to matrix symmetry, as defined in Section~\ref{sec:intro},
% $\Sigma_n\colon \Mcal_n \to \Rd{\squarebrack{n}^{(2)} \times \squarebrack{n}^{(2)}}$ is the diffusion coefficient function such that $\Sigma_n(A_n)$, for any $A_n\in\Mcal_n$, acts on symmetric matrices to output symmetric matrices.
and% $\round{\Delta M_{n, k}}_{k\in\Integer_+}$ is given by
\[
    \Delta M_{n, k} \coloneqq  n^2g_n\round{W_{n,k};\xi_{k+1} }  -n^2\nabla R_n\round{W_{n, k}} , \qquad k\in\Integer_+ .
\]
Observe that $\round{\Delta M_{n, k}}_{k\in\Integer_+}$ is an $n\times n$ symmetric matrix valued martingale difference sequence with respect to the filtration $\round{\Fcal_k}_{k\in\Integer_+}$ where $\mathcal{F}_k\coloneqq \sigma\big(\set{ W_{n, 0}, \xi_{i+1}, G_{n, i}}_{i\in\set{0}\cup[k-1]}\cup\set{\xi_{k+1}} \big)$ for $k\in\Integer_+$.
%\RaghavS{check indexing. Should also include $W_{n,0}$}.
%\RaghavS{$\set{ W_{n, 0}, \xi_{i+1}, G_{n, i} \given 0\leq i\le k-1}\cup\set{\xi_{k+1}}$?}
Without the martingale difference term $\tau_{n,k}\Delta M_{n, k}$, equation~\eqref{eq:PSGD+Noise} reduces to the projected GD iterates with additive noise, $\round{V_{n,k}}_{k\in\Integer_+}$ starting at $V_{n,0}=W_{n,0}$, described in~\eqref{eq:PGD+Noise}, re-written below
\begin{align}
    V_{n,k+1} &= P\round{V_{n,k} - n^2\tau_{n,k}\nabla R_n\round{V_{n,k}} + \tau_{n,k}^{1/2}G_{n,k}}, \qquad k\in\Integer_+. \label{eq:PGD+Noise}\tag{PNGD}
\end{align}
Let $W^{(n)}_{k} \coloneqq K\round{W_{n,k}}$ and $V^{(n)}_{k} \coloneqq K\round{V_{n,k}}$ for all $k\in\Integer_+$, and let $W^{(n)}$ and $V^{(n)}$ be piecewise constant interpolations of $\big(W^{(n)}_k\big)_{k\in\Integer_+}$ and $\big(V^{(n)}_k\big)_{k\in\Integer_+}$ respectively with the step size sequence $\tauvec_n$. Using Gr\"{o}nwall's inequality and an obvious coupling between the processes~\eqref{eq:PSGD+Noise} and~\eqref{eq:PGD+Noise}, we show in Lemma~\ref{lem:ApproximationWithandWithoutRandomDrift} that the two processes are close as $\abs{\tauvec_n}\to 0$.

\begin{lemma}\label{lem:ApproximationWithandWithoutRandomDrift}
    Let $R\colon \Wcal\to \R$ be such that the Fr\'echet-like derivative $\phi=D_{\Wcal}R$ exists. Suppose Assumptions~\ref{asmp:R_ell_phi}, and~\ref{asmp:small_noise} hold.
    Let $n\in\Natural$. Let $W_n$ and $V_n$ be the piecewise constant interpolations (see Definition~\eqref{def:interpolation}) of $\round{W_{n, k}}_{k\in\Integer_+}$ and $\round{V_{n,k}}_{k\in\Integer_+}$ respectively, as defined in~\eqref{eq:PSGD+Noise} and~\eqref{eq:PGD+Noise}, with step size sequence $\tauvec_n \coloneqq \round{\tau_{n,k}}_{k\in\Integer_+}$. Then, there exists a universal constant $C>0$ such that for any $T>0$ we have
    \[
        \E{\sup_{s\in\interval{0,T}} \enorm{W^{(n)}(s)-V^{(n)}(s)}^2} \leq C\sigma^2 T\abs{\tauvec_n}\exp\squarebrack{C\kappa_2^2 T^2}.
    \]
\end{lemma}

\begin{proof}
    Let $W_{n}$ and $V_{n}$ be the piecewise constant interpolations of $\round{W_{n, j}}_{j\in\Integer_+}$ and $\round{V_{n, j}}_{j\in\Integer_+}$ respectively as defined in Definition~\ref{def:interpolation}.
    Define $\Delta\colon \R_+\to\R_+$ as
    \begin{align}
        \Delta(t) &\coloneqq \E{\sup_{s\in\interval{0,t}} \normF{W_n(s)-V_n(s)}^2} , \qquad t\in\R_+ .
    \end{align}
    Let $k\in\Integer_+$ be such that $t\in[t_{n,k},t_{n,k+1})$. Then, using~\cite[Theorem 1]{slominski1994approximation},
    \begin{align}
        \begin{split}
            \Delta(t)  &\leq C\E{ \round{\sum_{j=0}^{k-1}\tau_{n,j}\normF{n^2\nabla R_n\round{W_{n,j}} - n^2\nabla R_n\round{V_{n,j}} }}^2  }\\
            &\qquad + C\E{\sum_{j=0}^{k-1}\tau_{n,j}^2\normF{\Delta M_{n,j}}^2},
        \end{split}
        \label{eq:Dpk_ineq}
    \end{align}
    where $C>0$ is some universal constant. From Assumption~\ref{asmp:R_ell_phi}, since $\phi$ is $\kappa_2$-Lipschitz as a map from $L^2\big(\interval{0,1}^{(2)}\big)$ to $L^2\big(\interval{0,1}^{(2)}\big)$, following equation~\eqref{eq:scaling_gradient} and the fact that $\normF{A_n}^2 = n^2\enorm{K(A_n)}^2$ for all $A_n\in\Mcal_n$, we see that the map $\nabla R_n\colon\Mcal_n\to\Rd{\squarebrack{n}^2}$ satisfies
    \begin{align}
        \normF{n^2\nabla R_n(A_n) - n^2\nabla R_n(B_n)}^2 \leq \kappa_2^2\normF{A_n-B_n}^2 ,\qquad \forall \ A_n,B_n\in\Mcal_n .\label{eq:grad_Rn_lipschitz}
    \end{align}
    Using the Cauchy-Schwarz inequality, and equation~\eqref{eq:grad_Rn_lipschitz}, we first bound the second term in equation~\eqref{eq:Dpk_ineq} as
    \begin{align}
        &\; \E{ \round{\sum_{j=0}^{k-1}\tau_{n,j}\normF{n^2\nabla R_n\round{W_{n,j}} - n^2\nabla R_n\round{V_{n,j}} }}^2  }\nonumber\\
        \leq&\; \E{ \sum_{j=0}^{k-1}\round{\tau_{n,j}^{1/2}}^2\cdot\sum_{j=0}^{k-1}\tau_{n,j}\normF{n^2\nabla R_n\round{W_{n,j}} - n^2\nabla R_n\round{V_{n,j}} }^2  }\nonumber\\
        \leq&\; \kappa_2^2 t \E{\sum_{j=0}^{k-1}\tau_{n,j} \normF{W_{n,j} - V_{n,j}}^2} \leq \kappa_2^2 t \int_{0}^{t}\Delta(s)\diff s,\label{eq:first_variation_bound}
    \end{align}
    where the last inequality follows by observing that if $s\in [t_{n, j}, t_{n, j+1})$ for some $j\in\Integer_+$, then 
    \[
        \E{\normF{W_n(s)-V_n(s)}^2}=\E{\normF{W_{n, j}-V_{n, j}}^2} \leq \Delta(s).
    \]
    Using Assumption~\ref{asmp:small_noise}, first note that
    \begin{align}
        \normF{\Delta M_{n,j}}^2 &= \normF{n^2g_n\round{W_{n,k};\xi_{k+1} }  -n^2\nabla R_n\round{W_{n, k}}}^2\nonumber\\
        &= n^2 \enorm{K\round{n^2g_n\round{W_{n,k};\xi_{k+1} }  -n^2\nabla R_n\round{W_{n, k}}}}^2 \leq n^2 \sigma^2.
    \end{align}
    We use the above to bound the first term in equation~\eqref{eq:Dpk_ineq} as
    \begin{align}
        \E{\sum_{j=0}^{k-1}\tau_{n,j}^2\normF{\Delta M_{n,j}}^2} &\leq n^2\sigma^2 t\abs{\tauvec_n},\label{eq:quadratic_variation_bound}
    \end{align}
    where $\abs{\tauvec_n}$ is defined in Section~\ref{sec:intro} as $\sup_{j\in\Integer_+} \tau_{n,j}$.
    
    Plugging back~\eqref{eq:first_variation_bound} and~\eqref{eq:quadratic_variation_bound} in equation~\eqref{eq:Dpk_ineq} we get
    \begin{align}
        \Delta(t) &\leq Cn^2\sigma^2t\abs{\tauvec_n} + C\kappa_2^2 t \int_{0}^{t}\Delta(s)\diff s,
    \end{align}
    \iffalse
    \begin{align}
        &\; \E{ \round{\sum_{j=0}^{k-1}\tau_{n,j}\normF{n^2\nabla R_n\round{W_{n,j}} - n^2\nabla R_n\round{V_{n,j}} }}^2  }\nonumber\\
        \leq&\; \E{ \sum_{j=0}^{k-1}\round{\sqrt{\tau_{n,j}}}^2\cdot\sum_{j=0}^{k-1}\tau_{n,j}\normF{n^2\nabla R_n\round{W_{n,j}} - n^2\nabla R_n\round{V_{n,j}} }^2  }\nonumber\\
        \leq&\; n^2\kappa_2^2 t \E{\sum_{j=0}^{k-1}\tau_{n,j} \normF{W_{n,j} - V_{n,j}}^2}.\label{eq:first_variation_bound}
    \end{align}
    Using equations~\eqref{eq:first_variation_bound} and~\eqref{eq:quadratic_variation_bound} in equation~\eqref{eq:Dpk_ineq}, yields us
    \begin{align}
        \Delta(t)  &\leq Cn^2\sigma^2 t\abs{\tauvec_n} + Cn^2\kappa_2^2 t \E{\sum_{j=0}^{k-1}\tau_{n,j} \normF{W_{n,j} - V_{n,j}}^2}\nonumber\\
        &\leq Cn^2\sigma^2 t\abs{\tauvec_n} + Cn^2\kappa_2^2 t \int_0^t\Delta(s)\diff s. \label{eq:Dpk_ineq_aux}
    \end{align}
    \fi
    % Using Assumption~\ref{asmp:small_noise} and using the definition of the norm of step-sizes, we further get
    % \begin{align}
    %     \Delta(t)  &\leq Cn^6\sigma^2 \sum_{j=0}^{k-1}\tau_{n,j} + Ck\abs{\tauvec_n}n^2\kappa_2^2\E{ \sum_{j=0}^{k-1}\tau_{n,j}\normF{W_{n,j} - V_{n,j} }^2  }.
    % \end{align}
    % Using the definition of $\Delta\colon\R_+\to\R_+$ and the fact that $t\in[t_{n,k},t_{n,k+1})$ for some $k\in\Integer_+$, we have
    % \begin{align}
    %     \Delta(t)  &\leq Cn^6\sigma^2 \abs{\tauvec_n}t + Cn^2\kappa_2^2 \abs{\tauvec_n} \int_0^t\Delta(s)\diff s. \label{eq:Dpk_ineq_aux}
    % \end{align}
    and applying Gr\"{o}nwall's inequality~\cite{gronwall1919note}, we obtain
    $\Delta(t)\leq Cn^2\sigma^2 t\abs{\tauvec_n}\exp\squarebrack{C\kappa_2^2 t^2}$.
%    The desired conclusion follows by observing that $n^2\enorm{K(A_n)}^2=\normF{A_n}^2$ for any $A_n\in \Mcal_n$.
\end{proof}
% \begin{remark}
%     The proof of Lemma~\ref{lem:ApproximationWithandWithoutRandomDrift} can be modified and improved to account for the case when $\tauvec_n = \round{\tau_{n,j}}_{j\in\Integer_+}$ is chosen to follow an asymptotic behaviour. Typical examples include when an asymptotic upper bound on $\tau_{n,j}$ w.r.t. $j$ is proportional to $j^{-1}$ or $j^{-1/2}$.
% \end{remark}

Our next step is to show that sequence of iterates defined in~\eqref{eq:PGD+Noise} is close to the solution of the SDE~\eqref{eq:RSDE} which we reproduce below
\begin{align}
    \begin{split}
        \diff X_{n}(t) &= -n^2\nabla R_n(X_{n}(t)) + \Sigma_n(X_n(t)) \circ \diff B_{n}(t)\\
        &\qquad\qquad\qquad\qquad\qquad\qquad\qquad - \diff L_{n}^+(t) + \diff L_{n}^-(t) ,\quad t\in\R_+,
    \end{split}
    \label{eq:SDE}\tag{RSDE}
\end{align}
where $B_n$ is an $n\times n$ symmetric matrix valued process whose entries are independent Brownian motions up to matrix symmetry, and $X_{n}(0) = V_{n,0} = W_{n,0} \in \Mcal_n$. The tuple $\round{X_{n},L_{n}^+,L_{n}^-}$ solves the Skorokhod problem with respect to the set $\Mcal_n$ (see Section~\ref{sec:RBM}). 

%Following Assumption~\ref{asmp:R_ell_phi}, since $\Wcal$ is contained in $L^\infty\big(\interval{0,1}^{(2)}\big)$, there exist non-negative constants $M_2\leq M_\infty <\infty$ such that $\enorm{\phi(W)}\leq M_2$ and $\norm{\infty}{\phi(W)}\leq M_\infty$\RaghavS{put this assumption} for all $W\in \Wcal$. 
In Lemma~\ref{lem:FixedDimensionContinumLimit} we compare~\eqref{eq:PGD+Noise} with a discretization of the SDE~\eqref{eq:RSDE}. This is obtained by coupling the discrete noise in~\eqref{eq:PGD+Noise} with the Brownian motion driving the SDE~\eqref{eq:SDE}. Combining these we conclude the convergence of ~\eqref{eq:PSGD+Noise} to the SDE~\eqref{eq:SDE} as $\abs{\tauvec_n}\to 0$.

\begin{lemma}
    \label{lem:FixedDimensionContinumLimit}
    % \RaghavS{take care of the diagonal}
    Let $n\in \Natural$. Let $B_n$ be an $n\times n$ symmetric matrix valued process whose coordinates are i.i.d. Brownian motion (up to matrix symmetry) defined on some probability space. Let $X_{n}$ be the strong solution of SDE~\eqref{eq:SDE} with initial condition $X_{n}(0)=V_{n,0}$ (see~\eqref{eq:PGD+Noise}).
    Then, there exists a c\`adl\`ag process $\widetilde{V}_{n}$ on $\Mcal_n$, defined on the same probability space as $B_n$, such that it has the same law as $V_n$, the piecewise constant interpolation (see Definition~\ref{def:interpolation}) of $\big(V_{n,k}\big)_{k\in\Integer_+}$ obtained from~\eqref{eq:PGD+Noise}. Moreover, for any $T\in\R_+$,
    \[
        \lim_{\abs{\tauvec_n} \to 0} \E{\sup_{s\in\interval{0,T}} \enorm{K(X_{n}(s)) - K\round{\widetilde{V}_{n}(s)}}^2 } = 0.
    \]
\end{lemma}

\begin{proof}
    Let $B_n$ be as given in the assumption and let $X_n$ be the strong solution of the SDE~\eqref{eq:SDE}. %\RaghavS{need to define $\widetilde{G}_{n,k}$ using Proposition~\ref{prop:nonasymp_donsker}.}
    % Set \RaghavS{state dependence and coupling}
    % \begin{equation}\label{eqn:couplingbtwDiscreteandContNoise}
    %     G_{n, k}\coloneqq \left(B_n(t_{n,k+1})-B_n(t_k)\right)/\sqrt{\tau_{n, k}},
    % \end{equation}
    % and let $(\widetilde{V}_{n, k})_{k\in\Integer_+}$ be as defined equation~\eqref{eqn:AuxGD+Noise}.
    Since the discrete noise in~\eqref{eq:PGD+Noise} is Gaussian (see Assumption~\ref{asmp:large_noise}), there is an obvious way to couple it with the Brownian motion driving the SDE in~\eqref{eq:SDE}.
    % Let $B_{n}(t)$ be an $n\times n$ symmetric matrix valued process whose entries are independent Brownian motions up to matrix symmetry defined on some probability space.
    Given $B_n$ and the step size sequence $\tauvec_n = \round{\tau_{n, k}>0}_{k\in\Integer_+}$, define the discrete time $n\times n$ symmetric matrix valued martingale difference sequence $\big(\widetilde{Z}_{n,k}\big)_{k\in\Integer_+}$ as
    \begin{equation}\label{eqn:GaussianCoupling}
        \widetilde{Z}_{n, k} \coloneqq \tau_{n,k}^{-1/2}\round{B_{n}(t_{n,k+1})-B_{n}(t_{n,k})} ,\qquad k\in\Integer_+ .
    \end{equation}
    % if $\beta>0$, and identically $0$ if $\beta = 0$.
    Note that the entries in $\widetilde{Z}_{n, k}$ are distributed as $N(0, 1)$ up to matrix symmetry for every $k\in\Integer_+$. Starting from $\widetilde{V}_{n, 0}=V_{n,0}$, we now define an auxiliary process $\big(\widetilde{V}_{n, k}\big)_{k\in\Integer_+}$, on the same probability space as $B_n$, iteratively as
    \begin{align}
    \label{eqn:AuxGD+Noise}
        \widetilde{V}_{n,k+1} &= P\round{\widetilde{V}_{n,k} - n^2\tau_{n,k}\nabla R_n\round{\widetilde{V}_{n,k}} + \tau_{n,k}^{1/2}\Sigma_n\round{\widetilde{V}_{n,k}}\circ\widetilde{Z}_{n,k}},\qquad k\in\Integer_+,%\label{eq:PGD+Noise}\tag{PNGD}
    \end{align}
    Following Assumption~\ref{asmp:large_noise}, $\widetilde{V}_{n, k}$ has the same law as $V_{n, k}$ for each $k\in\Integer_+$.
    % Let $V^{(n)}_k \coloneqq K(V_{n,k})$, and $\widetilde{V}^{(n)}_k \coloneqq K\big(\widetilde{V}_{n,k}\big)$ for every $k\in\Integer_+$; and $V^{(n)}\colon\R_+ \to \Wcal_n$ and $\widetilde{V}^{(n)}\colon\R_+ \to \Wcal_n$ be the piecewise constant interpolation (see Definition~\ref{def:interpolation}) of $\big(V^{(n)}_k\big)_{k\in\Integer_+}$ and $\big(\widetilde{V}^{(n)}_k\big)_{k\in\Integer_+}$ respectively.
    Let $\widetilde{V}_n\colon\R_+\to \Mcal_n$ be piecewise constant interpolation of $\round{\widetilde{V}_{n, k}}_{k\in\Integer_+}$.
    The particular choice of $\big(\widetilde{Z}_{n, k}\big)_{k\in\Integer_+}$ in equation~\eqref{eqn:GaussianCoupling} allows us to couple $\widetilde{V}_n$ with the strong solution of the SDE~\eqref{eq:RSDE}.
    Let $\widetilde{G}_{n,j} \coloneqq \Sigma_n\round{\widetilde{V}_{n,j}}\circ\widetilde{Z}_{n,j}$ for all $j\in\Integer_+$. The curve $\widetilde{V}_n$ can be written as
    \begin{align}
         \widetilde{V}_n(t) = \widetilde{V}_{n,0} - \sum_{j=0}^{k-1} n^2\tau_{n,j}\nabla R_n(\widetilde{V}_{n,j}) + \sum_{j=0}^{k-1} \tau_{n,j}^{1/2}\widetilde{G}_{n,j} + \sum_{j=0}^{k-1}\tau_{n,j}\round{L_{n,j}^- - L_{n,j}^+},
    \end{align}
    for $t\in[t_{n,k},t_{n,k+1})$. Here $\round{L_{n, j}^{\pm}}_{j\in\Integer_+}$ is chosen so that the piecewise constant interpolation (see Definition~\ref{def:interpolation}) of $\round{V_{n,k},L_{n,k}^-,L_{n,k}^+}_{k\in\Integer_+}$ solves the Skorokhod problem with respect to $\Mcal_n$ (see Section~\ref{sec:RBM}). 
    
    Also consider three auxiliary processes $Y_n$, $\overline{Y}_n$, and $\widehat{Y}_n$ taking values over $n\times n$ real symmetric matrices, defined as
    \begin{align}
        Y_n(t) &\coloneqq X_n(0) - \int_0^t n^2 \nabla R_n(X_n(s))\diff s + \int_0^t\Sigma_n(X_n(s))\circ \diff B_n(s),\\
        \widehat{Y}_n(t) &\coloneqq X_n(0) - \int_0^t n^2 \nabla R_n\round{\widetilde{V}_n(s)}\diff s + \int_0^t\Sigma_n\round{\widetilde{V}_n(s)}\circ \diff B_n(s),\\
        \overline{Y}_n(t) &\coloneqq X_n(0) - \sum_{j=0}^{k-1}n^2\tau_{n,j}\nabla R_n(\widetilde{V}_{n,j}) + \sum_{j=0}^{k-1}\tau_{n,j}^{1/2} \widetilde{G}_{n,j},
    \end{align}
    for every $k\in\Integer_+$ and all $t\in[t_{n,k},t_{n,k+1})$. Observe that the curves $X_n$ and $\widetilde{V}_n$ can be obtained by applying the Skorokhod map to the curves $Y_n$ and $\overline{Y}_n$ pointwise respectively. Let $\widehat{V}_n\colon\R_+\to\Mcal_n$ be obtained from $\widehat{Y}_n$ by applying the Skorokhod map. First observe that using the Lipschitzness of the Skorokhod map, $\phi$ and $\Sigma_n$ (see Assumption~\ref{asmp:R_ell_phi}, Assumption~\ref{asmp:large_noise}, Section~\ref{sec:RBM} and equation~\eqref{eq:grad_Rn_lipschitz}), we obtain
    \begin{align}
        & \E{\sup_{t\in\interval{0,T}} \normF{\widehat{V}_n(t)  -  X_n(t)}^2} \leq 16\E{\sup_{t\in\interval{0,T}} \normF{\widehat{Y}_{n}(t) - Y_{n}(t) }^2 } \nonumber\\
        &\leq\; 16 \E{ \sup_{t\in\interval{0,T}} \normF{\int_0^t n^2 \nabla R_n(X_n(s)) - n^2\nabla R_n\round{\widetilde{V}_n(s)}\diff s}^2 }\nonumber\\
        &\quad\quad+16\E{\sup_{t\in\interval{0, T}} \normF{\int_{0}^{t} \round{\Sigma_n(X_n(s))-\Sigma_n\round{\widetilde{V}_n(s)}}\circ \diff B_n(s)}^2}\nonumber\\
        &\leq  16 \kappa_2^2 \E{ \int_0^T\normF{ X_n(s) - \widetilde{V}_n(s)}^2\diff s }\nonumber\\
        &\qquad +64\E{\int_{0}^{T}\normF{\Sigma_n(X_n(s))-\Sigma_n\round{\widetilde{V}_n(s)}}^2\diff s}\nonumber\\
        &\leq 80 \kappa_2^2 \int_0^T \E{ \sup_{s\in\interval{0,t}}\normF{ X_n(s) - \widetilde{V}_n(s)}^2 }\diff s,\label{eq:triangle_2}
    \end{align}
    where the second last inequality follows from Doob's maximal inequality~\cite[page 14, Theorem 3.8.iv]{KS91} and the fact that for all $A_n\in\Mcal_n$, $\normF{A_n}^2 = n^2\enorm{K(A_n)}^2$. For any $t\in[0,T]$, define $k_t \coloneqq \argmin_{j\in\Integer_+}\setinline{t\geq t_{n,j}}$. Using the Lipschitzness of Skorokhod map (see Section~\ref{sec:RBM}) we obtain %\RaghavS{Lipschitz constant of state dependence}
    \begin{align}
        \mathop{\mathbb{E}}\Biggl[\sup_{s\in\interval{0,T}} &\normF{\widetilde{V}_n(t) - \widehat{V}_n(t)}^2\Biggr] \leq 16\E{\sup_{t\in\interval{0,T}} \normF{\overline{Y}_{n}(t) - \widehat{Y}_{n}(t) }^2 }\nonumber\\
        % \E{\sup_{s\in\interval{0,T}} \normF{\widetilde{V}_n(t) - \widehat{V}_n(t)}^2} &\leq 16\E{\sup_{t\in\interval{0,T}} \normF{\overline{Y}_{n}(t) - \widehat{Y}_{n}(t) }^2 }\nonumber\\
        % &\nonumber\\
        &\leq 32 \E{ \sup_{t\in\interval{0,T}} \normF{\int_0^t n^2\nabla R_n\round{\widetilde{V}_n(s)}\diff s - \sum_{j=0}^{k_t-1}n^2\tau_{n,j}\nabla R_n\round{\widetilde{V}_{n,j}}  }^2 }\nonumber\\
        &+ 32 \E{ \sup_{t\in\interval{0,T}} \normF{\sum_{j=0}^{k_t-1}\tau_{n,j}^{1/2}\Sigma_n\round{\widetilde{V}_{n, j}}\circ \widetilde{Z}_{n,j} - \int_0^t\Sigma_n\round{\widetilde{V}_n(s)}\circ \diff B_n(s)}^2 }, \label{eq:triangle_1}
    \end{align}
    where the last inequality follows from Assumption~\ref{asmp:large_noise}.
    
    We now bound the first term from the above inequality~\eqref{eq:triangle_1}. To this end observe that
    \begin{align}
        &\; \E{ \sup_{t\in\interval{0,T}} \normF{\int_0^t n^2\nabla R_n\round{\widetilde{V}_n(s)}\diff s - \sum_{j=0}^{k_t-1}n^2\tau_{n,j}\nabla R_n\round{\widetilde{V}_{n,j}}  }^2 } \nonumber\\
        =&\; \E{ \sup_{t\in\interval{0,T}} \normF{n^2(t-t_{n,k_t})\nabla R_n\round{\widetilde{V}_{n,k}}}^2 }\leq \abs{\tauvec_n}^2 \E{ \sup_{t\in\interval{0,T}} \normF{n^2 \nabla R_n\round{\widetilde{V}_{n,k}}}^2 } \nonumber\\
        =&\; n^2\abs{\tauvec_n}^2 \E{ \sup_{t\in\interval{0,T}} \enorm{\phi\round{\widetilde{V}^{(n)}(t)}}^2 } \leq n^2\abs{\tauvec_n}^2 M_2^2,\label{eq:triangle_1a}
    \end{align}
    for some constant $M_2\in\R_+$ by Assumption~\ref{asmp:R_ell_phi}.
    
    We now bound the second term in the inequality~\eqref{eq:triangle_1}. Using the coupling defined in~\eqref{eqn:GaussianCoupling} and noting that $\widetilde{V}(s)=\widetilde{V}_{n, j}$ for $s\in [t_{n, j}, t_{n,j+1})$ (see Definition~\ref{def:interpolation}), we obtain that
    \begin{align}
        \begin{split}
            &\E{\sup_{t\in\interval{0,T}} \normF{\sum_{j=0}^{k_t-1}\tau_{n,j}^{1/2}\Sigma_n\round{\widetilde{V}_{n, j}}\circ \widetilde{Z}_{n,j} - \int_0^t\Sigma_n\round{\widetilde{V}_n(s)}\circ \diff B_n(s)}^2}\\
            &=\E{\sup_{t\in [0, T]}\normF{\Sigma_n\round{\widetilde{V}_{n, k_t}}\circ \round{B_n(t)-B_n(t_{n, k_t})}}^2} \leq M_\infty^2n^2C_{1,T}\abs{\tauvec_n}\log\inv{\abs{\tauvec_n}},
        \end{split}\label{eq:donsker}
    \end{align}
    where the last inequality follows from Assumption~\ref{asmp:large_noise} and~\cite[Lemma A.4]{slominski2001euler} for $C_{1,T}\in\R_+$.

    Now define $\Delta\colon\R_+\to\R_+$ as
    \[
        \Delta(t) \coloneqq \E{ \sup_{s\in\interval{0,t}}\normF{ X_n(s) - \widetilde{V}_n(s)}^2 }, \qquad t\in\R_+.
    \]
    Using the triangle inequality by combining equations~\eqref{eq:triangle_2},~\eqref{eq:triangle_1},~\eqref{eq:triangle_1a} and~\eqref{eq:donsker},
    % Proposition~\ref{prop:nonasymp_donsker} (see below), %~\eqref{eq:triangle_1b},~\eqref{eq:triangle_1c}
    we get
    \begin{align}
        \Delta(T) &\leq 32n^2\abs{\tauvec_n}^2 M_2^2 + 32n^2M_\infty^2C_{1,T}\abs{\tauvec_n}\log\inv{\abs{\tauvec_n}} + 80 \kappa_2^2\int_0^T \Delta(t)\diff t.
    \end{align}
    Applying Gr\"{o}nwall's inequality~\cite{gronwall1919note}, we get
    \begin{align}
        \Delta(T) &\leq 32n^2\round{\abs{\tauvec_n}^2 M_2^2 + M_\infty^2C_{1,T}\abs{\tauvec_n}\log\inv{\abs{\tauvec_n}}}\exp\squarebrack{80 \kappa_2^2 T}.
    \end{align}
    Taking limit as $\abs{\tauvec_n}\to 0$ on the above bound, completes the proof.
\end{proof}

We combine Lemma~\ref{lem:ApproximationWithandWithoutRandomDrift} and~\ref{lem:FixedDimensionContinumLimit} to conclude the proof of Theorem~\ref{thm:SGD_to_SDEn}. Moreover, we also obtain the following non-asymptotic error rate
 \[
        \E{\sup_{s\in\interval{0,T}} \enorm{W^{(n)}(s)-K(X_{n})(s)}^2} \leq Cn^2(M+\sigma^2T)\abs{\tauvec_n}\log\inv{\abs{\tauvec_n}}\exp\squarebrack{C \kappa_2^2 T}
    \]
for some constants $C, M<\infty$.

\subsection{Convergence of Projected Stochastic Gradient Descent}\label{sec:without_noise}
In the absence of ``large noise'' (i.e., when $\Sigma_n\equiv 0$), the SDE~\eqref{eq:SDE} reduces to the SDE
\begin{align}
    \diff X_n(t) = -n^2\nabla R_n(X_n(t))\diff t + \diff L^{-}_n(t) - \diff L^{+}_n(t),\qquad X_n(0) = W_{n,0},\label{eq:RSDE_sigma0}
\end{align}
As we describe in Section~\ref{sec:beta_0}, it is show in~\cite[Theorem 4.4, Theorem 4.14]{oh2021gradient} that if the solution of
\begin{equation}\label{eqn:GF_indicator}
    \diff X_n(t) = -n^2\nabla R_n(X_n(t))\circ\indicator{G_n(X_n(t))}{}\diff t,
\end{equation}
exists, where $G_n(A)$ is the subset of $\squarebrack{n}^{2}$ defined as
\begin{align}\label{eq:G_n}
    \begin{split}
        G_n(A) &\coloneqq \set{(i,j)\in\squarebrack{n}^{2} \given \abs{A(i,j)}<1}\\
        &\qquad\cup \set{(i,j)\in\squarebrack{n}^{2} \given A(i,j)=1, \pdiff_{i,j} R_n(A)>0}\\
        &\qquad\qquad \cup \set{(i,j)\in\squarebrack{n}^{2} \given A(i,j)=-1, \pdiff_{i,j} R_n(A)<0},
    \end{split}
\end{align}
for all $A\in\Mcal_n$, then the solution $X_n$ is a gradient flow on $\Mcal_n$ in a suitable sense. In this section, we will argue that the solutions $X_n$ of equation~\eqref{eq:RSDE_sigma0} and~\eqref{eqn:GF_indicator} are equal. To this end, we define processes $L_n^{\pm}$ as
\begin{align}\label{eq:L^pm_def}
    \begin{split}
        L^+_n(t) &\coloneqq -\int_0^t n^2\nabla R_n(X_n(s))\circ\indicator{\set{X_n(s)=+1,\nabla R_n(X_n(s))<0}}{}\diff s,\\
        L^-_n(t) &\coloneqq +\int_0^t n^2\nabla R_n(X_n(s))\circ\indicator{\set{X_n(s)=-1,\nabla R_n(X_n(s))>0}}{}\diff s,
    \end{split}
\end{align}
for $t\in\R_+$, and equation~\eqref{eqn:GF_indicator} can be rewritten as
\begin{align}
    \diff X_n(t) &= -n^2\nabla R_n(X_n(t))\circ\indicator{G_n(X_n(t))}{} + \diff L^-_n(t) - \diff L^+_n(t),
\end{align}
and the processes $L_n^+$ and $L_n^-$ satisfy the following conditions:
\begin{enumerate}
    \item The processes $X_{n}$, $L_{n}^+$ and $L_n^-$ are adapted processes.
    \item The processes $L_{n}^-$ and $L_n^+$ are non-decreasing processes.
    \item For every $(i,j)\in\squarebrack{n}^{2}$,
    \[
        \begin{split}
            \int_0^\infty \indicator{}{X_{n,(i,j)}(t) > -1} \diff L_{n,(i,j)}^-(t) &= 0, \quad\text{and}\\
            \int_0^\infty \indicator{}{X_{n,(i,j)}(t) < +1} \diff L_{n,(i,j)}^+(t) &= 0.
        \end{split}
    \]
\end{enumerate}
Following Section~\ref{sec:RBM}, these conditions ensure that the processes $L_n^+$ and $L_n^-$ are unique and $(X_n,L_n^+,L_n^-)$ solves the Skorokhod problem with respect to the set $\Mcal_n$. This proves Theorem~\ref{thm:zeroTempLimit}.%\clearpage
\section{Convergence of the finite dimensional SDEs}\label{sec:convergence_sde}
\subsection{The limit at infinity: infinite exchangeable array of diffusions}\label{subsec:exch_array} 

Let $\mathcal{E}$ be a standard Borel space. The sets $\squarebrack{n}^{(2)}$ and $\Natural^{(2)}$ will refer to the set of natural number pairs $(i,j)$ in $\Natural^2$ and $\squarebrack{n}^2$ respectively, such that $i<j$. Recall that an $\mathcal{E}$-valued exchangeable (symmetric) array refers to a doubly indexed collection of random elements $\left(\zeta_{i,j}\coloneqq \zeta_{\{i,j\}} \in \mathcal{E} \right)_{(i,j)\in\Natural^{(2)}} \eqqcolon\zeta$ that remain invariant in law under finite permutations of natural numbers $\N$. Two special cases of $\mathcal{E}$ that are important to us are $\mathcal{E}=[-1,1]$ and $\mathcal{E}=C[0, \infty)$ with the usual Borel topology. The Aldous-Hoover representation theorem~\cite{AldExchange, HooverExchange, hoover1982row} says that given any exchangeable array as above, there exists a measurable function $f\colon [0,1]\times \interval{0,1}^{(2)}\times \interval{0,1} \rightarrow \mathcal{E}$ such that $\zeta_{i,j}=f\left( U, U_i, U_j, U_{i,j} \right) = f\left( U, U_j, U_i, U_{i,j} \right)$ for $(i,j)\in\Natural^{(2)}$, where $U$, $\round{U_i}_{i \in \N}$, $\round{U_{i,j}=U_{\{i,j\}}}_{(i,j)\in \N^{(2)}}$ are i.i.d. $\mathrm{Uni}[0,1]$ random variables. The function $f$ is typically not unique. Following~\cite{austin2008exchangeable}, we say that $\zeta$ is directed by $f$.

The relationship between exchangeable arrays and graphons follows from the Aldous-Hoover representation~\cite{diaconis2007graph}. Assume that $\zeta_{i,j}$s are real valued and take values in the closed interval $[-1,1]$. An infinite exchangeable array gives rise to a \textit{random} graphon reminiscent of the de Finetti representation theorem for exchangeable sequences of random variables. Although we believe that the following result is well-known, we could not find a statement to this effect in the literature. However, it inspires our later constructions. 

\begin{lemma}\label{lem:random_graphon}
    Let $\zeta\in\interval{-1,1}^{\Natural^{(2)}}$ be an infinite exchangeable array directed by $f$. Consider the family of symmetric kernels $\left( g_u,\; u \in \interval{0,1}\right)$ defined by 
    \begin{equation}\label{eq:graphonexch}
        g_u(x, y)\coloneqq \E{f(u, x, y, V)}, \qquad u\in\interval{0,1}, \quad (x,y)\in\interval{0,1}^{(2)},
    \end{equation}
    where the above expectation is with respect to a $\mathrm{Uni}[0,1]$ random variable $V$. Then, for $u \in \interval{0,1}$, given $\setinline{U=u}$,
    \begin{equation}\label{eq:excha_graphon_conv}
        \lim_{n\to\infty}\delta_{\cut}\left( K\round{\round{\zeta_{i,j}=f(u, U_i, U_j, U_{i,j})}_{(i,j)\in\squarebrack{n}^{(2)}}}, [g_u] \right) = 0,\qquad \text{a.s.}
    \end{equation}
    %\RaghavS{Should $[g_u]$ above be $[g_U]$ instead? $[g_U]$ is a r.v. but $[g_u]$ is not.}
\end{lemma}

%Then, for $u\in [0,1]$, the function $g_u\colon\interval{0,1}^2\to\interval{-1,1}$ defined as $g_u(x,y)\coloneqq g(u,x,y)$ for $(x,y)\in\interval{0,1}^2$, is a symmetric kernel. When $U$ is distributed as $\mathrm{Uni}[0,1]$, then $[g_U]$ is a random graphon.
    
\begin{proof}
    Fix $(i, j)\in \N^{(2)}$ and note that $f(U, U_i, U_j, U_{i,j})=f(U, U_j, U_i, U_{i,j})$ since $\zeta_{i,j}=\zeta_{j,i}$ and $U_{i,j}=U_{j,i}$. Therefore, $\E{f(U, U_i, U_j, U_{i,j})\given U, U_i, U_j} = \E{f(U, U_j, U_i, U_{i,j})\given U, U_i, U_j}$, and,
    \begin{align*}
        g_u(x, y)&=\E{f(U, U_i, U_j, U_{i,j})\given U=u, U_i=x, U_j=y}\\
        &=\E{f(U, U_j, U_i, U_{i,j})\given U=u, U_i=x, U_j=y}=g_u(y, x) ,
    \end{align*}
    for a.e. $(x,y)\in\interval{0,1}^{(2)}$. Since the maps $f$, $\mathbb{E}$ and $\squarebrack{\slot{}}$ are all measurable, their composition is also measurable. Because $U$ is a random variable, $\squarebrack{g_U}$ is also a random variable  obtained as a composition of measurable maps.
    
    To see \eqref{eq:excha_graphon_conv}, start with the Aldous-Hoover representation $\zeta_{i,j}=f(U, U_i, U_j, U_{i,j})$ for every $(i,j)\in\Natural^{(2)}$. Condition on $\{U=u\}$ throughout for $u\in\interval{0,1}$. For any finite simple graph $F$, with $k$ vertices,
    \begin{equation}\label{eq:whatistf}
    \begin{split}
        h_F\left( K\round{\round{\zeta_{i,j}}_{(i,j)\in\squarebrack{n}^{(2)}}} \right) &= \frac{1}{n^{\downarrow k}} \sum_{i_1, i_2, \ldots, i_k} \prod_{\set{j, l}\in E(F)} \zeta_{i_ji_l}\\
        &= \frac{1}{n^{\downarrow k}}\sum_{i_1, i_2, \ldots, i_k} \prod_{\set{j, l}\in E(F)} f(u, U_{i_j}, U_{i_l}, U_{i_j,i_l}) ,
    \end{split}
    \end{equation}
    where the summation runs over the $n^{\downarrow k} \coloneqq n!/(n-k)!$ many injections from $\squarebrack{k}$ to $\squarebrack{n}$, and $h_F\colon\Wcal\to\R$ is the homomorphism density function of $F$~\cite[Section 7.2]{lovasz2012large}. Notice that 
    \[
    \begin{split}
        \E{h_F\left( K\round{\round{\zeta_{i,j}}_{(i,j)\in\squarebrack{n}^{(2)}}} \right)} &= \int_{[0,1]^k} \prod_{\set{j,l} \in E(F)} \E{f(u, u_j, u_l, V)} \diff u_1 \cdots \diff u_k  = h_F(g_u) ,
    \end{split}
    \]
    where $g_u$ is defined in \eqref{eq:graphonexch}. Hence, the lemma will be true if we show that the strong law of large numbers holds. That the weak law of large numbers holds, can be seen by a variance computation. That the convergence is a.e. follows from Borel-Cantelli lemma~\cite[Theorem 4.18]{kallenberg21}. We skip the standard argument. The conclusion holds following the inverse counting lemma~\cite[Lemma 10.32]{lovasz2012large}.
\end{proof}

\begin{remark}\label{rem:random_graphon}
    As a corollary of the previous result, although the function $f$ is not unique in the Aldous-Hoover representation, the law of the random graphon $[g_U]$ is indeed unique. 
\end{remark}

%In the discussion below we will abuse notation and use $\phi$ as a function on both graphons as well as exchangeable arrays. 

%We define a \textit{projection} operator $\Gamma\colon\rr^{\Natural^{(2)}}\to\Wcal$ that takes infinite exchangeable arrays to its kernel. The operator $\Gamma$, given by an expectation, is linear and monotone in the following sense. For linearity, suppose on the same probability space we have two infinite exchangeable arrays $\xi$ and $\zeta$ directed by measurable functions $f$ and $g$ respectively, then, for any $\alpha \in \rr$, $\xi+ \alpha \zeta$ is another exchangeable array that is directed by the function $f+\alpha g$. It follows from~\eqref{eq:graphonexch} that $\Gamma(\xi+\alpha \zeta)= \Gamma(\xi) + \alpha \Gamma(\zeta)$. 
%In fact $[\Gamma(\xi)]$ and $[\Gamma(\zeta)]$ are random coupled graphons. 
%As for monotonicity, suppose, as before, one has two exchangeable arrays $\xi$ and $\zeta$ such that $\xi_{i,j} \le \zeta_{i,j}$ for every $(i,j)\in\Natural^{(2)}$. Then, $\Gamma(\xi)(x,y)\le \Gamma(\zeta)(x,y)$ for all $(x,y)\in [0,1]^2$. In particular, if $\abs{\xi_{i,j}} \le \zeta_{i,j}$, for all $(i,j)\in\Natural^{(2)}$, then $\abs{\Gamma(\xi)(x,y)} \le \Gamma(\zeta)(x,y)$ for a.e. $(x,y)\in\interval{0,1}^2$.

% For the remaining of this section $\Natural^{(2)}$ will refer to the set of all natural number pairs $(i,j)\in\Natural^2$ such that $i< j$.
Consider $\left(C[0, \infty)\right)^{\N^{(2)}}$ with the natural filtration generated by the coordinate process. Enlarge the filtration by expanding the probability space to accommodate the countably many i.i.d. $\mathrm{Uni}[0,1]$ random variables $\round{U_i}_{i \in \N}$ and including the sigma algebra generated by them in the sigma algebra at time zero. Endow this filtered probability space with a probability measure $P^\infty$ that denote the joint law of $(U_i)_{i \in \Natural}$ and that of an independent array of countably many independent Brownian motions (BMs) $\set{B_{i,j}=B_{\{i,j\}}}_{(i,j)\in\Natural^{(2)}}$. Finally we turn the natural filtration to one that is right-continuous and complete, thereby satisfying the so-called usual conditions and denote it by $\Fcal = \left(\mathcal{F}_t\right)_{t \in \rr_+}$. All our processes will be adapted to this filtration associated with this set-up. Note that all uniform random variables $\round{U_i}_{i\in\Natural}$ are measurable with respect to $\mathcal{F}_0$.

Let $\phi$ and $\Sigma$ be two functions from $\Wcal$ to $L^\infty\big(\interval{0,1}^{(2)}\big)$ that are both $\kappa_2$-Lipschitz functions on kernels with respect to the the $L^2$ norm $\enorm{}$ (Assumption~\ref{asmp:R_ell_phi} and~\ref{asmp:large_noise}). Our goal is to construct, on the above probability space with filtration $\round{\Fcal_t}_{t\in\R_+}$, an exchangeable array of reflected diffusions satisfying
\begin{align}\label{eq:infinite_SDE_1}
    \diff X_{i,j}(t) &= -\phi\round{\Gamma(t)}(U_i, U_j)\diff t + \Sigma\left( \Gamma(t)\right)(U_i, U_j)\diff B_{i,j}(t) + \diff L^-_{i,j}(t) - \diff L^+_{i,j}(t) ,
\end{align}
with the initial condition $X_{i, j}(0)=W_0(U_i, U_j)$ for all $(i,j)\in\Natural^{(2)}$, for some $W_0\in \Wcal$ and 
\[\Gamma(t)(x, y)=\E{X_{1, 2}(t)\given U_1=x, U_2=y}.\]

We construct a diffusion with more general drift as follows. Let $b\colon \interval{-1,1}\times \Wcal \to L^\infty\big(\interval{0,1}^{(2)}\big)$ be satisfy Assumption~\ref{asmp:b_lip_l2}. Given $W_0 \in \Wcal$, let $X \coloneqq \round{X_{i,j}\coloneqq X_{\{i,j\}}}_{(i,j) \in \N^{(2)}}$, be the solution of the following system of SDE taking values in $[-1,1]^{\N^{(2)}}$ with the initial condition $\round{X_{i,j}(0)=W_0(U_i, U_j)}_{(i,j)\in\Natural^{(2)}}$, and satisfying 
\begin{align}\label{eq:infinite_SDE}
    \begin{split}
        \diff X_{i,j}(t) &= b\round{X_{i,j}(t),\Gamma(t)}(U_i, U_j)\diff t + \Sigma\left( \Gamma(t)\right)(U_i, U_j)\diff B_{i,j}(t)\\
        &\qquad\qquad\qquad\qquad\qquad\qquad\qquad\qquad + \diff L^-_{i,j}(t) - \diff L^+_{i,j}(t) ,
    \end{split}
\end{align}

% \SP{But what is $b$? Hasn't been defined yet. What is its relationship with $\phi$? Want to say something like $b: \interval{-1,1} \times \mathcal{W} \rightarrow \mathcal{W}$ is a map satisfying the following Lipschitz property. And $\phi(W)(x,y)=b( W(x,y), W )(x,y)$.}

%\RaghavS{Need assumption: $\exists\ \kappa_\cut$ s.t. for all $(x,y)\in\interval{0,1}^{(2)}$,
%\begin{align*}
    % \abs{b(X,W)(x,y) - b(Y,W)(x,y)} &\leq L \abs{X-Y}, \qquad \forall\ X,Y\in[-1,1], \ W\in\Wcal,\\
    % \abs{b(W_2(x,y),W_1)(x,y) - b(W_2(x,y),W_2)(x,y)} &\leq \kappa_2 \abs{W_1(x,y)-W_2(x,y)}, \qquad \forall\ W_1,W_2\in\Wcal,\\
 %   \abs{b(Z,U)(x,y) - b(Z,V)(x,y)} &\leq \kappa_\cut \cutnorm{U-V} \qquad \forall\ U,V\in\Wcal, \ Z\in\interval{-1,1}.
%\end{align*}
%}
%\SP{The Lipschitz property w.r.t cut metric assumption is needed only for convergence and not existence. So, for the next section. In the section the Lip assumption should be just for $L^2$. At least, you say this in a remark.}

for $(i,j)\in\Natural^{(2)}$ and $t\in \R_+$. The processes $L^-_{i,j}$ and $L^+_{i,j}$ are such that $(X_{i,j},L^+_{i,j},L^-_{i,j})$ solves the Skorokhod problem with respect to $\interval{-1,1}$ (see Section~\ref{sec:RBM}), i.e., $L^-_{i,j}$ and $L^+_{i,j}$ are non-decreasing processes that keep the processes $X_{i,j}$s in the closed interval $\interval{-1,1}$. The kernel valued process $\Gamma \colon \R_+ \to \Wcal$ is adapted to the sigma algebra generated by the uniform random variables $(U_i)_{i \in \Natural}$, and the independent BMs $\round{B_{i,j}}_{(i,j)\in\Natural^{(2)}}$,
% \RaghavS{adapted w.r.t. $\Fcal$?}\SP{By construction of the filtration.}
and given by
\begin{equation}\label{eq:whatisgammat}
    \Gamma(t)(x,y) \coloneqq  \E{ X_{1,2}(t) \given U_1=x, U_2=y},
\end{equation}
for $(x,y)\in\interval{0,1}^{(2)}$ and $t\in\R_+$. Note that if the solution $X$ of the system of SDEs~\eqref{eq:infinite_SDE} exists, then conditioned over the sigma algebra $\Fcal_0$, the coordinate processes of $X$ are all independent but not necessarily identically distributed. In particular, taking $b(z, W)(x, y)=-\phi(W)(x, y)$, we recover the system of diffusions in~\eqref{eq:infinite_SDE_1}.
% Without the conditioning on $\round{\Fcal_t}_{t\in\R_+}$, $X$ is an infinite exchangeable array taking values in $C[0,\infty)$.

It is not obvious if an infinite-dimensional stochastic process satisfying~\eqref{eq:infinite_SDE} and~\eqref{eq:whatisgammat} exists, although it is obvious that such a process, if it exists, will be an infinite exchangeable array taking values in $\mathcal{E}=C[0, \infty)$. In the rest of this section, under Assumption~\ref{asmp:b_lip_l2} we show that the process $\left( X, \Gamma \right)$ is indeed well-defined. As will be made clear in Proposition \ref{prop:mck_vlasov}, the limiting object $\Gamma$ is the counterpart to the measure-valued solution of the McKean-Vlasov equation, while every $X_{i,j}$ for $(i,j)\in\Natural^{(2)}$ is the counterpart to the non-linear evolution of a randomly chosen particle evolving in the McKean-Vlasov interacting system. It should be noted that the particles in this McKean-Vlasov interaction correspond to the edges of the graphs not the vertices. The McKean-Vlasov equation here describes how the graphon itself evolves in time and it is different from the McKean-Vlasov system described in the introduction where the McKean-Vlasov equation describes the evolution of particles which may possibly depend on some underlying graphon.

\begin{assumption}\label{asmp:b_lip_l2}
    For a.e. $(x,y)\in\interval{0,1}^{(2)}$, $W_1,W_2\in\Wcal$ and $z_1,z_2\in\interval{-1,1}$, the drift function $b\colon\interval{-1,1}\times \Wcal \to L^\infty([0,1]^{(2)})$ satisfies
    \begin{enumerate}
        \item There exists $L\in\R_+$ such that         $\sup_{W\in \Wcal}\abs{b(z_1,W)(x,y) - b(z_2,W)(x,y)} \leq L\abs{z_1-z_2}$.
        \item There exists $\kappa\in\R_+$ such that $\sup_{z\in \interval{-1,1}} \enorm{b(z,W_1) - b(z,W_2)} \leq \kappa \enorm{W_1-W_2}$.
    \end{enumerate}
\end{assumption}
Observe that Assumption~\ref{asmp:b_lip_l2} implies Assumption~\ref{asmp:R_ell_phi}(\ref{item:phi_lip}) for $\kappa_2^2 = 2(L^2+\kappa^2)$ and that $\norm{\infty}{b(z, W)}\leq C$ uniformly over all $z\in [-1, 1]$ and $W\in \Wcal$.

To argue about the existence of a unique solution of the system of SDEs~\eqref{eq:infinite_SDE}, we construct a sequence of stochastic processes $\left( X^{(k)}, \Gamma^{(k)}\right)_{k \in \Integer_+}$ on $C\big([0, \infty),\interval{-1,1}^{\Natural^{(2)}} \times \Wcal \big)$ iteratively. Start by defining $\round{X^{(0)},\Gamma^{(0)}}$ as $X^{(0)}_{i,j}(t) \equiv W_0(U_i, U_j)$,  $\Gamma^{(0)}(t)\equiv W_0$, for all $(i,j)\in\Natural^{(2)}$, and $t\in\R_+$. The induction proceeds by showing that whenever $\round{X^{(k)}, \Gamma^{(k)}}$ for $k\in\Integer_+$ is well defined, $X^{(k)}$ is an infinite exchangeable array (Lemma \ref{lem:ind_inf_exch} below) and, $\Gamma^{(k)}$ is a deterministic process of kernels (Lemma \ref{lem:deterministic_gamma}). Note that these claims are clearly true for $k=0$. Then, inductively, define the process $X^{(k+1)}$ as the strong solution to the coordinatewise reflected SDE:
\begin{align}\label{eq:picardsde}
    \begin{split}
        \diff X^{(k+1)}_{i,j}(t) &= b\round{X^{(k)}_{i,j}(t),\Gamma^{(k)}(t)}(U_i, U_j)\diff t + \Sigma\left(\Gamma^{(k)}(t)\right)(U_i, U_j)\diff B_{i,j}(t)\\
        &\qquad\qquad\qquad\qquad\qquad\qquad\qquad\qquad + \diff L^{(k+1)-}_{i,j}(t) - \diff L^{(k+1)+}_{i,j}(t),
    \end{split}
\end{align}
for $t\in\R_+$, with the same initial condition $X^{(k+1)}_{i,j}(0)=W_0(U_i, U_j)$ for all $(i,j)\in\Natural^{(2)}$. As usual, $L^{(k+1)-}_{i,j}$ and $L^{(k+1)+}_{i,j}$ are processes such that $\big(X^{(k+1)}_{i,j},L^{(k+1)+}_{i,j},L^{(k+1)-}_{i,j}\big)$ solves the Skorokhod problem with respect to $\interval{-1,1}$ (see Section~\ref{sec:RBM}) for every $(i,j)\in\Natural^{(2)}$. Since the drift and diffusion functions $\phi$ and $\Sigma$ are deterministic and Lipschitz (Assumption~\ref{asmp:R_ell_phi}), given $\mathcal{F}_0$, every process $X^{(k)}$ for $k\in\Natural$ exists uniquely in the strong sense.

In fact, given $\mathcal{F}_0$, the entries of the array $X^{(k+1)}$ are independent and distributed as reflected Brownian motions (RBMs) with Lipschitz (but time-varying) drifts and diffusion coefficients. In particular, the kernel $\Gamma^{(k+1)}$ is constructed from the array $X^{(k+1)}$ (which over the entire probability space is exchangeable, as we show next in Lemma~\ref{lem:ind_inf_exch}) as described in equation~\eqref{eq:graphonexch} in Lemma~\ref{lem:random_graphon}, and is therefore defined as
\begin{equation}\label{eq:picardgamma}
    \Gamma^{(k+1)}(t)(x,y)\coloneqq  \E{ X^{(k+1)}_{1,2}(t) \given U_1=x, U_2=y}, \qquad t\in\R_+.
\end{equation}
The kernel $\Gamma^{(k+1)}(t)$ is well-defined for a.e. $(x,y)\in \interval{0,1}^{(2)}$ and all $t\in\R_+$. The induction hence continues.

\begin{lemma}\label{lem:ind_inf_exch}
    Suppose that, for some $k\in\Integer_+$, there is a unique in law solution to the SDE~\eqref{eq:picardsde} for $X^{(k+1)}$ and that $\Gamma^{(k+1)}$ is a deterministic process of kernels. Then the process $X^{(k+1)}$ is an infinite exchangeable array taking values in $\mathcal{E}=C[0, \infty)$, equipped with the usual locally uniform metric.
\end{lemma}
\begin{proof}
    To argue the exchangeability, let $\sigma\colon\N \rightarrow \N$ be a finite permutation of the natural numbers $\Natural$. Note that $\sigma$ fixes every large enough natural number. We need to argue that $\big( X^{(k+1)}_{i,j}\big)_{(i,j)\in \N^{(2)}}$ has the same law as $\big(X^{(k+1)}_{\sigma_i, \sigma_j}\big)_{(i,j) \in \Natural^{(2)}}$ in the sense of equality of the two probability measures on $\left(C[0, \infty)\right)^{\N^{(2)}}$.

    Let $\widetilde{U}_i \coloneqq U_{\sigma_i}$, for all $i \in \Natural$. Then $\big( \widetilde{U}_i\big)_{i \in \Natural}$ is again a sequence of i.i.d. $\mathrm{Uni}[0,1]$ random variables. Let $Y^{(k+1)}_{i,j}\equiv X^{(k+1)}_{\sigma_i, \sigma_j}$ for every $(i,j)\in\Natural^{(2)}$. Since $Y^{(k+1)}_{i,j}(0)=W_0(U_{\sigma_i}, U_{\sigma_j})\eqqcolon W_0(\widetilde{U}_i, \widetilde{U}_j)$. It follows that $\big(Y^{(k+1)}_{i,j}(0)\big)_{(i,j)\in\Natural^{(2)}}$ has the same distribution as $\big(X^{(k+1)}_{i,j}(0) \big)_{(i,j)\in\Natural^{(2)}}$. Moreover for every $(i,j)\in\Natural^{(2)}$,  the process $Y^{(k+1)}$ satisfies the SDEs
    \[
        \begin{split}
            \diff Y^{(k+1)}_{i,j}(t) &= b\round{X_{\sigma_i,\sigma_j}^{(k)}(t),\Gamma^{(k)}(t)}(U_{\sigma_i}, U_{\sigma_j})\diff t + \Sigma\round{\Gamma^{(k)}(t))(U_{\sigma_i}, U_{\sigma_j}} \diff B_{\sigma_i,\sigma_j}(t)\\
            &\qquad\qquad\qquad\qquad + \diff L^{(k+1)-}_{\sigma_i,\sigma_j}(t) - \diff L^{(k+1)+}_{\sigma_i,\sigma_j}(t)\\
            &= b\round{Y_{i,j}^{(k)}(t),\Gamma^{(k)}(t)}(\widetilde{U}_i, \widetilde{U}_j)\diff t + \Sigma\big(\Gamma^{(k)}(t))(\widetilde{U}_i, \widetilde{U}_j\big) \diff B_{\sigma_i,\sigma_j}(t)\\
            &\qquad\qquad\qquad\qquad + \diff L^{(k+1)-}_{\sigma_i,\sigma_j}(t) - \diff L^{(k+1)+}_{\sigma_i,\sigma_j}(t) ,
        \end{split}
    \]
    for $(i,j)\in\Natural^{(2)}$ and $t\in\R_+$. Note that, $\Gamma^{(k)}$ does not get affected by the permutation $\sigma$. 

    Relabeling $\widetilde{B}_{i,j}\coloneqq B_{\sigma_i, \sigma_j}$, $\widetilde{L}^{(k+1)-}_{i,j} \coloneqq L^{(k+1)-}_{\sigma_i,\sigma_j}$ and $\widetilde{L}^{(k+1)+}_{i,j} \coloneqq L^{(k+1)+}_{\sigma_i,\sigma_j}$ for every $(i,j)\in\Natural^{(2)}$, leaves their joint law unchanged, and we get 
    \begin{align*}
        \diff Y^{(k+1)}_{i,j}(t) &= b\round{Y_{i,j}^{(k)}(t),\Gamma^{(k)}(t)}(\widetilde{U}_i, \widetilde{U}_j)\diff t + \Sigma\round{\Gamma^{(k)}(t)}(\widetilde{U}_i,\widetilde{U}_j)\diff\widetilde{B}_{i,j}(t)\\
        &\qquad\qquad\qquad\qquad\qquad\qquad\qquad\qquad+ \diff \widetilde{L}^{(k+1)-}_{i,j}(t) - \diff \widetilde{L}^{(k+1)+}_{i,j}(t) ,
    \end{align*}
    for every $(i,j)\in\Natural^{(2)}$ and $t\in\R_+$. Since $X^{(k+1)}$ and $Y^{(k+1)}$ follow the same system of recursive SDEs~\eqref{eq:picardsde}, their equivalence in law follows from the uniqueness in law of the SDE.
\end{proof}

% \RaghavS{I feel in probability theory, the word `deterministic' can be very ambiguous and unclear if we have several sigma algebras floating around. Can we state the below lemma as: ... the process $\Gamma^{(k)}$ is adapted w.r.t. the filtration $\round{\Fcal_t}_{t\in\R_+}$? Continuity immediately comes from the definition of this filtration.}

\begin{lemma}\label{lem:deterministic_gamma}
    Under the same assumption as in Lemma~\ref{lem:ind_inf_exch} and Assumption~\ref{asmp:b_lip_l2}, the kernel-valued map $t \mapsto \Gamma^{(k)}\left( t\right)$, is deterministic and absolutely continuous. Moreover, for each $t\in \R_{+}$, we have
    \begin{equation}\label{eq:gamma_graphon_conv}
        \lim_{n\rightarrow \infty} \delta_{\cut}\left( \squarebrack{K\round{\round{X^{(k)}_{i,j}(t)}_{(i,j)\in\squarebrack{n}^{(2)}}}} , \left[\Gamma^{(k)}(t)\right] \right) = 0, \qquad \text{a.s.}
    \end{equation}
\end{lemma}

\begin{proof}
    By definition, for $(x,y) \in \interval{0,1}^{(2)}$, and $t\in\R_+$,
    \(
        \Gamma^{(k)}(t)(x,y) \coloneqq  \E{ X^{(k)}_{1,2}(t) \given U_1=x, U_2=y}.
    \)
    % This is a well-defined function on $\interval{0,1}^{(2)}$ that is independent of both $(U_i)_{i \in \Natural}$ and the array $\left( B_{i,j} \right)_{(i,j)\in \Natural^{(2)}}$.
    This is a deterministic kernel for every $t\in\R_+$. To see~\eqref{eq:gamma_graphon_conv}, repeat the proof of Lemma~\ref{lem:random_graphon}. Notice that, there is no random variable $U$ as in Lemma~\ref{lem:random_graphon} (also see Remark~\ref{rem:random_graphon}). This is now a consequence of Kolmogorov's zero-one law~\cite[Theorem 4.13]{kallenberg21}. For $n\in\Natural$, let $\mathcal{G}_n$ be the sigma algebra generated by $U_n$ and the i.i.d. standard Brownian motions $B_{i,j}$s for the set of indices $\set{(i,j)\in\Natural^{(2)}\given j=n}$. This is a sequence of independent sigma algebras. Consider its tail sigma algebra $\mathcal{T}\coloneqq \cap_{n\in\Natural} \vee_{\ell\ge n} \mathcal{G}_\ell$. This is a trivial sigma algebra by the Kolmogorov zero-one law.
    
    Consider, for any finite simple graph $F$ and $t\in\R_+$, the limiting homomorphism densities $\lim_{n\to\infty}h_F\big(K\big(\big(X^{(k)}_{i,j}(t)\big)_{(i,j)\in\squarebrack{n}^{(2)}}\big)\big)$, as in equation~\eqref{eq:whatistf}. These limiting homomorphism densities do not depend on finitely many elements in $\setinline{X^{(k)}_{i,j}(t)}_{(i,j)\in\Natural^{(2)}}$ or $\set{U_i}_{i\in\Natural}$. In particular, such limits are measurable with respect to the tail sigma algebra $\mathcal{T}$. Exactly as in the proof of Lemma~\ref{lem:random_graphon}, it follows that 
    \[
        \lim_{n\rightarrow \infty} \delta_{\cut}\left( \squarebrack{K\round{\round{X^{(k)}_{i,j}(t)}_{(i,j)\in\squarebrack{n}^{(2)}}}}, \squarebrack{\Gamma^{(k)}(t)} \right) = 0.
    \]
    In particular, the graphon $\left[\Gamma^{(k)}(t)\right]$ is measurable with respect to $\mathcal{T}$, and thus constant a.e.
    
    Finally, the absolute continuity of $t\mapsto \Gamma(t)$ follows from the path continuity of the process $X^{(k)}_{1, 2}$ and our assumptions on $b$ and $\Sigma$.
\end{proof}

\begin{proposition}\label{prop:exists}
    Assume that the drift functions $b\colon\interval{-1,1}\times \Wcal \to L^\infty\big(\interval{0,1}^{(2)}\big)$ satisfies Assumption~\ref{asmp:b_lip_l2}, and the diffusion coefficient function $\Sigma\colon \mathcal{W} \rightarrow L^\infty\big(\interval{0,1}^{(2)}\big)$ is bounded and $\kappa_2$-Lipschitz in $\enorm{}$ (Assumption~\ref{asmp:large_noise}). Then the sequence of processes taking values in $C\left( [0, \infty), \interval{-1,1}\times \Wcal \right)$ given by $\big(\big(X^{(k)}_{1,2}(t),  \Gamma^{(k)}(t)\big)_{t\in\R_+}\big)_{k\in\Integer_+}$, converges locally uniformly in the $2$-product metric of $\interval{-1,1}$ and $(\Wcal,d_2)$, to a pathwise unique process $\big(X_{1,2}(t),\Gamma(t)\big)_{t \in \R_+}$ starting from $\Gamma(0)=W_0\in\Wcal$ and $X_{1,2}(0)=W_0(U_1, U_2)$. That is, for every $t\in\R_+$, 
    \begin{align}
        \lim_{k\to\infty}\sup_{s\in[0,t]}\left[\abs{X^{(k)}_{1,2}(s) - X_{1,2}(s)}^2 + \enorm{\Gamma^{(k)}\left(s\right)-\Gamma(s)}^2\right]=0, \qquad a.s.
    \end{align}
    In particular, the limiting processes $X_{1,2}$ is continuous and $\Gamma$ is absolutely continuous and deterministic.
\end{proposition}

\begin{proof}
    The proof is a standard Picard iteration based proof of existence of solutions of SDEs. See, for example, the proof of~\cite[Theorem 2.9, page 289]{KS91}. Hence, we will skip some of the details and refer the reader to the above cited reference. 

    We will take $k\rightarrow \infty$ and produce a limit. Start by noticing that the process $X^{(k+1)}_{1,2}\colon\R_+\to\interval{-1,1}$ is the result of applying the Skorokhod map~\cite{kruk2007explicit} pathwise to the ``noise before reflection'' process $Y^{(k+1)}_{1,2}$ obtained as the unique strong solution to the SDE:
    \begin{align}\label{eq:aux_process_Y}
        \diff Y^{(k+1)}_{1,2}(t) = b\round{X_{1,2}^{(k)}(t),\Gamma^{(k)}(t)}(U_1, U_2)\diff t + \Sigma\left(\Gamma^{(k)}(t) \right)(U_1, U_2) \diff B_{1,2}(t),
    \end{align}
    for $t\in\R_+$, with initial conditions $Y^{(k+1)}_{1,2}(0)=X^{(k+1)}_{1,2}(0)=W_0(U_1, U_2)$ for all $k\in\Integer_+$.
    
    Fix $t\in\R_+$ and consider $\sup_{s\in[0,t]}\abs{X^{(k+1)}_{1,2}(s) - X^{(k)}_{1,2}(s)}$ for any $k\in\Natural$. Since the Skorokhod map is $4$-Lipschitz in the local uniform norm (see Section~\ref{sec:RBM}), the above distance is bounded by $4\sup_{s\in[0,t]}\abs{Y^{(k+1)}_{1,2}(s) - Y^{(k)}_{1,2}(s)}$. Now for every fixed $k\in\Natural$, from equation~\eqref{eq:aux_process_Y} we have
    \begin{align}\label{eq:decomp_semi}
        \begin{split}
            &Y_{1,2}^{(k+1)}(t) - Y_{1,2}^{(k)}(t)\\
            &= \int_0^t \left(b\round{X_{1,2}^{(k-1)}(t),\Gamma^{(k-1)}(t)}(U_1, U_2) - b\round{X_{1,2}^{(k)}(t),\Gamma^{(k)}(t)}(U_1, U_2)\right) \diff s\\
            &\qquad - \int_0^t \round{ \Sigma\round{\Gamma^{(k-1)}}(U_1, U_2) - \Sigma\round{\Gamma^{(k)}}(U_1, U_2) }\diff B_{1,2}(s).
        \end{split}
    \end{align}
    Define $\Delta,M\colon \R_+ \to \R$ for $t\in\R_+$ as
    \[
    \begin{split}
        \Delta(t) &\coloneqq \int_0^t \left(b\round{X_{1,2}^{(k-1)}(t),\Gamma^{(k-1)}(t)}(U_1, U_2) - b\round{X_{1,2}^{(k)}(t),\Gamma^{(k)}(t)}(U_1, U_2)\right) \diff s,\\
        M(t) &\coloneqq \int_0^t \round{ \Sigma\round{\Gamma^{(k-1)}}(U_1, U_2) - \Sigma\round{\Gamma^{(k)}}(U_1, U_2) }\diff B_{1,2}(s).
    \end{split}
    \]
    
    Note that, for a kernel $A\in\Wcal$, we have $\enorm{A}^2= \E{A^2(U_1,U_2)}$, for $U_1,U_2$ i.i.d. as $\mathrm{Uni}[0,1]$. Using Jensen's inequality and interchanging expectation with integral and Assumption~\ref{asmp:b_lip_l2},
    \begin{align}
            &\E{\sup_{s\in\interval{0,t}}\Delta^2(s)}\nonumber\\
            &\le t \E{\int_0^t \abs{b\round{X_{1,2}^{(k-1)}(t),\Gamma^{(k-1)}(t)}(U_1, U_2) - b\round{X_{1,2}^{(k)}(t),\Gamma^{(k)}(t)}(U_1, U_2)}^2 \diff s }\nonumber\\
            &= t \int_0^t \enorm{b\round{X_{1,2}^{(k-1)}(t),\Gamma^{(k-1)}(t)} - b\round{X_{1,2}^{(k)}(t),\Gamma^{(k)}(t)}}^2 \diff s\nonumber\\
            &\leq  2 \kappa^2t \int_0^t \enorm{\Gamma^{(k-1)}(s) - \Gamma^{(k)}(s)}^2 \diff s + 2L^2t \int_0^t \E{\abs{X^{(k-1)}(s) - X^{(k)}(s)}^2}\diff s.\label{eq:E_sup_delta^2}
    \end{align}
    For $M$, we use the fact that it is a stochastic integral of a bounded integrand with respect to a Brownian motion, and hence a continuous martingale.
    % \RaghavS{what does `bounded integral with respect to a Brownian motion' mean? Are we using any assumption?}.\SP{I have changed the language slightly.}\RaghavS{So we need $\Sigma$ to be in $L^\infty([0,1]^2)$? I was hoping probably $L^2([0,1]^2)$ would be sufficient.}\SP{Unfortunately, boundedness is \textit{very} important for the volatility function. There are counterexamples when things do not exist without it.}
    By an application of Doob's maximal inequality~\cite[Theorem 3.8.iv, page 14]{KS91}, we get that,
    \[
    \begin{split}
        \E{\sup_{s\in\interval{0,t}}M^2(s)}\le 4 \int_0^t \E{\abs{\Sigma\round{\Gamma^{(k-1)}(s)}(U_1, U_2) - \Sigma\round{\Gamma^{(k)}(s)}(U_1, U_2)}^2} \diff s.
    \end{split}
    \]
    % Here and below $C_1$ will refer to some generic positive constant whose value might change with each occurrence but it does not depend on any parameter of this problem except the time horizon $[0,t]$ and the Lipschitz constant of $\phi$ and $\Sigma$.
    Using the assumption that $\Sigma$ is $\kappa_2$-Lipschitz in $\enorm{}$ and the same argument as above,
    \begin{align}
        \E{\sup_{s\in\interval{0,t}}M^2(s)}\le 4 \kappa_2^2 \int_0^t \enorm{\Gamma^{(k-1)}(s) - \Gamma^{(k)}(s)}^2 \diff s. 
    \end{align}
    Now, taking absolute values on both sides on \eqref{eq:decomp_semi}, we immediately get,
    % \[
    %   \sup_{s\in[0,t]}\abs{  X^{(k+1)}_{1,2}(s) -  X^{(k)}_{1,2}(s)} \le   \sup_{ s\in \interval{0,t} }\abs{Y^{(k+1)}_{1,2}(s) -  Y^{(k)}_{1,2}(s)} \le \Delta(s) + M(s).
    % \]
    \begin{align}
        & \E{\sup_{s\in\interval{0,t}} \abs{X_{1,2}^{(k+1)}(s) - X_{1,2}^{(k)}(s)}^2}\nonumber\\
        \leq & 16 \E{\sup_{s\in\interval{0,t}} \abs{Y_{1,2}^{(k+1)}(s) - Y_{1,2}^{(k)}(s)}^2} \leq  32 \E{\sup_{s\in\interval{0,t}}\Delta^2(s) + \sup_{s\in\interval{0,t}}M^2(s) }\nonumber\\
        \leq & 64 (\kappa^2t+2\kappa_2^2)\int_0^t \enorm{\Gamma^{(k-1)}(s) - \Gamma^{(k)}(s)}^2 \diff s\nonumber\\
        &\qquad + 64 L^2 t \int_0^t \E{\abs{X^{(k-1)}(s) - X^{(k)}(s)}^2}\diff s. 
    \end{align}
    % \[
    % \sup_{s\in[0,t]}\E{\enorm{  X^{(k+1)}_{i,j}(s) -  X^{(k)}_{i,j}(s)}^2} \le C_1 \int_0^t \enorm{\left(\Gamma^{(k-1)}(s) - \Gamma^{(k)}(s)\right)}^2 \diff s. 
    % \]
    Using the fact that the operator $\Gamma$, given by a conditional expectation \eqref{eq:picardgamma}, and, therefore, must have a smaller $L^2$ norm%\RaghavS{unclear}, 
    \[
        \begin{split}
            \sup_{s\in[0,t]} & \enorm{\Gamma^{(k+1)}(s) - \Gamma^{(k)}(s)}^2
            \le \E{\sup_{s\in\interval{0,t}} \abs{X_{1,2}^{(k+1)}(s) - X_{1,2}^{(k)}(s)}^2}.
        \end{split}
    \]
    Combining the last two bounds above, one gets the recursive bound 
    \[
        \begin{split}
            \mathbb{E}&\left[\sup_{s\in[0,t]}{\abs{  X^{(k+1)}_{1,2}(s) -  X^{(k)}_{1,2}(s)}^2}+ \sup_{s\in[0,t]}\enorm{\Gamma^{(k+1)}(s)-\Gamma^{(k)}(s)}^2\right]\\
            &\le 128 ((\kappa^2+L^2)t+4\kappa_2^2) \int_0^t \E{\abs{X^{(k-1)}(s) - X^{(k)}(s)}^2}\diff s.
        \end{split}
    \]
    %\RaghavS{should $\mathbb{E}$ and $\sup_{s\in\interval{0,t}}$ come in the reversed order? Also, should it be the absolute value $\abs{\slot{}}^2$ or $\enorm{}^2$ for the exchangeable array?}
    The rest of the argument follows exactly as in~\cite[page 290]{KS91} by applications of Gr\"onwall's lemma~\cite{gronwall1919note} and the Borel-Cantelli lemma~\cite[Theorem 4.18]{kallenberg21}. We skip the similar argument for pathwise uniqueness. See the proof of \cite[Proposition 2.13, page 291]{KS91}.
\end{proof}

\begin{proposition}\label{prop:mck_vlasov}
    Suppose the assumptions in Proposition \ref{prop:exists} holds. 
    Given any kernel $W_0\in \mathcal{W}$, there exists a pathwise unique strong solution to the coupled system~\eqref{eq:infinite_SDE} and~\eqref{eq:whatisgammat} in the following sense. In any probability space supporting countably many i.i.d. $\mathrm{Uni}[0,1]$ random variables $\left(U_i\right)_{i \in \Natural}$ and an independent infinite (symmetric) array of i.i.d. standard Brownian motions $\left( B_{i,j} \right)_{(i,j)\in \Natural^{(2)}}$, one can construct an infinite exchangeable array of reflected diffusions $\left(X_{i,j} \right)_{(i,j) \in \Natural^{(2)}}$ that satisfy~\eqref{eq:infinite_SDE} and~\eqref{eq:whatisgammat} and every $X_{i,j}$ is pathwise unique.
    
    Moreover, for every $t\in\R_+$, $[\Gamma(t)]$ can be recovered as the $\delta_\cut$ limit of the sequence of graphons $\big(\big\lbrack K\big(\round{X_{i,j}(t)}_{(i,j) \in [n]^2}\big)\big\rbrack\big)_{n\in\Natural}$ locally uniformly in time. That is, for any $t\in \R_+$,
    \begin{equation}\label{eq:gamma_graphon_conv2}
        \lim_{n\rightarrow \infty} \sup_{s\in \interval{0,t}}\delta_{\cut}\left( \squarebrack{K\round{\round{X_{i,j}(s)}_{(i,j)\in\squarebrack{n}^{(2)}}}} , \left[\Gamma(s)\right] \right) = 0, \qquad \text{a.s.}
    \end{equation}
\end{proposition}

%\textcolor{red}{SP: improve above to a locally uniform over time bound.}

\begin{proof}
    Start with the countably many i.i.d. $\mathrm{Uni}[0,1]$ random variables $\left(U_i \right)_{i \in \Natural}$ and an independent infinite (symmetric) array of i.i.d. standard Brownian motions $\left( B_{i,j} \right)_{(i,j)\in \Natural^{(2)}}$ and construct the deterministic process $\Gamma$ in Proposition~\ref{prop:exists}.
    
    Given $\Gamma$ and $\left(U_i \right)_{i \in \Natural}$ and following the system of SDEs~\eqref{eq:infinite_SDE}, the diffusions $X_{i,j}$s are independent (but not identically distributed) reflected Brownian motions with deterministic bounded time-dependent drifts for $(i,j)\in\Natural^{(2)}$. So, they exist in a pathwise or strong sense exactly as the process $X_{1,2}$ does in Proposition~\ref{prop:exists} and satisfies the constraint~\eqref{eq:infinite_SDE} since $\Gamma$ is a fixed point of the Picard iterations.
    
    It is obvious from the symmetry of the construction that the infinite array $\left( X_{i,j}\right)_{(i,j)\in \Natural^{(2)}}$ is exchangeable in the sense of Section~\ref{subsec:exch_array} with $\mathcal{E}=C[0, \infty)$, the set of continuous functions from $[0,\infty)$ to $\rr$.   
    
    For the limit \eqref{eq:gamma_graphon_conv2} we will make use of the following result from \cite[Proposition 8.12]{lovasz2012large}, which states that for any $V\in \mathcal{W}$,
    \begin{equation}\label{eq:cutnormcomp}
        \cutnorm{V}^4 \le h_{C_4}\left(V \right) \le 4 \cutnorm{V}. 
    \end{equation}
    Here $C_4$ is the cyclic graph with four vertices and $h_{C_4}(V)$ is the homomorphism density function of the simple graph $C_4$. We will apply this for the choice of $V_n(t)  \coloneqq K\left( (X_{i,j}(t))_{(i,j)\in [n]^2} \right)- K\round{\round{\Gamma(t)(U_{i}, U_j)}_{(i,j)\in [n]^2}}$. Thus,
    \[
        % \begin{split}
        %     H_n(t)& \coloneqq h_{C_4}(V_n(t)) = \frac{1}{n^{\downarrow 4}}\sum_{i_1,i_2,\ldots,i_4} \prod_{l=1}^4 \round{X_{i_l, i_{l+1}}(t) - \Gamma(t)(U_{i_l}, U_{i_{l+1}})}\\
        %     &=\frac{1}{n^{\downarrow 4}}\sum_{i_1,i_2,\ldots,i_4} \prod_{l=1}^4 \round{X_{i_l, i_{l+1}}(t) - \E{X_{i_l, i_{l+1}}(t) \given \mathcal{F}_0}}, 
        % \end{split}
        \begin{split}
            H_n(t)& \coloneqq h_{C_4}(V_n(t)) = \frac{1}{n^{\downarrow 4}}\sum_{i_1,i_2,\ldots,i_4} \prod_{l=1}^4 \round{X_{i_l, i_{l+1}}(t) - \Gamma(t)(U_{i_l}, U_{i_{l+1}})}\\
            &=\frac{1}{n^{\downarrow 4}}\sum_{i_1,i_2,\ldots,i_4} \prod_{l=1}^4 \round{X_{i_l, i_{l+1}}(t) - \E{X_{i_l, i_{l+1}}(t) \given \mathcal{F}_0}}, 
        \end{split}
    \]
    with the convention that, when $l=4$, $l+1\equiv 1$. The above sum is over all injections in $[n]^{[4]}$.
    
    Notice that $H_n(0)=0$. The fact that for each $t\in\R_+$, $\lim_{n\rightarrow \infty} H_n(t)=0$ almost surely follows similarly to the proof of Lemma~\ref{lem:random_graphon}. We now show that $t\mapsto H_n(t)$ is equicontinuous. From which, using a standard argument, we can show that almost surely, $H_n(t)\to 0$ for each $t\in \R_+$, that is,
    \[
        \lim_{n\to\infty}\delta_{\cut}\left( \squarebrack{K\round{\round{X_{i,j}(s)}_{(i,j)\in\squarebrack{n}^{(2)}}}} , \left[\Gamma(s)\right] \right) = 0, \qquad \text{a.s.}\quad \forall\ s\in [0, t].
    \]
    To show that $\round{H_n}_{n\in\Natural}$ is equicontinuous, we first observe that for any $s_1,s_2 \in\interval{0,t}$,
    \begin{align}
    \begin{split}
        &\abs{H_n(s_2)-H_n(s_1)}\\
        &\leq 16 \enorm{K\round{\round{X_{i,j}(s_2)}_{(i,j)\in\squarebrack{n}^{(2)}}} - K\round{\round{X_{i,j}(s_1)}_{(i,j)\in\squarebrack{n}^{(2)}}}}\\
        &\qquad\qquad\qquad\qquad\qquad\qquad\qquad\qquad\qquad + 16\enorm{\Gamma(s_2) - \Gamma(s_1)},
        % &\leq 4\delta_{2}\left( \left[\Gamma(t)\right] , \left[\Gamma(s)\right] \right)+ 4 \delta_{2}\left( \squarebrack{K\round{\round{X_{i,j}(t)}_{(i,j)\in\squarebrack{n}^{(2)}}}} , K\round{\round{X_{i,j}(s)}_{(i,j)\in\squarebrack{n}^{(2)}}}\right),
    \end{split}\label{eq:RHS_equicontinuity}
    \end{align}
    where the inequality follows by an application of the counting lemma~\cite[Lemma 10.23, Exercise 10.27]{lovasz2012large}, the triangle inequality and using the fact that the cut norm $\cutnorm{}$ is upper bounded by the $L^2$ norm $\enorm{}$. 
    
    %Note that by the definition of $\Gamma$ in equation~\eqref{eq:whatisgammat}, and by the contractive property of conditional expectations, we have
   % \[
   %     \enorm{\Gamma(s_2) - \Gamma(s_1)} \leq \enorm{K\round{\round{X_{i,j}(s_2)}_{(i,j)\in\squarebrack{n}^{(2)}}} - K\round{\round{X_{i,j}(s_1)}_{(i,j)\in\squarebrack{n}^{(2)}}}}.
   % \]
  %  Note that
  %  \[
   %     \begin{split}
   %         &\enorm{K\round{\round{X_{i,j}(s_2)}_{(i,j)\in\squarebrack{n}^{(2)}}} - K\round{\round{X_{i,j}(s_1)}_{(i,j)\in\squarebrack{n}^{(2)}}}}^2\\
   %         &= \frac{1}{n^2}\sum_{(i, j)\in [n]^{(2)}}\abs{X_{i, j}(s_2)-X_{i, j}(s_1)}^2.
   %     \end{split}
  %  \]
  %  Combining the above inequalities we obtain
   % \begin{align}\label{eqn:final}
    %    \abs{H_n(s_2)-H_n(s_1)}^2 \leq \frac{2^9}{n^2}\sum_{(i, j)\in [n]^2} \abs{X_{i, j}(s_2)-X_{i, j}(s_1)}^2.
    %\end{align}
    Using the Lipschitzness of the Skorokhod map (see equation~\eqref{eq:skorokhod_lipschitz}), we therefore obtain
    \begin{align}
        &\enorm{K\round{\round{X_{i,j}(s_2)}_{(i,j)\in\squarebrack{n}^{(2)}}} - K\round{\round{X_{i,j}(s_1)}_{(i,j)\in\squarebrack{n}^{(2)}}}}^2\nonumber\\
        &\leq \frac{2^{4}}{n^2}\sum_{(i, j)\in [n]^{(2)}}\abs{Y_{i, j}(s_2)-Y_{i, j}(s_1)}^2\nonumber\\
        &\leq \frac{2^{5}}{n^2}\sum_{(i, j)\in [n]^{(2)}} \abs{\int_{s_1}^{s_2}b\round{X_{1,j}(u),\Gamma(u)}(U_i, U_j)\diff u}^2\nonumber\\
        &\qquad\qquad + \frac{2^{5}}{n^2}\sum_{(i, j)\in [n]^{(2)}}\abs{\int_{s_1}^{s_2}\Sigma(\Gamma(u))(U_i, U_j)\diff B_{i, j}(u)}^2\nonumber\\
        &\leq 2^{5} M_\infty^2 \abs{s_2-s_1}^2 +  \frac{2^{5}}{n^2}\sum_{(i, j)\in [n]^2}\abs{\int_{s_1}^{s_2}\Sigma(\Gamma(u))(U_i, U_j)\diff B_{i, j}(u)}^2.
    \end{align}
    Now let $\abs{s_2-s_1}\leq \delta$ for some $\delta>0$. Set for all $(i,j)\in\squarebrack{n}^{(2)}$,
    \[
        \eta_{i, j} \coloneqq \sup_{\substack{s_1, s_2\in [0, t],\\\abs{s_2-s_1}\leq \delta}}\abs{\int_{s_1}^{s_2}\Sigma(\Gamma(u))(U_i, U_j)\diff B_{i, j}(u)}^2.
    \]
 
    From~\cite[Lemma A.4]{slominski2001euler}, there exist constants $C_{1,t},C_{2,t}\in\R_+$ depending of $t$, such that for all $(i,j)\in\squarebrack{n}^{(2)}$,
    \begin{align}\label{eqn:etamean}
        % \E{\eta_{i,j}} \leq \int_{s_1}^{s_2}\abs{\Sigma(\Gamma(u))(U_i, U_j)}^2\diff u \leq M_\infty^2\abs{t_2-t_1}\log(\inv{\abs{t_2-t_1}}).
        \E{\eta_{i,j}} \leq M_\infty^2 C_{1,t}\delta\abs{\log\inv{\delta}},\qquad\text{and}\qquad    \E{\eta_{i,j}^2} \leq M_\infty^4 C^2_{2,t}\delta^2\log^2\inv{\delta}.
    \end{align}
    Since, $\eta_{i, j}$s are independent and have finite variance, it follows from the Chebyshev's inequality~\cite[Lemma 5.1]{kallenberg21} that 
    \begin{align*}
        \Prob{\abs{\frac{1}{n^2}\sum_{(i, j)\in [n]^{(2)}}\eta_{i, j} - \E{\eta_{i,j}}} \geq \max_{(i,j)\in\squarebrack{n}^{(2)}}\mathrm{Var}^{1/2}(\eta_{i, j}) } &\leq \frac{1}{n^2}.
    \end{align*}
    Using the Borel-Cantelli lemma~\cite[Theorem 4.18]{kallenberg21}, it follows that almost surely, 
    \begin{equation}\label{eqn:secondterm}
        \frac{1}{n^2}\sum_{(i, j)\in [n]^{(2)}}\eta_{i, j} \leq M_\infty^2 (C_{1,t}+C_{2,t})\delta\abs{\log\inv{\delta}}, %(C^{1/2}+M_\infty^2)\abs{t_2-t_1}\abs{\log\inv{\abs{t_2-t_1}}} %2M_\infty^2\abs{t_2-t_1}\log\inv{\abs{t_2-t_1}},
    \end{equation}
    for all $n\in\Natural$, sufficiently large. Combining equations~\eqref{eq:RHS_equicontinuity} and~\eqref{eqn:secondterm}, we obtain that almost surely, for all $n\in\Natural$ sufficiently large, we have 
    \begin{align*}
        \sup_{\substack{s_1, s_2\in [0, t],\\\abs{s_2-s_1}\leq \delta}} \abs{H_n(s_2)-H_n(s_1)} &\leq  2^8 M_\infty \round{\delta +  (C_{1,t}+C_{2,t})^{1/2}\delta^{1/2}\log^{1/2}\inv{\delta}} +16\omega(\delta), %\sup_{\substack{s_1, s_2\in [0, t],\\\abs{s_2-s_1}\leq \delta}} \norm{2}{\Gamma(s_2)-\Gamma(s_1)}.
    \end{align*}
    where $\omega(\delta)\coloneqq \sup_{\substack{s_1, s_2\in [0, t],\abs{s_2-s_1}\leq \delta}} \norm{2}{\Gamma(s_2)-\Gamma(s_1)}$ is the modulus of continuity of the curve $t\mapsto \Gamma(t)$.
    Since $s\mapsto \Gamma(s)$ is continuous in $(\Wcal, d_2)$ (and independent of $n$), it follows that, almost surely, $\round{H_n}_{n\in\Natural}$ is equicontinuous. Since $\round{H_n}_{n\in\Natural}$ is equicontinuous uniformly bounded almost surely, the proof is complete by a standard application of Arzel\`a-Ascoli theorem~\cite[Theorem 47.1]{munkres2000topology}.
\end{proof}

\begin{proposition}\label{prop:velocity}
%The limiting curve $\Gamma$ in Proposition~\ref{prop:mck_vlasov} is absolutely continuous in $d_2$ (see Definition~\ref{def:AC}). 
Suppose that $\Sigma \equiv \beta > 0$ and $b(z, W)=-\phi(W)$. Then, the limiting curve $\Gamma$ in Proposition~\ref{prop:mck_vlasov} has a velocity
\begin{align}
    \begin{split}
        \dot{\Gamma}(t) &= -\phi(\Gamma(t))- \squarebrack{p_{\beta^2 t}^{(+1)}(W_0,\phi\circ\Gamma,\beta) - p_{\beta^2 t}^{(-1)}(W_0,\phi\circ\Gamma,\beta)},
    \end{split}
    \label{eq:velocity}
\end{align}
where $p_{s}^{(\pm 1)}(W_0,\phi\circ\Gamma,\beta)(x,y)$ is the density of the real-valued reflected Brownian motion $Z$ at $\pm 1$, at time $s\in\R_+$, starting at $Z(0) = W_0(x,y)$, satisfying
\[
    \diff Z(s) = -\inv{\beta^2}\phi(\Gamma(s/\beta^2))(x,y)\diff s + \diff B(s) + \diff L^-(s) - \diff L^+(s), \qquad s\in\R_+,
\] 
%$\Psi(s;\beta) \coloneqq -\inv{\beta^2}\phi(\Gamma(s/\beta^2))(x,y)$; and 
where $(Z,L^+,L^-)$ solves the Skorokhod problem with respect to the set $\interval{-1,1}$ (see Section~\ref{sec:RBM}).
% $X_{1,2}(t)$ at $\pm 1$, with the initial condition $X_{1,2}(0)=W_0(x,y)$. 
\end{proposition}

\begin{proof}
Given $(U_1,U_2)=(x,y)$, the process $X_{1,2}$ is a diffusion with a Lipschitz drift and a constant diffusion coefficient. Using~\eqref{eq:whatisgammat} and It\^o's formula, we get 
\begin{align}
    \begin{split}
        \deriv{}{t}&\Gamma(t)(x, y) = -\deriv{}{t}\phi(\Gamma(t))(x, y)\\
        &\qquad\qquad + \deriv{}{t}\E{L^-_{1,2}(t)\given U_1=x, U_2=y} - \deriv{}{t}\E{L^+_{1,2}(t)\given U_1=x, U_2=y}.
    \end{split}\label{eq:velocity_xy}
\end{align}
%which proves that the curve $\Gamma$ is in $\mathrm{AC}\big(L^\infty(\interval{0,1}^{(2)}),d_2\big)$.

% When $\Sigma \equiv \beta > 0$, following~\cite[Exercise (1.12), page 407]{RY},
Now consider the reflecting diffusion $Z$ which solves the SDE
\begin{align}
    \diff Z(s) = \Psi(s;\beta)\diff s + \diff B(s) + \diff L^-(s) - \diff L^+(s), \qquad s\in\R_+,\label{eq:scaled_process}
\end{align}
starting at $Z(0) = W_0(x,y)$, such that $(Z,L^+,L^-)$ solves the Skorokhod problem with respect to the set $\interval{-1,1}$, and $\Psi(s;\beta) \coloneqq -\inv{\beta^2}b\round{\Gamma(s/\beta^2)}(x,y)$ for all $s\in\R_+$ (see Section~\ref{sec:RBM}). By reparametrizing $s=\beta^2 t$ and setting $Z(s)=X_{1,2}(t)$, we get back our reflected diffusion $X_{1,2}$ in law following
\begin{align*}
    \diff Z(\beta^2t) &= -\inv{\beta^2}\phi\round{\Gamma(t)}(x,y)\diff (\beta^2t) + \diff B(\beta^2t) + \diff L^-(\beta^2t) - \diff L^+(\beta^2t),\nonumber\\
    \implies X_{1,2}(t) &= -\phi\round{\Gamma(t)}\diff t + \beta\diff B(t) + \diff L^-(\beta^2t) - \diff L^+(\beta^2t), \qquad t\in\R_+,
\end{align*}
where the processes $(L^{+}(\beta^2 t))_{t\in\R_+}$ and $(L^{-}(\beta^2 t))_{t\in\R_+}$ constrain the process $X_{1,2}$ in the interval $\interval{-1,1}$ (see Section~\ref{sec:RBM}). Here the equality is in law. We use the fact that the solution of both the above SDEs agree in law since the distribution of $B(\beta^2t)$ and $\beta B(t)$ coincide for all $\beta\in\R_+$. Let $p_s^{(\pm 1)}(W_0,\phi\circ\Gamma,\beta)(x,y)$ denote the transition density of the solution of SDE~\eqref{eq:scaled_process} at time $s\in\R_+$ at the boundary $\pm 1$, then the transition density of the process $X_{1,2}$ at time $t$ at the boundary $\pm 1$ is $p_{\beta^2 t}^{(\pm 1)}(W_0,\phi\circ\Gamma,\beta)(x,y)$.

Using~\cite[Exercise (1.12), page 407]{RY} and equation~\eqref{eq:velocity_xy}, we deduce that %$p_{\beta^2 t}^{(\pm 1)}(W_0,b\circ (X_{1,2},\Gamma),\beta)(x,y)$ is continuous in time $t$ and the starting point $W_0(x,y)$ then
\begin{align}
    \deriv{}{t}\E{L^{\pm}_{i,j}(t)} &= p_{\beta^2 t}^{(\pm 1)}(W_0,b\circ (X_{1,2},\Gamma),\beta)(x,y),
\end{align}
which gives us the desired result.
%combined with equation~\eqref{eq:velocity_xy}, gives us the stated formula. 
\end{proof}

\begin{remark}
    Note that the (pointwise) velocity of the curve $\Gamma$ at time $t\in\R_+$ is not $-(\phi\circ \Gamma)(t)$ when $\beta>0$. That is, $\Gamma$ is not a gradient flow of the function $R$ when $\beta>0$, and the effect of the boundary $\set{-1,1}$, as seen in~\eqref{eq:velocity}, is qualitatively different from that when $\beta = 0$ (see Section~\ref{sec:beta_0}).
\end{remark}

\subsection{Convergence of the finite dimensional processes}\label{sec:convergence_SDEn}

Consider now the finite dimensional SDE \eqref{eq:RSDE}: 
\begin{equation}\label{eq:sden2}
    \diff X_n(t) = -n^2\nabla R_n(X_n(t))\diff t + \Sigma_n(X_n(t)) \circ \diff B_n(t) + \diff L^{-}_n(t) - \diff L^{+}_n(t).
\end{equation}
%with initialization $X_n(0)=W_{n, 0}$.

The Fr\'echet-like derivative of $R$ is a symmetric kernel-valued map from $\mathcal{W} \rightarrow L^\infty\big( \interval{0,1}^{(2)}\big)$. Thus, for $(x,y) \in \interval{0,1}^{(2)}$, there is a real-valued map $\phi_{x,y}\colon \mathcal{W} \rightarrow \rr$ given by $\phi_{x,y}\left( V\right)= \phi(V)\round{x,y}$ for all $V\in\Wcal$. This is the same map that we get when we replace $(x,y)$ by $(y,x)$. To show that the finite dimensional processes converge as $n\to\infty$, we will need to put further assumptions on the drift and diffusion functions.

\begin{assumption}\label{asmp:b_lip_cut}
    There exists a constant $\kappa_\cut\in\R_+$ such that for all $W_1,W_2\in\Wcal$, the drift function $b\colon\interval{-1,1}\times \Wcal \to L^\infty\big(\interval{0,1}^{(2)}\big)$ and the diffusion coefficient function $\Sigma\colon\Wcal \to L^\infty\big(\interval{0,1}^{(2)}\big)$ satisfy
    \begin{align*}
        \sup_{(x,y)\in \interval{0,1}^2}\sup_{z\in \interval{-1,1}}\abs{b(z,W_1)(x,y) - b(z,W_2)(x,y)} &\leq \kappa_\cut \cutnorm{W_1-W_2},\qquad \text{and}\\
        \sup_{(x,y)\in \interval{0,1}^2}\abs{\Sigma(W_1)(x,y) - \Sigma(W_2)(x,y)} &\leq \kappa_\cut \cutnorm{W_1-W_2}.
    \end{align*}
\end{assumption}

\begin{proposition}
    Suppose the assumptions in Proposition~\ref{prop:exists} and Assumption~\ref{asmp:b_lip_cut} hold. Then, for any sequence of initial kernels $\big(W^{(n)}_0 \in\Wcal_n\big)_{n\in\Natural}$ that converges to $W_0\in\Wcal$ in the $L^2\big(\interval{0,1}^{(2)}\big)$ norm $\enorm{}$, i.e., whenever 
    \begin{equation}\label{eq:initassump}
        \lim_{n\rightarrow \infty}\enorm{W_{0}^{(n)}-W_0}=0,
    \end{equation}
    the process of random kernels $\round{K(X_n(t))}_{t\in\R_+}$ obtained from solutions of the SDEs~\eqref{eq:sden2}, converges locally uniformly in the cut norm as $n\to\infty$, in probability, to the limiting process $\Gamma \colon\R_+ \to \Wcal$, with $\Gamma(0)=W_0$, established in Proposition~\ref{prop:mck_vlasov}.
\end{proposition}

\begin{proof}
    Consider a probability space satisfying the assumptions of Proposition \ref{prop:mck_vlasov} and an infinite exchangeable array of diffusions $(X_{i,j})_{(i,j)\in \Natural^{(2)}}$ on it. For $k\in [n]$ and any $t\in \rr_+$, consider the sampled $k\times k$ symmetric matrix $\Gamma(t)[k]$ whose $(i,j)$-th element is $\Gamma(t)(U_i, U_j)$, $(i,j)\in [k]^{(2)}$. Consider also the corresponding $k\times k$ matrix of diffusions $X^{(k)}(\cdot) \coloneqq \left( X_{(i,j)} \right)_{(i,j)\in [k]^{(2)}}$.  
    
    Now consider $K(X_n(t))$ from a solution of SDEs~\eqref{eq:sden2}. One may construct a sampled $k\times k$ matrix from this kernel as well. We estimate the cut distance of this sampled matrix from $\Gamma(t)[k]$ by coupling this sampled matrix with  $K\round{X^{(k)}}$ in a particular way.
    
    Notice that, for any $(i,j)\in [k]^{(2)}$ and $(m_i, m_j)\in [n]^{(2)}$, if $U_i\in ((m_i-1)/n, m_i/n]$ and $U_j \in ((m_j-1)/n, m_j/n]$, then $K(X_n(t))(U_i, U_j)\equiv X_{n, m_i, m_j}(t)$. Let $E_k(n)$ denote the event that that no two $U_i, U_{i'}$, for distinct $i,i'\in[k]^{(2)}$, falls in the same interval $((m-1)/n, m/n]$. Under this event every entry of the sampled diffusions will be run by independent standard Brownian motions. Before we use this property to proceed with our coupling, let us show that $E_{k}(n)$ happens with high probability as $k$ is fixed and $n \rightarrow \infty$. Order the uniform random variables as $U_{(1)}< U_{(2)} < \ldots < U_{(k)}$. Clearly $E^c_k(n)$ implies that there is at least one pair $(U_{(i)}, U_{(i+1)})$ for $i\in[k-1]$, such that $U_{(i+1)} - U_{(i)} \le 1/n$. Hence 
    $
        \Prob{ E_k^c(n) } \le \Prob{\min_{i\in [k-1]} \round{U_{(i+1)} - U_{(i)}} \le \frac{1}{n}}
    $.
    But $\min_{i\in [k-1]} \round{U_{(i+1)} - U_{(i)}}$ has a density at zero and hence the above probability is $O(1/n)$, which goes to zero as $n\rightarrow \infty$. Thus $\lim_{k\to\infty}\lim_{n\rightarrow \infty} \Prob{E_k(n)}=1$.
    
    On the event $E_k(n)$, every $m_i$, $i\in [k]$, is distinct. Consider the corresponding independent Brownian motion $B_{i,j}$ from the diffusion $X_{i,j}$ from equation~\eqref{eq:infinite_SDE}. Since~\eqref{eq:sden2} admits a strong solution, construct a solution where the entry processes $X_{n,m_i,m_j}(\cdot)$ is driven by $B_{i,j}$, $(i,j)\in [k]^{(2)}$, while the rest of the entries of $X_{n}$ are driven by a disjoint subset of $(B_{i,j})_{(i,j)\in \Natural^2}$. Thus, one couples $K(X_n)(\cdot)(U_i, U_j)$ with $X_{i,j}$ which are both driven by the same Brownian motion and having a starting value of $W^{(n)}_{0}(U_i,U_j)$ and $W_0(U_i, U_j)$, respectively. Our subsequent analysis will be on the event $E_k(n)$ and it is unimportant how the coupling is done on $E^c_k(n)$.
    
    Define, $\widetilde{X}_{n,i,j}(t) \coloneqq K(X_n(t))(U_i, U_j),\; (i,j)\in [k]^{2}$. The evolution of $\widetilde{X}_{n,1,2}$, for example, can be described by the SDE
    \begin{align*}
        \begin{split}
            \diff \widetilde{X}_{n,1,2}(t) &= b\round{\widetilde{X}_{n,1,2}(t),K(X_n(t))}(U_1,U_2)\diff t + \Sigma(K(X_n(t)))(U_1,U_2) \diff B_{1,2}(t)\\
            &\qquad\qquad + \diff L^-_{n,1,2}(t) - \diff L^+_{n,1,2}(t), 
        \end{split}
    \end{align*}
    with the initial condition $\widetilde{X}_{n,1,2}(0)=W^{(n)}_{0}(U_1, U_2)$. Since $X_{1,2}$ is also driven by the same Brownian motion, by using the Lipschitz property of the Skorokhod map and the triangle inequality, it follows that for any $(U_1,U_2)=(u_1,u_2)$ on the event $E_k(n)$, $\sup_{s\in\interval{0,t}}  \abs{\widetilde{X}_{n,1,2}(s) - X_{1,2}(s)}^2$ is at most
    \begin{align}
        \begin{split}
            &48\int_0^t \abs{b\round{X_{1,2}(s),\Gamma(s)}(u_1, u_2)-b\round{\widetilde{X}_{n,1,2}(s),K(X_n(s))}(u_1, u_2)}^2\diff s\\
            &\qquad +48\sup_{s\in\interval{0,t}}\abs{\int_0^s\round{\Sigma(\Gamma(r))(u_1,u_2) - \Sigma(K(X_n(r)))(u_1,u_2)} \diff B_{1,2}(r)}^2\\
            &\qquad\qquad +48\abs{\widetilde{X}_{n,1,2}(0) - X_{1,2}(0)}^2.
        \end{split}\label{eq:triangle_b_sigma_init}
    \end{align}
    We can now use Assumption~\ref{asmp:b_lip_l2} and~\ref{asmp:b_lip_cut} on the first term in~\eqref{eq:triangle_b_sigma_init} to get 
    \begin{align}
        \begin{split}
            &\; \abs{b\round{X_{1,2}(s),\Gamma(s)}(u_1, u_2)-b\round{\widetilde{X}_{n,1,2}(s),K(X_n(t))}(u_1, u_2)}^2\\
            \leq&\; 2L^2\abs{X_{1,2}(s) - \widetilde{X}_{n,1,2}(s)}^2 + 2\kappa_{\cut}^2 \cutnorm{\Gamma(s) - K(X_n(s))}^2, \qquad s\in\R_+.
        \end{split}\label{eq:using_b_lip}
    \end{align}
    Define for $s\in[0,t]$,
    \[
        M^{(n)}(s) \coloneqq \int_0^s\round{\Sigma(\Gamma(r))(u_1,u_2) - \Sigma(K(X_n(r)))(u_1,u_2)} \diff B_{1,2}(r),
    \]
    which makes the second term in~\eqref{eq:triangle_b_sigma_init} equal to $48 \sup_{s\in[0,t]}M^2(s)$. Using Markov's inequality followed by Doob's maximal inequality~\cite[page 14, Theorem 3.8.iv]{KS91}, we obtain
    \begin{align}
        \Prob{ \sup_{s\in[0,t]}M^{(n)}(s)^2 \geq 2\lambda_k\E{M^{(n)}(t)^2} } &\leq \round{2\lambda_k\E{M^{(n)}(t)^2}}^{-1}\E{\sup_{s\in[0,t]}M^{(n)}(s)^2}\nonumber\\
        &\leq \round{2\lambda_k\E{M^{(n)}(t)^2}}^{-1}\E{M^{(n)}(t)^2} = 2\lambda_k^{-1},
    \end{align}
    for every $\lambda_k > 0$. Let $\round{\lambda_k}_{k\in\Natural}$ satisfy $\lim_{k\to\infty}\lambda_k = \infty$. The choice of $\lambda_k$ will be made later.

    Therefore, with probability at least $1-2\lambda_k^{-1}$,
    \begin{align}
        \begin{split}
            \sup_{s\in[0,t]}M^{(n)}(s)^2 &\leq 2\lambda_k \E{M^{(n)}(t)^2}\\
            &= 2\lambda_k \int_0^t \abs{\Sigma(\Gamma(s))(u_1,u_2) - \Sigma(K(X_n(s)))(u_1,u_2)}^2 \diff s\\
            &\leq 2\lambda_k\kappa_\cut^2 \int_0^t \cutnorm{\Gamma(s) - K(X_n(s))}^2 \diff s.
        \end{split}\label{eq:markov_doob_sigma}
    \end{align}
    By the abuse of notation, we redefine the event $E_k(n)$ to intersect with the event where the above bound holds. By a union bound, we still have $\lim_{k\to\infty}\lim_{n\to\infty}\Prob{E_k(n)} = 1$.
    
    Using equations~\eqref{eq:using_b_lip} and~\eqref{eq:markov_doob_sigma} in equation~\eqref{eq:triangle_b_sigma_init} we get
    % Combining with the previous bound, 
    \begin{equation}\label{eq:one_more_gronwall}
        \begin{split}
            \sup_{s\in\interval{0,t}}& \abs{\widetilde{X}_{n,1,2}(s) - X_{1,2}(s)}^2 \le 48\abs{W^{(n)}_{0}(U_1, U_2) - W_0(U_1, U_2)}^2 \\
            &+ 96\kappa_{\cut}^2(\lambda_k + 1) \int_0^t \cutnorm{\Gamma(s) - K(X_n(s))}^2\diff s\\
            &\qquad + 96L^2 \int_0^t \abs{X_{1,2}(s) - \widetilde{X}_{n,1,2}(s)}^2\diff s.
        \end{split}
    \end{equation}
    Replacing the role of $(1,2)$ by any other $(i,j)\in [k]^{(2)}$, and summing over, we get 
    \begin{equation}
        \begin{split}
            \sup_{s\in\interval{0,t}} &\frac{1}{k^2}\sum_{(i,j)\in [k]^{(2)}}\abs{\widetilde{X}_{n,i,j}(s) - X_{i,j}(s)}^2\\
            &\le \frac{48}{k^2}\sum_{(i,j)\in [k]^{(2)}}\abs{W^{(n)}_{0}(U_i, U_j) - W_0(U_i, U_j)}^2\\
            &\qquad+ 96\kappa_{\cut}^2(\lambda_k + 1) \int_0^t \cutnorm{\Gamma(s) - K(X_n(s))}^2\diff s\\
            &\qquad\qquad+ 96 L^2\int_0^t\frac{1}{k^2}\sum_{(i,j)\in [k]^{(2)}} \abs{X_{i,j}(s) - \widetilde{X}_{n,i,j}(s)}^2\diff s.
        \end{split}\label{eq:triangle_lipschitz}
    \end{equation}
    By the triangle inequality,
    \begin{align}
        \begin{split}
            \sup_{s\in \interval{0,t}}&\cutnorm{K\round{\round{\widetilde{X}_{n,i,j}(s)}_{(i,j)\in [k]^{(2)}}} - K\round{\round{\Gamma(s)(U_i, U_j)}_{(i,j)\in [k]^{(2)}}} }^2\\
            &\leq 2\sup_{s\in \interval{0,t}}\cutnorm{K\round{\round{\widetilde{X}_{n,i,j}(s)}_{(i,j)\in [k]^{(2)}}} - K\round{\round{X_{i,j}(s)}_{(i,j)\in [k]^{(2)}}}}^2\\
            &\qquad+ 2\sup_{s\in \interval{0,t}}\cutnorm{ K\round{\round{\Gamma(s)(U_i, U_j)}_{(i,j)\in [k]^{(2)}}} - K\round{\round{X_{i,j}(s)}_{(i,j)\in [k]^{(2)}}}}^2.
        \end{split}\label{eq:triagle_cut_sq}
    \end{align}
    Then notice that the kernel 
    \[
        \frac{1}{2}K\round{\round{\widetilde{X}_{n,i,j}(s)}_{(i,j)\in [k]^{(2)}}} - \frac{1}{2}K\round{\round{\Gamma(s)(U_i, U_j)}_{(i,j)\in [k]^{(2)}}}
    \]
    has entries in $[-1,1]$ and is sampled from the kernel $\frac{1}{2}K(X_n(s)) - \frac{1}{2}\Gamma(s)$. By~\cite[Lemma 10.6]{lovasz2012large}, the difference 
    \[
        \cutnorm{K\round{\round{\widetilde{X}_{n,i,j}(s)}_{(i,j)\in [k]^{(2)}}} - K\round{\round{\Gamma(s)(U_i, U_j)}_{(i,j)\in [k]^{(2)}}}}^2 - \cutnorm{ K(X_n(s)) - \Gamma(s) }^2
    \]
    lies in the interval $\interval{-24/k -36/k^2, 64k^{-1/4}+256k^{-1/2}}$ with probability at least $1-4\eu^{-k^{1/2}/10}$, for all $n\ge k$.
    Using this in~\eqref{eq:triagle_cut_sq} we get
    \begin{align}
        \begin{split}
            \sup_{s\in \interval{0,t}}&\cutnorm{K\round{\round{\widetilde{X}_{n,i,j}(s)}_{(i,j)\in [k]^{(2)}}} - K\round{\round{X_{i,j}(s)}_{(i,j)\in [k]^{(2)}}}}^2\\
            &\geq \inv{2}\cutnorm{ K(X_n(s)) - \Gamma(s) }^2 - 320k^{-1/4}\\
            &\qquad - \sup_{s\in \interval{0,t}}\cutnorm{ K\round{\round{\Gamma(s)(U_i, U_j)}_{(i,j)\in [k]^{(2)}}} - K\round{\round{X_{i,j}(s)}_{(i,j)\in [k]^{(2)}}}}^2.
        \end{split}\label{eq:after_using_lovasz}
    \end{align}
    % 999
    % \begin{align}
    %     \begin{split}
    %         \sup_{s\in \interval{0,t}}&\cutnorm{K\round{\round{\widetilde{X}_{n,i,j}(s)}_{(i,j)\in [k]^{(2)}}} - K\round{\round{X_{i,j}(s)}_{(i,j)\in [k]^{(2)}}}}^2\\
    %         &\leq 2\sup_{s\in \interval{0,t}}\cutnorm{ K\round{\round{\Gamma(s)(U_i, U_j)}_{(i,j)\in [k]^{(2)}}} - K\round{\round{X_{i,j}(s)}_{(i,j)\in [k]^{(2)}}}}^2\\
    %         &\qquad + 2\sup_{s\in \interval{0,t}}\cutnorm{ K(X_n(s)) - \Gamma(s) }^2 + 640 k^{-1/4},
    %     \end{split}\label{eq:after_using_lovasz}
    % \end{align}
    with probability at least $1-4\eu^{-k^{1/2}/10}$. By an abuse of notation, we redefine the event $E_{k}(n)$ to intersect with the event where the above bound holds. We still have $\lim_{k\to\infty}\lim_{n\rightarrow\infty} \Prob{E_k(n)}=1$.
    % Continuing, we multiply equation~\eqref{eq:triangle_lipschitz} by $2$ and use the fact that the $L^2$ norm is lower bounded by the cut norm. Further lower bounding using equation~\eqref{eq:after_using_lovasz}, and rearranging terms we get

    % \begin{align}
    %     \begin{split}
    %         \sup_{s\in \interval{0,t}}&\enorm{K\round{\round{\widetilde{X}_{n,i,j}(s)}_{(i,j)\in [k]^{(2)}}} - K(\round{X_{i,j}(s))_{(i,j)\in [k]^{(2)}}}}^2\\
    %         & + \sup_{s\in \interval{0,t}}\cutnorm{K\round{\round{\widetilde{X}_{n,i,j}(s)}_{(i,j)\in [k]^{(2)}}} - K\round{\round{X_{i,j}(s)}_{(i,j)\in [k]^{(2)}}}}^2\\
    %         &\leq \frac{48}{k^2}\sum_{(i,j)\in [k]^{(2)}}\abs{W^{(n)}_{0}(U_i, U_j) - W_0(U_i, U_j)}^2 + 96\kappa_{\cut}^2(\lambda_k + 1) \int_0^t \cutnorm{\Gamma(s) - K(X_n(s))}^2\diff s\\
    %         &\qquad\qquad+ 96 L^2\int_0^t \enorm{K\round{\round{\widetilde{X}_{n,i,j}(s)}_{(i,j)\in [k]^{(2)}}} - K(\round{X_{i,j}(s))_{(i,j)\in [k]^{(2)}}}}^2\diff s\\
    %         &\qquad\qquad\qquad + 2\sup_{s\in \interval{0,t}}\cutnorm{ K\round{\round{\Gamma(s)(U_i, U_j)}_{(i,j)\in [k]^{(2)}}} - K\round{\round{X_{i,j}(s)}_{(i,j)\in [k]^{(2)}}}}^2\\
    %         &\qquad\qquad\qquad\qquad + 2\cutnorm{ K(X_n(s)) - \Gamma(s) }^2 + 640 k^{1/4}
    %     \end{split}\label{eq:before_gronwall}
    % \end{align}
    We first lower bound twice the left hand side of equation~\eqref{eq:triangle_lipschitz} using equation~\eqref{eq:after_using_lovasz} as
    \begin{align}
        \begin{split}
            2\sup_{s\in \interval{0,t}}&\enorm{K\round{\round{\widetilde{X}_{n,i,j}(s)}_{(i,j)\in [k]^{(2)}}} - K\round{\round{X_{i,j}(s)}_{(i,j)\in [k]^{(2)}}}}^2\\
            &\geq \sup_{s\in \interval{0,t}}\enorm{K\round{\round{\widetilde{X}_{n,i,j}(s)}_{(i,j)\in [k]^{(2)}}} - K\round{\round{X_{i,j}(s)}_{(i,j)\in [k]^{(2)}}}}^2\\
            &\qquad + \sup_{s\in \interval{0,t}}\cutnorm{K\round{\round{\widetilde{X}_{n,i,j}(s)}_{(i,j)\in [k]^{(2)}}} - K\round{\round{X_{i,j}(s)}_{(i,j)\in [k]^{(2)}}}}^2\\
            &\geq \sup_{s\in \interval{0,t}}\enorm{K\round{\round{\widetilde{X}_{n,i,j}(s)}_{(i,j)\in [k]^{(2)}}} - K\round{\round{X_{i,j}(s)}_{(i,j)\in [k]^{(2)}}}}^2\\
            &\qquad + \inv{2}\cutnorm{ K(X_n(s)) - \Gamma(s) }^2 - 320k^{-1/4}\\
            &\qquad - \sup_{s\in \interval{0,t}}\cutnorm{ K\round{\round{\Gamma(s)(U_i, U_j)}_{(i,j)\in [k]^{(2)}}} - K\round{\round{X_{i,j}(s)}_{(i,j)\in [k]^{(2)}}}}^2.
        \end{split}\label{eq:L2_cut_lovasz}
    \end{align}

    Here we used the fact that the $L^2$ norm is lower bounded by the cut norm. Using equation~\eqref{eq:L2_cut_lovasz} back in equation~\eqref{eq:triangle_lipschitz} (multiplied by $2$), and rearranging terms we get
    
    \begin{align}
        \begin{split}
            \sup_{s\in \interval{0,t}}&\enorm{K\round{\round{\widetilde{X}_{n,i,j}(s)}_{(i,j)\in [k]^{(2)}}} - K\round{\round{X_{i,j}(s)}_{(i,j)\in [k]^{(2)}}}}^2\\
            &\qquad + \inv{2}\sup_{s\in \interval{0,t}}\cutnorm{ K(X_n(s)) - \Gamma(s) }^2\\
            &\leq \sup_{s\in \interval{0,t}}\cutnorm{ K\round{\round{\Gamma(s)(U_i, U_j)}_{(i,j)\in [k]^{(2)}}} - K\round{\round{X_{i,j}(s)}_{(i,j)\in [k]^{(2)}}}}^2\\
            &\qquad + 320k^{-1/4} + \frac{96}{k^2}\sum_{(i,j)\in [k]^{(2)}}\abs{W^{(n)}_{0}(U_i, U_j) - W_0(U_i, U_j)}^2\\
            &\qquad + 192 L^2\int_0^t \enorm{K\round{\round{\widetilde{X}_{n,i,j}(s)}_{(i,j)\in [k]^{(2)}}} - K\round{\round{X_{i,j}(s)}_{(i,j)\in [k]^{(2)}}}}^2\diff s\\
            &\qquad + 192\kappa_{\cut}^2(\lambda_k + 1) \int_0^t \cutnorm{\Gamma(s) - K(X_n(s))}^2\diff s.
        \end{split}\label{eq:before_gronwall}
    \end{align}
    Now let 
    \[
        \begin{split}
            A_k  &\coloneqq  \sup_{s\in\interval{0,t}}\cutnorm{K\round{\round{\Gamma(s)(U_i, U_j)}_{(i,j)\in [k]^{(2)}}} - K\round{\round{X_{i,j}(s)}_{(i,j)\in [k]^{(2)}}}}^2,\\
            B_k(n) &\coloneqq \frac{96}{k^2}\sum_{(i,j)\in [k]^{(2)}}\abs{W^{(n)}_{0}(U_i, U_j) - W_0(U_i, U_j)}^2+ 320k^{-1/4}.
        \end{split}
    \]

    Applying Gr\"onwall's inequality~\cite{gronwall1919note} and noticing that the first term on the left of equation~\eqref{eq:before_gronwall} is always non-negative, gives us that on the event $E_k(n)$,
    \begin{align}
    \label{eq:After_Gronwall}
        \begin{split}
            \sup_{s\in \interval{0,t}}&\enorm{K\round{\round{\widetilde{X}_{n,i,j}(s)}_{(i,j)\in [k]^{(2)}}} - K\round{\round{X_{i,j}(s)}_{(i,j)\in [k]^{(2)}}}}^2\\
            &+ \sup_{s\in\interval{0,t}}\cutnorm{K(X_n(s)) - \Gamma(s)}^2 \le 2\left( A_k + B_k(n)  \right) \exp\round{192(L^2+2\kappa_\cut^2(\lambda_k + 1))t},
        \end{split}
    \end{align}
    for every $n\geq k$. Note that
    \[
        \E{\abs{W^{(n)}_{0}(U_i, U_j) - W_0(U_i, U_j)}^2} = \enorm{W_0^{(n)}-W_0}^2 \rightarrow 0,
    \]
    as $n\rightarrow \infty$, by assumption~\eqref{eq:initassump}. By a variance bound it follows that 
    \[
        \lim_{k\rightarrow \infty} \lim_{n\rightarrow \infty} B_k(n)=0,
    \]
    in probability. Also, $\lim_{k\rightarrow \infty} A_k=0$ by Proposition~\ref{prop:mck_vlasov}. Since $\lim_{k\to\infty}\lim_{n\rightarrow \infty}\Prob{E_k(n)}=1$,
    \begin{align*}
        \begin{split}
            \lim_{n\rightarrow \infty}\sup_{s\in\interval{0,t}}&\cutnorm{K(X_n(s)) - \Gamma(s)} =0,\qquad\text{and}\\
            \lim_{k\to\infty}\lim_{n \to \infty}\sup_{s\in \interval{0,t}}&\inv{k^2}\normF{\round{K\round{X_n(s)}(U_i,U_j)}_{(i,j)\in [k]^{(2)}} - \round{X_{i,j}(s)}_{(i,j)\in [k]^{(2)}}}^2 = 0,
        \end{split}
    \end{align*}
    in probability, by choosing $\round{\lambda_k}_{k\in\Natural}$ (depending on $\round{A_k,\lim_{n\to\infty} B_k(n)}_{k\in\Natural}$) that increases sufficiently slowly to infinity as $k\to\infty$. This proves our claim.
\end{proof}

\begin{remark}\label{rem:Perturbation}
    Note that the proof is robust with respect to small perturbations of drift. More precisely, consider two processes $X_n$ and $\widetilde{X_n}$ satisfying~\eqref{eq:sden2} with drift functions $R_n$ and $\widetilde{R}_n$ respectively such that $\norm{2}{n^2R_n(A)-n^2\widetilde{R}_n(A)}\to 0$ as $n\to \infty$. Then, $K(X_n)$ and $K(\widetilde{X}_n)$ converge to the same limiting McKean-Vlasov SDE.
\end{remark}

\begin{remark}\label{rem:Rate}
To get a non-asymptotic error rate, we need to control on $A_k$ and $B_k(n)$. Observe that $B_k(n)$ depends on the initial condition and in general it can be arbitrarily slow. However, assuming that the initial condition is i.i.d., one can use Chebyshev's inequality to obtain $\Prob{B_k(n)\geq 66k^{-1/4}}\leq k^{-3/2}$. On the other hand, it follows from the arguments in Proposition~\ref{prop:mck_vlasov} that there exists a constant $M_{t}$ (depending only on $t$) such that for any $\delta>0$ we have $\Prob{A_k\geq M_t(\delta \log(1/\delta))^{1/4}}\leq k^{-2}+ t\delta^{-1}\eu^{\frac{128}{\delta\log(1/\delta)}} \eu^{-k\delta \log(1/\delta)/2}$. 
    
In particular, choosing $\delta=64\sqrt{k^{-1}\log k}$ and $\lambda_k=\log(k)/\round{16\cdot 384t(L^2+2\kappa_\cut^2)}$, we have the left hand side of~\eqref{eq:After_Gronwall} bounded by $M_tk^{-1/16}\log^{3/2}k$ with probability at least $1-\frac{k^2}{n}-4k^{-\inv{\kappa^2t}} - 2t\eu^{-\sqrt{k}/20}- 2k^{-3/2}$, where $\kappa=32\sqrt{6}\round{L^2+2\kappa_\cut^2}^{1/2}$. Since $t$ is fixed, we can choose $k$ to be a suitable function of $n$, say $k=n^{2/7}$, to get a non-asymptotic rate of convergence. Moreover, using the remark after the proof of Lemma~\ref{lem:FixedDimensionContinumLimit}, we can get a non-asymptotic rate of convergence with finite $n$ and $\abs{\tauvec_n}$.
\end{remark}

\section{Examples}\label{sec:example}

In this section, we will verify our assumptions for a class of functions introduced as linear functions in~\cite[Section 5.1]{oh2021gradient}. Let  $\set{Z_i}_{i\in\squarebrack{n}}$ be i.i.d. $\mathrm{Uni}\interval{0,1}$. For any kernel $W\in\Wcal$ and any $n\in\Natural$, sample a random matrix $G_n[W]$ as $G_n[W] \coloneqq \round{W(Z_i,Z_j)}_{(i,j)\in\squarebrack{n}^{(2)}} \in\Mcal_n$. Let $\rho_n([W])$ denote its law, i.e., $\mathrm{Law}(G_n[W]) = \rho_n([W])$. Now let $R\colon\Wcal \to \R$ be defined as a linear function, i.e.,
\[
    R(W) \coloneqq \int_{\Mcal_n} R_n(z)\rho_n([W])(\diff z), \qquad \forall\ W\in\Wcal,
\]
%such that $R_n\in L^1(\rho_n)$ and $\nabla R_n \in L^\infty(\rho_n)$. 
Let $(\Omega,\Acal)$ be the standard measurable space on $\interval{0,1}^n$. 
Let $\ell\colon \Wcal \times \Omega$ be the function defined as
\[
    \ell(W,Z) \coloneqq R_n\round{\round{W(Z_i,Z_j)}_{(i,j)\in\squarebrack{n}^{(2)}}}.
\]

Let $R_n$ satisfy Assumption~\ref{asmp:R_ell_phi}(\ref{item:Rn_C1}) and let $R$ admit a Fr\'echet-like derivative evaluation map $\phi\colon\Wcal \to L^\infty\big(\interval{0,1}^{(2)}\big)$ (see~\cite[Section 5]{oh2021gradient} for conditions). The map $\phi$ then satisfies
\begin{align}
    \phi(W)(x,y) = \sum\nolimits_{(i,j)\in\squarebrack{n}^2} \E{ \nabla R_n \round{\round{W(Z_p,Z_q)}_{(p,q)\in\squarebrack{n}^{(2)}}} \given (Z_i,Z_j) = (x,y)},
\end{align}
and $D_\Wcal \ell(\slot{};Z)$ for $Z\in\interval{0,1}^n$ satisfies
\begin{align}
    (D_\Wcal \ell(\slot{};Z))(W)(x,y) = \sum\nolimits_{(i,j)\in\squarebrack{n}^2} \nabla R_n \round{\round{W(Z_p,Z_q)}_{(p,q)\in\squarebrack{n}^{(2)}} \big\vert_{(Z_i,Z_j) = (x,y)}},
\end{align}
for $W\in\Wcal$ and $(x,y)\in\interval{0,1}^{(2)}$. 
%\Tripathi{No scaling?}\RaghavS{This follows from~\cite[Equation 97]{oh2021gradient}. Isn't the scaling already due to the sum $\sum_{(i,j)\in[n]^2}$?}
\subsection{Scalar Entropy and Homomorphism density}
\label{sec:scalar entropy and homomorphism example}
Examples like the scalar entropy and the homomorphism density functions considered in~\cite[Section 5.1-5.2]{oh2021gradient}, all satisfy Assumption~\ref{asmp:R_ell_phi} for some $\kappa_2\in\R_+$ since $\opnorm{\mathrm{Hess}(R_n)}$ exists and is bounded uniformly in the domain. Specifically, for homomorphism density function $R = H_F$ for a simple graph $F$ with $n$ vertices and $m$ edges $\set{e_l}_{l=1}^m$, the constants $\kappa_2 = mn(n-1)$, and for scalar entropy $R = \Ent$, the constant $\kappa_2 = 2\eps^{-1}(1-\eps)^{-1}$ on its domain $\Wcal_\eps \coloneqq \setinline{W\in\Wcal \given \eps\leq W \leq 1-\eps}$ where $\eps\in(0,1/2)$. Since this implies that there exists $M_\infty\in\R_+$ such that $\norm{\infty}{\phi(W)}\leq M_\infty$ for all $W$ in the domain, these example also satisfy Assumption~\ref{asmp:small_noise} for $\sigma = M_\infty$. 

In the following, we define $b\colon\interval{-1,1}\times \Wcal \to L^\infty\big(\interval{0,1}^{(2)}\big)$ as $b(W(x,y),W)(x,y) = -\phi(W)(x,y)$ for all $W\in\Wcal$ and a.e. $(x,y)\in\interval{0,1}^{(2)}$. We will now verify Assumption~\ref{asmp:b_lip_cut} when $R$ is the sum of scalar entropy and some homomorphism density $H_F$ for a simple graph $F$ with $n$ vertices and $m$ edges. Note that for this example, we have
\begin{align}
    b(z,W)(x,y) = \log\frac{z}{1-z} + \phi_{H_F}(W)(x,y),\qquad z = W(x,y)\in \interval{\eps,1-\eps},\label{eq:phi_scalar_entropy+hom}
\end{align}
for a.e. $(x,y)\in[0,1]^{(2)}$ where from~\cite[Equation 113]{oh2021gradient},
\[
    \begin{split}
        \phi_{H_F}(W)(x,y) &= \sum_{l=1}^m \E{\prod_{r=1,r\neq l}^m W(Z_{e_r}) \given Z_{e_l}=(x,y)}\\
        &\eqqcolon \sum_{l=1}^m \mathbf{t}_{x,y}(F_{e_l},W), \quad (x,y)\in\interval{0,1},
    \end{split}
\]
$Z_{e} = (Z_{e(1)},Z_{e(2)})$ and $F_{e_l}$ is the simple graph obtained from $F$ by removing the edge $e_l$. It is shown in~\cite[Section 5.1.2]{oh2021gradient} that the map $W \mapsto \mathbf{t}_{(\cdot,\cdot)}(F_{e},W)$ continuous as a map from $(\Wcal,d_\cut)$ to $\big(L^\infty\big([0,1]^{(2)}\big),d_\cut\big)$. To show that $\phi_{H_F}(\slot{})(x,y)$ is Lipschitz in the cut norm for every $(x,y)\in\interval{0,1}^{(2)}$, it is sufficient to show that $\mathbf{t}_{x,y}(F_e,\slot{})$ is Lipschitz in the cut norm for $e\in\set{e_l}_{l=1}^m$. For $W_1,W_2\in\Wcal$, note that
\begin{align*}
    \mathbf{t}_{x,y}(F_e,W_1) - \mathbf{t}_{x,y}(F_e,W_2) = \sum_{\set{p,q}\in E(F_e)} I_{p,q},
\end{align*}
where for any $\set{p,q}\in E(F_e)$,
\begin{align}
    I_{p,q} \coloneqq \int_{\interval{0,1}^{n-2}}\round{W_1(x_p,x_q)-W_2(x_p,x_q)}\prod_{(i,j)\in E(F_e)\setminus \set{p,q}}W_1(x_i,x_j) \prod_{v\in V(F_e)\setminus e}\diff x_v.
\end{align}
Following the proof in~\cite[Lemma 10.24]{lovasz2012large}, we get $\abs{I_{p,q}} \leq \cutnorm{W_1-W_2}$, which yields
\begin{align}
\label{eqn:t_calc}
    \abs{\mathbf{t}_{x,y}(F_e,W_1) - \mathbf{t}_{x,y}(F_e,W_2)} \leq (m-1)\cutnorm{W_1-W_2},
\end{align}
i.e., the Lipschitz constant of $\mathbf{t}_{x,y}(F_e,\slot{})$ for every $e\in E(F)$ is $m-1$. This implies that the Lipschitz constant of $\phi(\slot{})(x,y)$ with respect to $\cutnorm{}$ is $m(m-1)$. Therefore, for $b$ as in equation~\eqref{eq:phi_scalar_entropy+hom}, we have
\begin{align}
    \abs{b(z,W_1)(x,y) - b(z,W_2)(x,y)} &= \abs{\phi_{H_F}(W_1)(x,y) - \phi_{H_F}(W_1)(x,y)}\nonumber\\
    &\leq m(m-1)\cutnorm{W_1-W_2}.
\end{align}
Therefore $b$ (as in equation~\eqref{eq:phi_scalar_entropy+hom}) satisfies Assumption~\ref{asmp:b_lip_cut} with $\kappa_\cut = m(m-1)$.

\subsection{Quadratic functions of homomorphism density}  

More generally, let $k\in \N$ and let $\set{F^{1}, \ldots, F^{k}}$ be a family of finite simple graphs. Let $c_1, \ldots, c_k\in [0, 1]$ be fixed constants. Define a function $R\colon\Wcal\to \R$ as 
\[
    R(W) \coloneqq \inv{2}\sum_{\alpha=1}^{k} \round{H_{F^{\alpha}}(W)-c_{\alpha}}^2.
\]
Note that a lower bound on $R$ is achieved if $H_{F^{\alpha}}\equiv c_{\alpha}$ for all $\alpha\in [k]$. We note that $R$ being a sum of squares of $k$ many functions satisfies Assumption~\ref{asmp:R_ell_phi}(\ref{item:phi_lip}).
%If this lower bound is indeed achieved, minimizing $R$ is, therefore, equivalent to finding a graphon with $k$ many prescribed homomorphism densities.

Moreover, let $\phi\colon\Wcal \to L^\infty\big([0,1]^{(2)}\big)$ denote the Fr\'echet-like derivative evaluation map of $R$. It follows from chain-rule that 
\begin{align*}
    \phi(W)(x, y) &= \sum_{\alpha=1}^{k}(H_{F^{\alpha}}(W)-c_{\alpha})\phi_{H_{F^{\alpha}}(W)}(W)(x, y)\;.   % &= \sum_{\alpha=1}^{k}\sum_{e\in F^{\alpha}}(H_{F^{\alpha}}(W)-c_{\alpha})\mathbf{t}_{x,y}(F^{\alpha}_e,W), \qquad W\in\Wcal,\ (x,y)\in[0,1]^{(2)}.
\end{align*}

Note that $W\mapsto \phi_{H_{F^{\alpha}}}(W)$ satisfies Assumption~\ref{asmp:R_ell_phi}(\ref{item:phi_lip}) with $\kappa_{2, \alpha}=m_{\alpha}(m_{\alpha}-1)$ where $m_{\alpha}$ is the number of edges in $F^{\alpha}$. Further note that for any finite graph $F$ and $U, V\in \Wcal$ we have $\abs{H_{F}(U)-H_{F}(V)}\leq |E(F)|\norm{\cut}{U-V}\leq \abs{E(F)}\norm{2}{U-V}$. A simple calculation using the fact that $\abs{(H_{F^{\alpha}}(W)-c_{\alpha})}\leq 1$ for all $W$ and that $\norm{2}{\phi_{H_{F}}(W)}\leq \abs{E(F)}$, we obtain that $\phi$ satisfies Assumption~\ref{asmp:R_ell_phi}(\ref{item:phi_lip}) with
\[
    \kappa_2\leq \sum_{\alpha=1}^{k}(m_{\alpha}^2+\kappa_{2, \alpha})\leq km^2,
\]
where $m=\max_{\alpha\in [n]}m_{\alpha}$.
%\RaghavS{can you explain the $(1+m_{\alpha}+\abs{c_{\alpha}})$ term?}\Tripathi{It follows from the following calculation. For a fixed $\alpha$ we check that \begin{align*}
%    \abs{\phi(W_1)(x, y)-\phi(W_2)(x, y)}&\leq \sum_{e\in F^{\alpha}}\abs{(H_{F^{\alpha}}(W_1)-c_{\alpha})\mathbf{t}_{x,y}(F^{\alpha}_e,W_1)-(H_{F^{\alpha}}(W_2)-c_{\alpha})\mathbf{t}_{x,y}(F^{\alpha}_e,W_2)}\\
%    &\leq \sum_{e\in F^{\alpha}}\abs{(H_{F^{\alpha}}(W_1)-c_{\alpha})}\abs{\mathbf{t}_{x,y}(F^{\alpha}_e,W_1)-\mathbf{t}_{x,y}(F^{\alpha}_e,W_2)}\\
%    &+ \sum_{e\in F^{\alpha}}\abs{(H_{F^{\alpha}}(W_1)-c_{\alpha})-H_{F^{\alpha}}(W_2)-c_{\alpha})}\mathbf{t}_{x,y}(F^{\alpha}_e,W_2)\\
%    &\leq \sum_{e\in F^{\alpha}}\abs{\mathbf{t}_{x,y}(F^{\alpha}_e,W_1)-\mathbf{t}_{x,y}(F^{\alpha}_e,W_2)}+ \abs{H_{F^{\alpha}}(W_1)-H_{F^{\alpha}}(W_2)},
%\end{align*}
%where we use the fact that $\abs{(H_{F^{\alpha}}(W_1)-c_{\alpha})}\leq 1$. Now for the second term can be bounded by $m_{\alpha}$ times the cut-norm of $W_1-W_2$ which in-turn can be bounded by the $L^2$. On the other hand, the first term, we already know is bounded by $k_{2, \alpha}$. I originally bounded the difference $\abs{(H_{F^{\alpha}}(W_1)-c_{\alpha})}\leq 1+c_{\alpha}$ which is not required. So, we can improve a little bit.}
%and $n=\max_{\alpha\in [n]}n_{\alpha}$. 

Similarly, for any edge $e$ in a finite simple graph $F$, note $W\mapsto \mathbf{t}_{x, y}(F_{e}, W)$ is $(m-1)$-Lipschitz in cut norm for every $(x, y)\in[0,1]^{(2)}$ and $W\mapsto H_{F}(W)$ is $m$-Lipschitz in cut norm where $m$ is the number of edges in $F$. Using the fact that $\norm{\infty}{\phi_{H_{F}}(W)}\leq m$ and $H_{F}(W)\in [0,1]$ for every $W\in\Wcal_0$, we conclude that $\phi(\slot{})(x, y)$ is $km^2$-Lipschitz with respect to $\norm{\cut}{}$ for a.e. $(x,y)\in[0,1]^{(2)}$ and hence $\phi$ satisfies Assumption~\ref{asmp:b_lip_cut}.

\subsection{Entropy minimization with edge-triangle constraints}
\label{subsec:EdgeTriangle}
We conclude with the discussion of the example mentioned in the Introduction. Recall the problem of minimizing the scalar entropy $\Ent$ over $\Graphons_0$ with prescribed edge density $H_{\mathrel{-}}(\slot{})=e \in[0,1]$ and triangle density $H_{\triangle}(\slot{})=\tau\in [0,1]$ (see~\cite[Section 5.1-5.2]{Oh2023}). As mentioned in~\cite{neeman2023typical}, in general this problem does not admit unique minimizer. 

Let us consider a relaxation of this problem. Let $\psi\colon\R\to \R$ be a non-decreasing convex function such that $\psi'(-\log(2))\eqqcolon A>1$. Consider minimizing the function
\[
    W\mapsto R(W)\coloneqq \inv{2}\left((H_{\mathrel{-}}(W)-e)^2+(H_{\triangle}(W)-\tau)^2\right) + \psi(\Ecal(W)).
\]
Since $\psi$ is non-decreasing, minimizing $\Ecal$ is equivalent to minimizing $\psi\circ \Ecal$. On the other hand, the term $\frac{1}{2}\left((H_{\mathrel{-}}(W)-e)^2+(H_{\triangle}(W)-\tau)^2\right)$ penalizes any deviation from the marginal constraint on the edge and triangle densities.

It follows from the previous discussion that  $W\mapsto \inv{2}(H_{\mathrel{-}}(W)-e)^2+\inv{2}(H_{\triangle}(W)-\tau)^2$ is $\lambda$-semiconvex with $\lambda=-8$. On the other hand, $\Ecal$ is $4$-semiconvex and therefore $\psi \circ \Ecal$ is $4A$-semiconvex. In particular, if $A>2$ then $R$ is strongly convex and hence admits a unique minimizer and the gradient flow converges exponentially fast to the minimizer of $R$. In this case, the gradient flow of $R$ converges exponentially fast to the minimizer.

For instance, take $\psi = 4\id$ and consider the optimization algorithm described in Definition~\ref{def:PNSGD}. For every $n\in\Natural$, $X_n\in\Mcal_n$, and $(i,j)\in[n]^{(2)}$, we can evaluate $g_{n,(i,j)}(X_n;\xi)$ as
\begin{align*}
    g_{n,(i,j)}(X_n;\xi) &\coloneqq 4\log\round{\frac{X_n(i,j)}{1-X_n(i,j)}} + \round{X_n(i_1,i_2) - e}\\
    &\qquad + \round{X_n(i_3,i_4)X_n(i_4,i_5)X_n(i_5,i_3) - \tau}X_n(i,i_6)X_n(i_6,j),
\end{align*}
where $\xi = (i_z)_{z\in[6]} \overset{\rm i.i.d.}{\sim} \mathrm{Uni}\round{[n]}^6$. Notice that $\Exp{\xi}{g_n(X_n;\xi)} = \nabla R_n(X_n)$, and Assumption~\ref{asmp:small_noise} is satisfied. Theorem~\ref{thm:SGD_to_SDEn} and Theorem~\ref{thm:zeroTempLimit} tell us that the~\eqref{eq:PNSGD} algorithm in the absence of large noise, converges to the minimizer of $R$ as the step size of the algorithm goes to zero, and $n\to\infty$.

If one takes $\psi=\id$ then the function $R$ is not guaranteed to be convex. Therefore, there may be multiple minimizers of $R$ as mentioned in~\cite{neeman2023typical}. Since $R$ is not strictly convex, the gradient flow may not converge to the minimizer, however, it does converge to a stationary point with a polynomial rate. 
%A similar method can be used to minimize a convex function $R$ over the space of graphons with finitely many homomorphism densities prescribed. 

\subsection{A linear regression problem}\label{sec: Linear regression}
Let $(X, Y)\in \mathbb{R}^n\times \mathbb{R}^n$ be a random vector. Consider the function $R_n$ on $\Mcal_n^0$, the set of symmetric $n\times n$ matrices with entries in $[0, 1]$ defined as
\begin{equation} \label{eq:gls}
    R_n(A)\coloneqq\frac{1}{n}\mathbb{E}\norm{2}{Y-n^{-1}AX}^2.
\end{equation}
The function $R_n$ in~\eqref{eq:gls} is permutation invariant if the joint distribution of $(X, Y)$ is exchangeable (i.e. for any permutation $\tau$, the distribution of $(X^\tau, Y^{\tau})$ is the same as that of $(X, Y)$, where $(X^\tau_i, Y^\tau_i)=(X_{\tau(i)}, Y_{\tau(i)})$. The function $R_n$ is also differentiable in the Euclidean sense. Let $X_n$ be $\Mcal_n^0$ valued process satisfying the SDE~\eqref{eq:exampleR} with drift function $R_n$. We now describe the McKean-Vlasov limit of $K(X_n)$ as $n\to \infty$. %Naturally, this SDE can be seen as a stochastic gradient algorithm to find the minimizer of $R_n$ for small parameter $\beta>0$. 

To this end, we first expand $R_n$ in~\eqref{eq:gls} and compute the $\nabla R_n$. Let $C_n, C_n'$ be $n\times n$ matrices such that $\Sigma(i, j) = \E{X_iX_j}, \quad \Sigma'(i, j) = \E{Y_iX_j}$. It follows from the exchangeability of $(X, Y)$ that 
\begin{align*}
    C_n(i, j) &= a\delta_{i\neq j}+b\delta_{i=j}, \qquad
    C_n'(i, j) = c\delta_{i\neq j}+ d\delta_{i=j}\;,
\end{align*}
where $a= \mathbb{E}(X_1X_2), b= \E{X_1^2}, c= \mathbb{E}(Y_1X_2), d=\E{X_1Y_1}$. With this notation, we can rewrite  $R_n$ as 
\begin{align*}
    R_n(A) &=  \E{Y_1^2}+ H_n(A)+ E_n(A)\;,
\end{align*}
where 
\begin{align*}
    H_n (A)&= \frac{a}{n^3}\sum_{i, j, k=1}^{n}A(i, j)A(i, k)-\frac{2c}{n^2}\sum_{i, j=1}^{n}A(i, j) =a\hom(P_3, A)-2c\hom(P_2, A),\\
    E_n(A) &= \frac{(b-a)}{n^3}\normF{A}^2 - \frac{2(d-c)}{n}\hom(P_2, A)\;.
\end{align*}
In particular, $\nabla R_n(A) = \nabla H_n(A) + \nabla E_n(A)$. Since the entries of $A$ are bounded, we also have
\begin{align*}
    |\nabla E_n(A)(i, j)| &\leq \frac{C}{n^3}\delta_{i\neq j}+ \frac{C}{n^2}\delta_{i=j}\,
\end{align*}
for some constant $C>0$. Therefore, 
\[
    \norm{2}{K\round{n^2\nabla H_n(A)-n^2\nabla R_n(A)}} = \norm{2}{K\round{n^2\nabla E_n(A)}}\leq \frac{C}{n}\to 0,
\] 
as $n\to \infty$. By Remark~\ref{rem:Perturbation}, the McKean-Vlasov limit of $\round{X_n}_{n\in\Natural}$ is the same as the McKean-Vlasov limit of the process $\round{Y_n}_{n\in\Natural}$ satisfying~\eqref{eq:exampleR} with drift function $\nabla H_n$ for all $n\in\Natural$. Since $H_n$ is a linear combination of homomorphism density functions and can be seen as the restriction of the function
$\mathcal{H}\colon \mathcal{W}\to \mathbb{R}$ given by 
\[
    \mathcal{H}(W) = \sigma_{Y}^2 + a \int W(x, y)W(x, z)\diff x\diff y\diff z-2c \int W(x, y)\diff x\diff y,
\]
it follows from our discussion in Section~\ref{sec:scalar entropy and homomorphism example} that $\round{Y_n}_{n\in\Natural}$ converges to a McKean-Vlasov limit~\eqref{eq:infinite_SDE0} and~\eqref{eq:whatisgammat0} with the drift $\phi$ defined as  
\[
    \phi(W)(x, y) \coloneqq -D\mathcal{H}(W)(x, y) = a\int W(x, z)\diff z -2c, \qquad (x,y) \in [0,1]^2
\]
In particular, any local minimizer must satisfy the condition $a\int W(x, z)\diff z=2c$.
The same method can be extended in an obvious manner to the squared norm in~\eqref{eq:gls} is replaced by any even positive power.

%\clearpage
% \input{sections/discussion}%\clearpage

%\input{sections/example}
%\clearpage
% \input{sections_arxiv/examples_CE}\clearpage
% \input{sections_arxiv/analogy_OT}\clearpage
% \input{sections_arxiv/discussion}\clearpage

%%%%%%%%%%%% Bib as in amsart template
% \begin{thebibliography}{100}
% \bibitem{Test} J. Smith, \emph{Test Title}, Test Publisher, New York, 19--. 
% \end{thebibliography}
%%%%%%%%%%%%

%%%%%%%%%%%%
% \nocite{*}
% \printbibliography
\bibliographystyle{alpha} % can use amsalpha but repeated authors are dashed
\bibliography{references}
%%%%%%%%%%%%

%\clearpage
% \appendix
% \input{sections_arxiv/appendix}
% \input{appendix/definitions}\clearpage
% \input{appendix/technical_lem}\clearpage

\end{document}